\documentclass[11pt]{article}
\usepackage{amsthm}
\usepackage{amscd}
\usepackage{mathpazo}
\usepackage{txfonts}

\usepackage{amsbsy,amsmath}
\usepackage{mathrsfs}
\usepackage{enumerate}
\usepackage{enumitem}

\usepackage{mathtools}
\allowdisplaybreaks[1]
\usepackage{color}
\usepackage{hyperref}
\usepackage{lastpage}
\usepackage{fancyhdr}
\usepackage[T1]{fontenc}
\usepackage[margin=3cm]{geometry}
\usepackage{graphicx}
\usepackage[dvipsnames]{xcolor}
\usepackage{centernot}
\allowdisplaybreaks[1]
\usepackage{tikz-cd}
\usepackage[normalem]{ulem}

\DeclareMathOperator{\Div}{div}
\DeclareMathOperator{\Tr}{Tr}
\DeclareMathOperator{\Var}{var}

\newtheorem{theorem}{Theorem}[section]
\newtheorem*{theorem*}{Theorem}
\newtheorem{definition}[theorem]{Definition}
\newtheorem{lemma}[theorem]{Lemma}
\newtheorem{proposition}[theorem]{Proposition}
\newtheorem{corollary}[theorem]{Corollary}
\newtheorem{example}[theorem]{Example}
\newtheorem{remark}[theorem]{Remark}
\newtheorem{assumption}[theorem]{Assumptions}

\newcommand{\R}{\mathbb{R}} 
\newcommand{\st}{\,\hat{\otimes}\,}

\newcommand{\N}{\mathbb{N}}

\newcommand{\XX}{\mathbb{X}}
\newcommand{\Z}{\mathbf{Z}}
\newcommand{\ZZ}{\mathbb{Z}}

\newcommand{\FF}{\mathbb{F}}

\newcommand{\GG}{\mathbb{G}}
\newcommand{\WW}{\mathbb{W}}
\newcommand{\KK}{\mathbb{K}}

\newcommand{\la}{[\![}
\newcommand{\ra}{]\!]}

\makeatletter
\@tfor\next:=abcdefghijklmnopqrstuvwxyzABCDEFGHIJKLMNOPQRSTUVWXYZ\do{%
\def\command@factory#1{%
\expandafter\def\csname b#1\endcsname{\mathbf{#1}}
\expandafter\def\csname cl#1\endcsname{\mathcal{#1}}
}
\expandafter\command@factory\next
}

\title{Rough nonlocal diffusions}
\author{
	Michele Coghi \thanks{
		WIAS Berlin,
		Mohrenstra\ss e 39,
		10117 Berlin.
		Support from the Berlin Mathematics Research Center MATH+ is gratefully acknowledged.
	}, 
	Torstein Nilssen \thanks{
		Institute of Mathematics, Technical University of Berlin, Germany, Financial support by the DFG via Research Unit FOR 2402 is gratefully acknowledged.}}
\date{}

\begin{document}

\maketitle

\begin{abstract}
We consider a nonlinear Fokker-Planck equation driven by a deterministic rough path which describes the conditional probability of a McKean-Vlasov diffusion with "common" noise. To study the equation we build a self-contained framework of non-linear rough integration theory which we use to study McKean-Vlasov equations perturbed by rough paths. We construct an appropriate notion of solution of the corresponding Fokker-Planck equation and prove well-posedness.

\bigskip

MSC Classification Numbers: 60H05, 60H15, 60J60, 35K55.


\end{abstract}

\tableofcontents

\section{Introduction}

The term diffusion is sometimes used interchangeably when talking either about the macroscopic (Eulerian) description of the density of a substance occupying some space or the infinitesimal (Lagrangian) description of the particles of the substance. Many physical phenomena are however inherently nonlinear in the sense that the dynamic of the system will depend not only on space but also on the density of the substance itself.
In this paper we study this type of nonlinear diffusion from both the Eulerian and Lagrangian perspective when the diffusion is perturbed by a rough path. 
We are motivated by dynamics that arise from interacting particle systems with common noise; 
$$
dX_t^{i} = \frac{1}{N} \sum_{j=1}^N  b\left( X_t^j ,X_t^i \right) dt +  \frac{1}{N} \sum_{j=1}^N \sigma\left( X_t^j ,X_t^i \right) dW_t^i +  \frac{1}{N} \sum_{j=1}^N \beta\left( X_t^j ,X_t^i \right) \circ dB_t .
$$
Here each particle $X^i$ is influenced by 2 independent sources of noise, the Brownian motion $B$ \footnote{Since we will in this paper only consider geometric rough paths, we shall consider Stratonovich integration for this term.} is visible for all particles (common noise) and the Brownian motion $W^i$ represents a noise term specific for particle $X^i$. Since $B$ is influencing every particle, taking the limit $N \rightarrow \infty$ will only average out the individual noise terms, giving, at least formally, the mean-field dynamics
\begin{equation} \label{CommonNoiseEquation}
\left\{ 
\begin{array}{ll}
dx_t & = \int_{\R^d} b(\omega , x_t) d\mu_t(\omega) dt + \int_{\R^d} \sigma(\omega , x_t) d\mu_t(\omega) dW_t + \int_{\R^d} \beta(\omega , x_t) d\mu_t(\omega)  \circ dB_t \\
\mu_t & = \mathcal{L}(x_t | \mathcal{F}_t^B) . \\
\end{array}
\right.
\end{equation}
We note that the conditional law $\mathcal{L}(x_t | \mathcal{F}_t^B)$ heuristically satisfies the non-local Fokker-Plank equation

\begin{equation} \label{BrownianFokkerPlank}
d\mu_t = \frac12 \Tr  \nabla^2 (\sigma(\mu,\cdot )_t \sigma(\mu,\cdot )_t^T \mu_t) dt -  \operatorname{div}( b(\mu,\cdot )_t \mu_t)dt  - \operatorname{div}( \beta(\mu,\cdot )_t \mu_t ) \circ dB_t ,
\end{equation}
where we have used the notation $\sigma(\mu,x)_t = \int_{\R^d} \sigma(\omega,x) d\mu_t(\omega)$ etc. and $\Tr  \nabla^2(a)  = \sum_{i,j=1}^d \partial_i \partial_j a^{i,j}$ for a matrix valued function $a$. In fact, we can also address the case when $\sigma$ is a certain type of Lipschitz nonlinearity on $\mathcal{P}(\R^d) \times  \R^d $, where $\mathcal{P}(\R^d)$ denotes the set of probability measures on $\R^d$, see Assumption \ref{asm: mckean-vlasov coefficients}. We will only address the case when $\beta$ and $b$ are linear in their second argument. 

In practice, \eqref{BrownianFokkerPlank} is difficult to solve since it needs to be formulated on a very large state space, namely $[C([0,T]; \mathcal{P}(\R^d))]^{\Omega}$ where $\Omega$ is the underlying probability space. Even when $\Omega$ is \emph{finite}, this space is too large to do analysis since it is difficult to find compact subsets that is used for proving well-posedness of \eqref{CommonNoiseEquation} and \eqref{BrownianFokkerPlank}. For a long time, well-posedness for equation \eqref{BrownianFokkerPlank} was known only for densities, see \cite{kurtz1999particle}. A proper well-posedness result in the space of measures was obtain just very recently in \cite{CoghiGess}.

In this paper we take a different approach, namely we study equation \eqref{CommonNoiseEquation} for a fixed sample path of the Brownian motion. Our method relies on the theory of rough paths and as such, allows the study of \eqref{CommonNoiseEquation} where $B$ is replaced by any path that can be lifted to a rough path. In particular, no markovianity or martingale structure is needed for the common noise.

From now on we replace $B$ by a (deterministic) rough path $\Z = (Z, \ZZ)$, and equation \eqref{BrownianFokkerPlank} becomes
\begin{equation} \label{RoughFokkerPlank}
\partial_t \mu = \frac12 \Tr \nabla^2 (\sigma(\mu,\cdot ) \sigma(\mu,\cdot )^T   \mu) -  \operatorname{div}( b(\mu,\cdot ) \mu)  - \operatorname{div}( \beta(\mu,\cdot ) \mu ) \dot{\Z} .
\end{equation}
The main contribution of this paper is the following.
\begin{theorem*}[see Theorems \ref{thm: main existence} and \ref{thm:Uniqueness}]
	Given a probability measure $\mu_0$ on $\R^d$ with finite $\rho$-th moment, for any $\rho \geq 2$, there exists a unique measure-valued path $\mu:[0,T] \to \mathcal{P}(\R^d)$, which solves \eqref{RoughFokkerPlank} with initial condition $\mu_0$.
\end{theorem*}
Moreover we will prove in Theorem \ref{thm: main existence} that the unique solution is given as $\mu_t:= \mathcal{L}(x_t)$, namely the law of solution $x$ to the McKean-Vlasov equation
\begin{equation} \label{MeanFieldRoughSDE}
dx_t = b( \mathcal{L}(x_t),x_t)dt + \sigma( \mathcal{L}(x_t),x_t)dW_t + \beta( \mathcal{L}(x_t),x_t)d\Z_t.
\end{equation}
We will show well-posedness of \eqref{MeanFieldRoughSDE} in Section \ref{sec:ControlledMeasures}.

The strategy to prove uniqueness to equation \eqref{RoughFokkerPlank} relies on showing that every solution must be the law of the McKean-Vlasov equation. 
As it will be clear in the proof of Theorem \ref{thm:Uniqueness}, this also necessitates to be able to have well-posedness of the equation
\begin{equation} \label{TimeDependentRoughSDE}
dx_t = b_t(x_t)dt + \sigma_t(x_t)dW_t + \beta_t(x_t) d\Z_t,
\end{equation}
for given time inhomogeneous functions $b$,$\sigma$ and $\beta$, where the time dependence is induced by the law. 
Moreover, a common approach to proving well-posedness of \eqref{MeanFieldRoughSDE} is to construct the solution as a fixed point in the space of measures on an appropriate function space. Towards this end one would e.g. define inductively
$$
dx_t^{n+1} = b( \mathcal{L}(x_t^{n}),x_t^{n+1})dt + \sigma( \mathcal{L}(x_t^{n}),x_t^{n+1})dW_t + \beta( \mathcal{L}(x_t^{n}),x_t^{n+1}) d\Z_t .
$$
Once again, it is necessary to give a meaning to equation \eqref{TimeDependentRoughSDE}. If we consider the case $b = \sigma = 0$ and $\beta_t(x) = \beta_t$ the equation reads
$$
dx_t = \beta_t d\Z_t .
$$
It is well-known that the above integration does not make sense unless we impose additional structure on $\beta$, namely that there exists a Taylor-type expansion around the irregular path $Z$, which is exactly the notion of controlled rough paths as introduced by Gubinelli in \cite{Gub04}. If one aims to solve a mean-field equation on the form
$$
dx_t = \beta(\mathcal{L}( x_t),x_t ) d\Z_t,
$$
where $\mathcal{L}(x_t)$ denotes the law of $x$, and $\beta$ is an appropriate function on the space of measures, it is reasonable to expect that $t \mapsto \beta( \mathcal{L}( x_t),x)$ has such a decomposition and that one could solve the equation as a fixed-point in an appropriate space of measures.

Following this logic, if we want to consider the equation with added Brownian motion \eqref{MeanFieldRoughSDE}
as a fixed-point, this would necessitate being able to solve equation \eqref{TimeDependentRoughSDE}.
The usual way, see \cite{DiFrSt14}, \cite{DOR}  and \cite{FNS}, to study this hybrid rough path and It\^{o} equation is to consider the joint rough path
\begin{equation} \label{JointLift}
\left(\begin{array}{c}
\mathbf{W}_{st} \\
\Z_{st} \\
\end{array} \right)
 := 
\left( \left(\begin{array}{c}
W_{t} - W_{s} \\
Z_{t} - Z_{s} \\
\end{array} \right) ,
\qquad
\left(\begin{array}{cc}
\int_s^t (W_{r} - W_{s}) dW_r & \int_s^t (W_{r} - W_{s}) dZ_r \\
\int_s^t  (Z_{r} - Z_{s}) dW_r & \ZZ_{st} \\
\end{array} \right) \right),
\end{equation}
and recast the equation on the form of a rough path equation
$$
dx_t = \left( \begin{array}{c}
\sigma_t \\
\beta_t \\
\end{array} \right) (x_t) d 
\left( \begin{array}{c}
\mathbf{W}_t \\
\Z_t \\
\end{array} \right) .
$$
Again, one would need to make an expansion of $(\sigma_t, \beta_t)^T $
in terms of the path $(W,Z)^T$.
However, thinking towards the goal of solving mean-field equations, the simplest examples shows that there is no reason to expect that $\sigma_t$ is controlled by a fixed Brownian path in any sense - the law of the solution is an average over \emph{all} Brownian sample paths.

Instead, if we define $W_{st}^{\sigma}(x) = \int_s^t \sigma_r(x) dW_r$ as a Wiener-It\^{o} integral and $Z_{st}^{\beta}(x) = \int_s^t \beta_r(x) d\Z_r$ as a rough path integral, then on small time scales one would expect
$$
\left| \int_s^t \sigma_r(x_r) dW_r - W_{st}^{\sigma}(x_s) \right| \vee \left| \int_s^t \beta_r(x_r) d\Z_r - Z_{st}^{\beta}(x_s) \right|,
$$
to be small, so that one could use $W^{\sigma}$ and $Z^{\beta}$ to define a notion of non-linear \footnote{We choose to call the integration non-linear since a mapping $x \mapsto \int f_r(x_r) dr$ is obviously never linear.} integration. 
At the heart of all stochastic integration is the difficulty that the above is \emph{not} enough to guarantee a canonically defined integration map in the pathwise sense. The most fundamental understanding of the rough path theory is that one can construct integrals once additional information about the driving path is given by some off-line argument e.g. stochastic integration.

\subsubsection*{Existing literature}
The stochastic equation, i.e. \eqref{CommonNoiseEquation} and \eqref{BrownianFokkerPlank} has been studied in \cite{kurtz1999particle} and \cite{MR1797090} but focusing on the case where the initial condition has a density. The measure-valued case was studied very recently in \cite{CoghiGess}. Under more restrictive conditions, either on the class of solutions or on the coefficients (like strong parabolicity), the well-posedness of solutions to SPDE of the type \eqref{BrownianFokkerPlank} had been previously considered by Dawson, Vaillancourt in \cite{DV95}.

McKean-Vlasov equations from a rough path perspective has already been introduced in \cite{CassLyons} and, more recently in \cite{bailleul2018mean}, focusing on the Lagrangian description. In \cite{bailleul2018mean} the equation is driven by a general random rough path, which gives the additional difficulty of needing to keep track of the rough path as a $L^p(\Omega)$-valued path. The latter space is present to consider a probability measure as the law of a random variable and Lions' approach to calculus for the Wasserstein metric. 
The approach by Gubinelli on controlled rough paths is then used to solve the equation as a fixed-point in the mixed $\R^d$ and $L^p(\Omega)$-space.

We mention also \cite{carmona2016} where the authors study mean-field games in the presence of a common noise as in \eqref{CommonNoiseEquation}. The authors use tightness arguments along with approximations to prove existence of a (probabilistically) weak solutions. Then, the authors prove a Yamada-Watanabe type principle for these equations to prove existence and uniqueness of (probabilistically) strong solutions.

\bigskip

In Section \ref{sec:NonLinear} we build a version of the rough path theory that allow for time dependent coefficients. 
The results in this section should be compared to \cite{BR} where the authors solves equations on this form. There, the main focus is flows build from a non-linear version of the sewing lemma.
Very recently, right before the completion of the present paper, the authors of \cite{NualartXia} introduce the very same object, here called a nonlinear rough path. The authors use a similar set up as in \cite{Gub04} to solve rough equations with time-dependent coefficients. 

The papers \cite{BR} and \cite{NualartXia} does not contain the same precise estimates as the present paper, which is crucially needed to set up a contraction mapping for the McKean-Vlasov equation \eqref{MeanFieldRoughSDE}.

\subsubsection*{Main contributions}

The main contribution of this paper is the formulation and well-posedness of the nonlinear Fokker-Planck equation in terms of the appropriate rough path topology. We believe this is the first paper to study a rough non-local diffusion from both the Lagrangian and Eulerian perspective. Furthermore we believe it is the first work to prove well-posedness of an equation with a nonlinearity in the noise term on this form.

It is plausible that the well-posedness of the McKean-Vlasov equation equation in the present paper can be seen as a particular case of the equation studied in \cite{bailleul2018mean} by doing a rough path lift of $W$ and $\bZ$ as in \eqref{JointLift}, but now as a rough path with values in an $L^p(\Omega)$-space.
However, our proof of the well-posedness of the nonlinear Fokker-Planck equation necessitate well-posedness of a rough path equation with time-dependent coefficients. As already mentioned, it is not reasonable to expect that the coefficients could be controlled by a single Brownian path thus one could not use \cite{bailleul2018mean} for the time dependent case. 
Moreover, for the same reason, time dependent coefficients are also needed to understand the McKean-Vlasov equation as a fixed point of \emph{linear} diffusions in an appropriate space of measures.

In addition, we prove a result on existence of a solution to a linear, possibly degenerate, rough PDE which could be of independent interest. 

\subsubsection*{Structure of the paper}

The paper is structured as follows. In Section \ref{sec:Notations} we introduce the necessary concepts from rough path theory, including controlled rough paths, that will be needed for the paper. In Section \ref{sec:NonLinear} we introduce the corresponding integration theory to handle non-linear integration and differential equations. In Section \ref{sec: rough non-linearities} we show how to concretely build rough drivers from It\^{o} integration theory and the theory of controlled paths. These examples will also act exactly as the rough drivers needed to formulate the McKean-Vlasov equation as a fixed point. Moreover, this section contains an average, in $\Omega$, It\^{o} formula that allows us to prove that the law of a diffusion solves the Fokker-Planck equation (linear or nonlinear). 
In Section \ref{sec:LinearPDE} we prove well-posedness for a linear RPDE with time dependent coefficients. 
In Section \ref{sec:ControlledMeasures} we construct the appropriate space for solving the McKean-Vlasov equation. 
In Section \ref{sec:NonLinearFP} we prove uniqueness of our main equation, which hinges on the results of the previous sections.

\section{Notations and preliminary results} \label{sec:Notations}

\subsection{H\"{o}lder and p-variation spaces}
\label{section: Holder and p-variation spaces}
For $T>0$ we let $\Delta_T$ denote the simplex $\Delta_T = \{ (s,t) \in [0,T]^2 : s < t\}$. For $\zeta >0$ and a Banach space $E$ we denote by $C_2^{\zeta}([0,T];E)$ the space of all continuous mappings $g : \Delta_T \rightarrow E$ such that
$$
[g]_{\zeta, h;E} := \sup_{(s,t) \in \Delta_T : |t-s| \leq h} \frac{\|g_{st}\|_E }{|t-s|^{\zeta}} < \infty .
$$
It can be checked that the above space is independent of $h$, and we will write for simplicity $[g]_{\alpha;E} := [g]_{\alpha,T;E}$. When it is clear from the context, we will also omit the Banach space $E$, writing $[g]_{\alpha,h}$ and $[g]_{\alpha}$. We let $C^{\zeta}([0,T];E)$ denote the space of all paths $f: [0,T] \rightarrow E$ such that the increment $\delta f_{st} := f_t - f_s$ belongs to $C^{\zeta}_2([0,T];E)$. For simplicity we will write $[f]_{\alpha,h;E} := [\delta f]_{\alpha,h;E}$.
It is well known that local and global H\"{o}lder norms are comparable for paths, in the sense that
\begin{equation} \label{eq:LocalToGlobal}
[f]_{\zeta;E} \leq [f]_{\zeta,h;E} (1 \vee 2 h^{\zeta - 1})
\end{equation}
for all $f \in C^{\zeta}([0,T];E)$ (see Exercise 4.25 in \cite{FrizHairer}). It is well known that the H\"{o}lder spaces are not separable. However, the subspace
$$
C_0^{\alpha}([0,T];E) := \left\{ f \in C^{\alpha}([0,T];E) : \lim_{h \rightarrow 0} [f]_{\alpha,h} = 0 \right\}
$$
is separable, as proved in Proposition \ref{prop:C0 separable}.

We let $C_2^{p-\Var}([0,T];E)$ be the space of all continuous mappings $g : \Delta_T \rightarrow E$ such that
$$
\la g \ra_{p,[s,t];E}  := \left( \sup_{\pi} \sum_{ \{t_i\} = \pi} \|g_{t_i t_{i+1}}\|_E^p \right)^{1/p} < \infty
$$
where the above supremum is taken over all partitions $\pi$ of $[s,t]$. If we define $w_g(s,t) := \la g  \ra_{p,[s,t];E}^p$ it can be shown that $(s,t) \mapsto w_g(s,t)$ is a control, namely continuous and superadditive i.e. $w_g(s,u) + w_g(u,t) \leq w_g(s,t)$. Moreover, we see that if there exists a control $w$ such that $\|g_{st}\|_E \leq w(s,t)^{1/p}$, then $w_g(s,t) \leq w(s,t)$, so that we could equivalently define
$$
\la  g  \ra_{p,[s,t];E}  =  \inf \left\{ w(s,t)^{1/p} \mid w \textrm{ is a control such that } \|g_{uv}\|_E \leq w(u,v)^{1/p} \textrm{ for } s \leq u < v \leq t \right\}.
$$
We will write $\la g \ra_{p;E} := \la g \ra_{p,[0,T];E}$ and when the space $E$ is clear from the context we will simply write $\la g \ra_{p,[s,t]}$ and $\la g \ra_p := \la g \ra_{p,[0,T]}$.

To see the relationship between H\"{o}lder continuity and $p$-variation, notice that for any partition $\pi$ we have
$$
\sum_{\pi} \|g_{t_i t_{i+1}} \|_E^p \leq \sum_{\pi} [g]_{\alpha;E}^p |t_{i+1} - t_i|^{\alpha p} = [g]_{\alpha;E}^p |t-s| 
$$
when $\alpha = 1/p$, which gives the bound 
\begin{equation} \label{eq:PVarToHolder}
w_g(s,t) \leq [g]_{\alpha;E}^{1/\alpha} |t-s|.
\end{equation}

Given a control $w$, we construct the \emph{greedy partition}, following \cite[Chapter 11]{FrizHairer}; for $\beta > 0$, define the partition $\{\tau_n\}_n$ as 
\begin{equation*}
\tau_0 = s, \qquad \tau_{n+1} = \inf\{ t \mid w(\tau_n,t) \geq \beta, \tau_n < t \leq T \} \wedge T,
\end{equation*}
so that $w(\tau_n, \tau_{n+1}) = \beta$, for all $n < N$, and $w(\tau_N, \tau_{N+1}) \leq \beta$. Define now the integer
\begin{equation}\
\label{eq: N beta}
N_{\beta}(w,[s,t]) := \sup\;\{ n\geq 0 \mid \tau_n <t \}.
\end{equation}

\subsection{Rough paths}
Assume $E$ is a Banach space and equip $E \otimes E$ with the projective tensor norm.
We call a pair 
$$
\bZ := (Z,\ZZ) \in C^{\alpha}([0,T];E) \times C^{2 \alpha}_2([0,T];E \otimes E)
$$
for $\alpha \in (\frac13 , \frac12)$ a rough path provided Chen's relation, 
\begin{equation}
\label{eq: chen rough paths}
\delta \ZZ_{s \theta t} = Z_{s \theta} \otimes Z_{\theta t},
\end{equation}
holds where we have defined the second order increment operator $\delta g_{s \theta t} : = g_{s t} - g_{\theta t} - g_{s \theta}$.
We denote by $\mathscr{C}^{\alpha}([0,T];E)$ the (non-linear) set of all rough paths which we equip with the subset metric,
$$
[\bZ - \bX]_{\alpha,h} := [Z - X]_{\alpha,h} + [\ZZ - \mathbb{X}]_{2 \alpha,h}. 
$$
For a path of bounded variation, $Z:  [0,T] \rightarrow E$ there is a canonical rough path, $\bZ = (Z, \int Z \otimes dZ)$ where the latter is the iterated integral
$
\big( \int Z \otimes dZ \big)_{st} = \int_s^t Z_{sr} \otimes dZ_r
$
which is well defined when $Z$ is of bounded variation. We denote by $\mathscr{C}_g^{\alpha}([0,T];E)$ the set of geometric rough paths, which is the closure of the set of bounded variation paths in the rough path metric. 

We notice that if $\bZ$ is geometric, then $\bZ$ is also weakly geometric which means $\textrm{sym}(\ZZ_{st}) = \frac12 Z_{st} \otimes Z_{st}$, and we denote by $\mathscr{C}_{wg}^{\alpha}([0,T];E)$ the set of all such rough paths. 
When $E$ is finite dimensional it is known that (see e.g. \cite[Proposition 8.12]{friz_victoir_2010}) if $\bZ$ is weakly geometric, there exists a sequence of smooth paths $Z^n$ such that $\bZ^n \rightarrow \bZ$ in $\mathscr{C}^{\bar{\alpha}}([0,T];E)$ for all $\bar{\alpha} < \alpha$. 

\subsubsection*{Controlled space}
Given a path $Z$ taking values in $\R^m$ we denote by $\mathscr{D}_Z^{2 \alpha}([0,T];E)$ the (linear) space of all controlled path, given by pairs $(Y,Y')$ of mappings 
$$
Y:[0,T] \rightarrow \mathcal{L}(\R^m; E), \qquad Y' : [0,T] \rightarrow \mathcal{L}(\R^{m \times m}; E)
$$
such that 
$$
Y^{\sharp}_{st} := \delta Y_{st} - Y_s' Z_{st} , \quad \Longrightarrow \quad Y^{\sharp} \in C_2^{2 \alpha}([0,T]; \mathcal{L}(\R^m;E)) .
$$
We call $Y'$ the Gubinelli derivative of $Y$. The above definition is sometimes better understood in coordinates
$
Y^{\sharp, i}_{st} := \delta Y_{st}^i - Y_s^{i,k} Z_{st}^k
$
where we abuse notation and write $Y^{i,k}$ for the matrix representing the Gubinelli derivative. Above and for the remainder of the paper we shall use the convention of summation over repeated indices. We equip the space of all controlled paths with the norm
$$
\|(Y,Y')\|_{Z,\alpha,h;E} := |Y_0| + [Y']_{\alpha,h;E} + [Y^{\sharp}]_{2 \alpha,h;E} . 
$$

\subsubsection*{Sewing lemma and rough path integration}

We recall here the main result used to obtain estimates in the theory of rough paths, namely the sewing lemma.

\begin{lemma} \label{SewingLemma}
Suppose $g : \Delta_T \rightarrow E$ is such that 
$$
[\delta g]_{\zeta,h;E} := \sup_{ s  < \theta < t : |t-s| \leq h} \frac{ \| \delta g_{s \theta t} \|_E }{|t-s|^{\zeta}} < \infty
$$
for some $\zeta >1$ and $h > 0$. Then there exists a unique pair $I(g): [0,T] \rightarrow E$ and  $I(g)^{\natural}: \Delta_T \rightarrow E$ such that 
$$
\delta I(g)_{st} = g_{st} + I(g)_{st}^{\natural}
$$
with $[I(g)^{\natural}]_{\zeta;E} \leq C[\delta g]_{\zeta,h;E}$ for $C$ depending only on $\zeta$.

In fact, we have $I(g)_{st} := \lim_{|\pi| \rightarrow 0} \sum_{\pi} g_{t_i t_{i+1}}$ and we think of $I(g)$ as being an integral with local expansion $g$. 
\end{lemma}

With this in hand we can define the rough path integral. Given a rough path $\bZ$ and a controlled path $(Y,Y') \in \mathscr{D}_Z^{2 \alpha}([0,T]; E)$, define the local expansion
$$
g_{st} := Y_s Z_{st} + Y_s' \ZZ_{st} := Y_s^k Z_{st}^k  + Y_s^{k,l} \ZZ_{st}^{l,k}  .
$$
Using Chen's relation it is straightforward to check that $[\delta g]_{3 \alpha;E} < \infty$ and we shall write $\int Y d\bZ := I(g)$.

This construction also gives rise to a new rough path, namely
\begin{equation}
\label{eq: rough path from rough integration}
X_t = \int_0^t Y_r d\bZ_r  , \qquad \XX_{st} = \int_s^t X_r \otimes  Y_r d\bZ_r - X_{s} \otimes X_{st}
\end{equation}
where the latter integral is defined by the local expansion
$$
X_s \otimes Y_s^{k} Z_{st}^{k} +  (Y^{l}_s \otimes  Y_s^{k} + X_s \otimes Y_s^{k,l} )  \ZZ_{st}^{l,k}.
$$
One can then check that $\bX := (X,\XX) \in \mathscr{C}^{\alpha}([0,T]; E)$ and that this operation is continuous from $\mathscr{D}_Z^{2 \alpha}([0,T];E)$ to $\mathscr{C}^{\alpha}([0,T]; E)$. Moreover, at least when $E$ is a separable Hilbert space, weak geometricity is preserved under rough path integration as spelled out in Lemma \ref{lem:WGtoWG}.

We shall also use the sewing lemma to get a priori estimates by a slight (straightforward) generalization of the sewing lemma.  
Assume that $g$ is such that there exists controls $w$ and $w_*$ and a positive function $k$ such that 
\begin{equation} \label{APrioriSewingFirst}
|\delta g_{sut} | \leq  w(s,t)^{\zeta} ( 1 + k_s) , \quad  |g_{st}| \leq  w_*(s,t)^{\zeta}
\end{equation}
for some $\zeta > 1$. Then there exists a universal constant $C$ such that
\begin{equation} \label{APrioriSewing}
|g_{st}| \leq C w(s,t)^{\zeta} (1 +  \sup_{r \in [s,t]} k_r) .
\end{equation}

\subsection{Taylor's formula}
For a path $y: [0,T] \rightarrow \R^d$ and a function $g : \R^d \rightarrow V$ (where $V$ is a finite-dimensional vector space) we use the notation
\begin{equation}
\label{def: brackets}
[g]^{k,y}_{st} :=  \int_0^1 (1 - \theta)^{k-1} g(y_s + \theta \delta y_{st}) d \theta. 
\end{equation}
With this notation at hand the first and second order Taylor's formula reads
$$
\delta g(y)_{st} = [\nabla g]_{st}^{1,y} \delta y_{st},  \qquad [ g]_{st}^{1,y} - g(y_s) =  [\nabla g]_{st}^{2,y} \delta y_{st}
$$
respectively. We obviously get $|[g]^{k,y}_{st}| \lesssim \|g\|_{\infty}$.

\subsection{Wasserstein metric}

We shall work with the Wasserstein metric on measures on H\"{o}lder spaces, but since separability of the underlying space is required for the Wasserstein metric to give a complete space, we shall use the subspaces $C_0^{\alpha}([0,T];\R^d)$. When the dimension is clear from the context we shall simply write $C_0^{\alpha}$. Given two probability measure $\mu, \nu \in \mathcal{P}(C^{\alpha}_0)$ say that $\pi \in \mathcal{P}(C^{\alpha}_0 \times C^{\alpha}_0)$ is a coupling of $\mu$ and $\nu$ provided its first (respectively second) marginal is equal to $\mu$ (respectively $\nu$). 
We define the Wasserstein metric 
$$
W_{\rho}(\mu, \nu) := \inf_{ \pi}  \left( \int_{C^{\alpha}_0 \times C^{\alpha}_0} [\omega - \bar{\omega}]_{\alpha}^{\rho} d \pi(\omega, \bar{\omega})  \right)^{1/\rho}
$$
where the above infimum ranges over all couplings $\pi$ of the measures $\mu$ and $\nu$. 
Since $C_0^{\alpha}$ is separable we have that $\mathcal{P}_{\rho}(C_0^{\alpha})$ is a complete space w.r.t. $W_{\rho}$. 

We note that the $\rho$-th moment of a probability measure $\mu$ can be written $W_{\rho}(\mu, \delta_0)^{\rho}$ where $\delta_0$ is the Dirac-Delta centered in the path constantly equal to $0$.

\subsection{Spatial function spaces}
	We fix $d \in \N$.
For any multi-index $\beta = (\beta_1, \dots, \beta_d)$, we set
\begin{equation*}
D^{\beta} = \left( \frac{\partial}{\partial x_1}\right)^{\beta_1} \left( \frac{\partial}{\partial x_2}\right)^{\beta_2} \cdots \left( \frac{\partial}{\partial x_d}\right)^{\beta_d}
\end{equation*}
and $\vert \beta \vert = \beta_1 + \cdots + \beta_d$.
For $p>1$ and an integer $k\geq 0$, we let $W^{k, p} = W^{k, p}(\R^d)$ be the Sobolev space of real-valued functions on $\R^d$ with finite norm
\begin{equation*}
\Vert f\Vert_{W^{k, p} }:= \left( \sum_{\vert \beta \vert \leq k} \int_{\R^d} \vert D^{\beta}f(x) \vert^p dx \right)^\frac{1}{p} < \infty.
\end{equation*}
Let $H^k := W^{k,2}(\R^{d}; \R^{d})$, be the Sobolev space of square integrable functions over $\R^d$, endowed with the norm $\| \cdot \|_{H^k} := \|\cdot \|_{W^{k,2}}$.
For a Hilbert space $H$, we endow the space of linear functionals $\mathcal{L}(\R^d; H)$ with the Hilbert-Schmidt norm
\begin{equation}
\label{def: hilberth schmidt}
\| A \|_{\mathcal{L}(\R^d; H)} := \left(\sum_{i=0}^d \|Ae_i\|_{H}^2\right)^{\frac{1}{2}},
\qquad
A\in \mathcal{L}(\R^d; H).
\end{equation}
Moreover, we call $M^2_T(H)$ the space of $H$-valued, time-continuous, square integrable martingales endowed with the norm
\begin{equation*}
\|M\|_{M^2_T(H)} := \sup_{t\in[0,T]}\| M_t \|_{L_{\omega}^2}.
\end{equation*}

Let $k > \frac{d}{2}$. We denote by $C_b^3 \otimes H^k$ the space of continuous functions $f : \R^d \times \R^d \rightarrow \R^d$ such that 
\begin{enumerate}
	\item For all $x\in \R^d$, the function $y\mapsto f(x,y) \in H^k$.
	\item For all $y\in \R^d$, the function $x\mapsto f(x,y) \in C_b^3$.
	\item We have
	\begin{equation}
	\label{def: tensor norm}
	\|f\|_{C_b^3 \otimes H^k}
	:=
	\left(\sum_{ 0\leq i \leq 3, |\beta|\leq k}\sup_{x\in \R^d} \int_{\R^d} |\nabla_1^i D_2^{\beta} f(x,y) |^2dy\right)^{\frac{1}{2}}
	< \infty,
	\qquad
	f\in C_b^3 \otimes H^k.
	\end{equation}
\end{enumerate}
We endow the space $C_b^3 \otimes H^k$ with the induced norm $\|f\|_{C_b^3 \otimes H^k}$. Above we have used the Frechet derivative in the first variable and the weak derivative in the second variable. 

Contrary to $H^l \otimes H^k$, this space is well suited for the convolution $f(x,y) = \sigma(x-y)$ and we see that $f \in C_b^3 \otimes H^k$ if $\sigma \in H^{3 + k}$.

\section{Non linear integration} \label{sec:NonLinear}
In this section we build the theory of rough paths to accommodate for time-dependent coefficients.
We aim to solve the equation
\begin{equation} \label{NonLinearEq}
\dot{x}_t = f_t(x_t), \quad x_0 = \xi \in \R^d
\end{equation}
for given function $f$ which is a distribution in time but regular in space. We shall use the framework akin to the definition by Davie in \cite{Davie}. To illustrate the set up, assume that $x$ is a smooth solution of \eqref{NonLinearEq}. Integrating the equation and using Taylor's formula we obtain
\begin{align*}
\delta x_{st} & = \int_s^t f_r(x_r) dr = \int_s^t f_r(x_s) + \nabla f_r(x_s) (\delta x_{sr})  + [\nabla^2 f_r]^{2,x}_{sr}( \delta x_{sr} \otimes \delta x_{sr}) dr \\
 & = \int_s^t f_r(x_s) dr +  \int_s^t [\nabla^2 f_r]^{2,x}_{sr}( \delta x_{sr} \otimes \delta x_{sr}) dr \\
  &  \quad  + \int_s^t \nabla f_r(x_s) \left(  \int_s^r f_u(x_s) du + \int_s^r \nabla f_u(x_s) ( \delta x_{su}     ) du  + \int_s^r [\nabla^2 f_u]^{2,x}_{su}( \delta x_{su} \otimes \delta x_{su}) du   \right)   dr\\
   & = F_{st}(x_s) + \FF_{st}(x_s)  +  x_{st}^{\natural}.
\end{align*}
Here we have defined the \textit{driver} $\bF := (F, \FF)$ of the equation as follows
\begin{equation} \label{eq:RDDefinition}
F_{st}(x) := \int_s^t f_r(x) dr \qquad \FF_{st}(x) := \int_s^t \nabla f_r(x) F_{sr}(x) dr,
\end{equation}
and the remainder as
\begin{equation} \label{ExplicitRemainder}
x_{st}^{\natural} := \int_s^t \int_s^r \nabla f_r(x_s)  \nabla f_u(x_s) ( \delta x_{su}     ) + [\nabla^2 f_u]^{2,x}_{su}( \delta x_{su} \otimes \delta x_{su}) du     dr
   + \int_s^t [\nabla^2 f_r]^{2,x}_{sr}( \delta x_{sr} \otimes \delta x_{sr}) dr.
\end{equation}
With the above notation, we rewrite equation \eqref{NonLinearEq} as
\begin{equation}
\label{eq: non linear equation}
dx_t = \bF_{dt}(x_t).
\end{equation}
As is usual in rough path theory, we shall now read the definition \eqref{eq:RDDefinition} in the opposite direction - we assume we are given a pair of functions $(F,\FF)$ satisfying some compatibility conditions (in Definition \ref{def:RoughDriver} below), and take this as a definition of the non-linearity $f$. We will then take $x^{\natural}$ to be implicitly defined and say that $x$ is a solution provided $x^{\natural}$ is of high time regularity.

We can read \eqref{NonLinearEq} in integral form as
$
x_t = x_0 + \int_0^t \bF_{dr}(x_r)
$
and can be regarded as a rough version of the semimartingale integration theory by Kunita in \cite{Kunita}.

We shall use a similar definition as in \cite{BR}, with a noticeable difference that we allow our driver to depend on two spatial points. Moreover, we will not only be dealing with weakly geometric drivers. 

\begin{definition} \label{def:RoughDriver}
For $p \in [2,3)$, a pair of functions $\bF= (F, \FF) \in C^{p- \Var}([0,T]; C_b^3(\R^d;\R^d)) \times C_2^{\frac{p}{2} - \Var}([0,T]; C_b^2(\R^d \times \R^d;\R^d))$ is called a $p$-\emph{rough driver} provided Chen's relation,
\begin{equation} \label{RDChen}
\delta \FF_{s u t}(x,y) =  F_{su}(x) \otimes \nabla F_{ut}(y)  : = F_{su}^i(x) \partial_i F_{ut}(y)
\end{equation}
holds. The set of all such pairs is equipped with the metrics
$$
\la \bF - \bG \ra_{p,[s,t]} := \la F - G \ra_{p, [s,t]; C_b^3} + \sqrt{ \la \FF - \GG  \ra_{ \frac{p}{2}, [s,t]; C_b^2}}.
$$
Most of the time we will work on the diagonal of the spatial points and write simply $\FF_{st}(x) := \FF_{st}(x,x)$, and we shall also write $\nabla F_{ut}(x) F_{su}(x) = F_{su}(x) \otimes \nabla F_{ut}(x)$.

For $\alpha \in (\frac13, \frac12]$ a pair of functions $\bF  = (F, \FF) \in C^{\alpha}([0,T]; C_b^3(\R^d;\R^d)) \times C_2^{ 2 \alpha}([0,T]; C_b^2(\R^d \times \R^d;\R^d))$  is called an $\alpha$-rough driver provided \eqref{RDChen} holds. The set of all such pairs is equipped with the metric 
$$
[\bF - \bG ]_{\alpha,h} := \|F - G\|_{\alpha,h; C_b^3} + \sqrt{\|\FF - \GG\|_{ 2 \alpha,h; C_b^2}}.
$$
\end{definition}

\begin{remark}
The reason for using both $p$-variation and $\alpha$-H\"{o}lder continuous drivers is that the construction using Kolmogorov continuity theorem (Lemma \ref{lem: rough drivers from Ito}, below) gives us more easily bounds in the sense of H\"{o}lder continuity. However, to estimate the difference between two solutions we need exponential bounds, and it is well known that even when $W$ is a Brownian motion, the random variable $[W]_{\alpha}$ is not exponentially integrable. This problem is circumvented by using $p$-variation, more specifically using the local accumulation $N_1(\|W\|_{p-var;[\cdot, \cdot]} ,[0,T])$, see Section \ref{section: integrability rough driver} for the details. 

From \eqref{eq:PVarToHolder} it is clear that if $\bF$ is an $\alpha$-rough driver, then it is also a $p$-rough driver with $p = \frac{1}{\alpha}$. When the notion is clear from the context, we shall simply say that $\bF$ is a rough driver. 
\end{remark}
\begin{example} \label{ex:PathToDriver}
Consider a rough path $\bX \in \mathscr{C}^{\alpha}([0,T]; C_b^3(\R^d;\R^d))$, where we identify $C_b^3(\R^d;\R^d) \otimes C_b^3(\R^d;\R^d)$ with a subspace \footnote{Since we are on the unbounded domain $\R^d$, we don't know if one can identify these spaces, but the inclusion is enough for our purposes} of $C_b^3(\R^d \times \R^d;\R^{d \times d})$ so that Chen's relation reads
$$
\delta \XX_{sut}^{i,j}(x,y) =  X_{su}^i(x)  X_{ut}^j (y).
$$
Let now $F_{st}(x) = X_{st}(x)$ and $\FF_{st}(x,y) = \nabla_y^{\otimes} ( \XX_{st}(x,y))$ where $\nabla_2^{\otimes}: C_b^3(\R^d \times \R^d; \R^{d \times d}) \rightarrow C_b^3(\R^d \times \R^d; \R^{d })$ is the multiplication of vector fields, i.e. the linear extension of the mapping defined by
$$
\big( \nabla_2^{\otimes}( f \otimes g)(x,y) \big)^j = f^i (x) \partial_i   g^j(y).
$$
It is straightforward to check that this gives a rough driver, and we notice that the mapping $\bX \mapsto \bF$ is continuous. 
\end{example}
With this at hand we can define the notion of a solution.
\begin{definition} \label{def:RDSolution}
Let $\bF$ be a rough driver as in Definition \ref{def:RoughDriver} and $\xi \in \R^d$. A path $x : [0,T] \rightarrow \R^d$ is called a solution to \eqref{eq: non linear equation} provided $x^{\natural}$ defined by
\begin{equation} \label{ExpansionEquation}
\delta x_{st} = F_{st}(x_s)  + \FF_{st}(x_s)  + x_{st}^{\natural},
\qquad
x_0 = \xi,
\end{equation}
is such that $x^{\natural} \in C_2^{\frac{p}{3} - \Var}([0,T]; \R^d)$.
\end{definition}

\begin{remark}
One drawback with this method compared to linear integration is the lack of "universality" in the It\^{o}-Lyons map; recall that the stochastic equation
$$
dx_t = V(x_t) \circ dB_t
$$
and its corresponding mapping  $B \mapsto x$ can be factorized into a discontinuous map, $B \mapsto (B, \int B dB)$  and a continuous one $(B, \int BdB) \mapsto x$. One of the nice features of this decomposition is the fact that $B \mapsto (B, \int B dB)$ is universal in the sense that it does not depend on the vector field $V$ driving the equation, which allows to fix a subset $\Omega_0 \subset \Omega$ for which one can do deterministic analysis on the differential equation. 

In our case, however, the subset of $\Omega$ will depend on the driving vector fields since we are building a non-linear integration theory depending on the coefficients. 
\end{remark}

\subsection{A priori estimates}
Let $\bF$ be a $p$-rough driver and assume $x$ is a solution of equation \eqref{eq: non linear equation} in the sense of Definition \ref{def:RDSolution}. In this section we use \eqref{APrioriSewingFirst} and \eqref{APrioriSewing} to deduce a priori estimates. We let $w_{\bF}$ be the smallest control such that 
$$
\|F_{st}\|_{C_b^3} \leq w_{\bF}(s,t)^{1/p} , \quad \|\FF_{st}\|_{C_b^2} \leq w_{\bF}(s,t)^{2/p}.
$$
Define the controlled quantity,
\begin{equation}
\label{def: x sharp}
x_{st}^{\sharp} := \delta x_{st} - F_{st}(x_s) = \FF_{st}(x_s) + x_{st}^{\natural}.
\end{equation}

\begin{lemma} \label{Composition}
Let $g\in C_b^2$, we have the following chain rule, $\forall s,t \in [0,T]$,
\begin{equation} \label{eq:CompositionRemainder}
g(x)_{st}^{\sharp} := \delta g(x)_{st} - \nabla g(x_s) F_{st}(x_s) 
\quad \Longrightarrow \quad 
|g(x)^{\sharp}_{st}| \leq \|g\|_{C_b^2} ( w_{\bF}(s,t)^{1/p} w_x(s,t)^{1/p}  + w_{x^{\sharp}}(s,t)^{2/p}).
\end{equation}
\end{lemma}

\begin{proof}
We have from Taylor's formula
$$
\delta g(x)_{st} = [\nabla g]^{1,x}_{st} \delta x_{st} = [\nabla g]^{1,x}_{st} F_{st}(x_s) + [\nabla g]^{1,x}_{st} x_{st}^{\sharp} = \nabla g(x_s) F_{st}(x_s) + g(x)_{st}^{\sharp},
$$
where
$$
g(x)_{st}^{\sharp} = \left( [\nabla g]^{1,x}_{st} - \nabla g(x_s) \right) F_{st}(x_s) + [\nabla g]^{1,x}_{st} x_{st}^{\sharp}.
$$
By the definition of brackets \eqref{def: brackets}, we get
$$
|[\nabla g]^{1,x}_{st} - \nabla g(x_s)| =  \left| \int_0^1 \nabla g(x_s + \theta \delta x_{st}) - \nabla g(x_s) d \theta \right| \leq \|g\|_{C_b^2} |\delta x_{st}| .
$$
The result follows.
\end{proof}

With this in hand we turn to an a priori estimate for the nonlinear RDE.
\begin{proposition} \label{prop:APriori}
Let $0<h\leq T$. There exists constants $C$ and $h$ depending only on $p$ such that for all $s,t$ such that $w_{\bF}(s,t) \leq h$ we have
$$
|x_{st}| \leq C w_{\bF}(s,t)^{1/p}, \quad |x_{st}^{\sharp}| \leq C w_{\bF}(s,t)^{2/p},  \quad |x_{st}^{\natural}| \leq C w_{\bF}(s,t)^{3/p},
$$
\end{proposition}

\begin{proof}
We start with the easily verifiable identity for a function $G$ and path $y$
$$
\delta G(y)_{sut} = \delta G_{sut}(y_s)  - \delta (G_{ut}(y_{\cdot}))_{su}.
$$
Using Chen's relation we get
\begin{align*}
\delta x_{sut}^{\natural} & = \delta (F_{ut}(x_{\cdot}))_{su}  + \delta (\FF_{ut}(x_{\cdot}))_{su} -  \delta \FF_{sut}(x_s) \\
 & = \delta (F_{ut}(x_{\cdot}))_{su} - \nabla F_{ut}(x_s)  F_{su}(x_s)  + \delta (\FF_{ut}(x_{\cdot}))_{su}  \\
  & = F_{ut}(x)_{su}^{\sharp} +  \nabla \FF_{ut}(x_s) F_{su}(x_s)   + \FF_{ut}(x)_{su}^{\sharp} . 
\end{align*}
We get from Lemma \ref{Composition}, provided $h<1$
$$
|F_{ut}(x)_{su}^{\sharp}| + |\FF_{ut}(x)_{su}^{\sharp}| \leq w_{\bF}(s,t)^{1/p} ( w_{\bF}(s,t)^{1/p} w_x(s,t)^{1/p} + w_{x^{\sharp}}(s,t)^{2/p}) 
$$
and clearly
$$
|\nabla \FF_{ut}(x_s)  F_{su}(x_s)  | \leq w_{\bF}(s,t)^{3/p} .
$$

From the sewing lemma there exists a constant $C$ such that
$$
|x_{st}^{\natural}| \leq C \big(   w_{\bF}(s,t)^{2/p} w_x(s,t)^{1/p} + w_{\bF}(s,t)^{1/p} w_{x^{\sharp}}(s,t)^{2/p}  + w_{\bF}(s,t)^{3/p} \big)
$$
From equations \eqref{ExpansionEquation} and \eqref{def: x sharp} we have
$$
|x_{st}| \leq w_{\bF}(s,t)^{1/p} + w_{\bF}(s,t)^{2/p} + w_{x^{\natural}}(s,t)^{3/p}  \qquad |x^{\sharp}_{st}| \leq w_{\bF}(s,t)^{2/p} + w_{x^{\natural}}(s,t)^{3/p}
$$
and consequently
$$
|x_{st}^{\natural}| \leq w_{x^{\natural}}(s,t)^{3/p} \leq  C \big(    w_{\bF}(s,t)^{1/p} w_{x^{\natural}}(s,t)^{3/p}  + w_{\bF}(s,t)^{3/p} \big) .
$$
If now $s,t$ is such that $C w_{\bF}(s,t)^{1/p} \leq \frac12$ we get
$$
w_{x^{\natural}}(s,t)^{3/p} \leq  C  w_{\bF}(s,t)^{3/p}  
$$
which gives
$$
|x_{st}| \leq C w_{\bF}(s,t)^{1/p} , \quad |x_{st}^{\sharp}| \leq C w_{\bF}(s,t)^{2/p} .
$$

\end{proof}

The above bound translates now to global estimates on the solution itself in the following way. 

\begin{lemma} \label{lem:APriori}
Assume now that $\bF$ is an $\alpha$-rough driver with $\alpha = \frac1p$. Then we have, for $h > 0$ small enough depending on $\bF$,
\begin{equation}
\label{eq: local alpha norm x}
[x]_{\alpha,h} \leq C [\bF]_{\alpha,h}. 
\end{equation}
Moreover, we have the global estimate
\begin{equation}
\label{eq:a priori solution}
[x]_{\alpha} \leq  C  ([\bF]_{\alpha} \vee [\bF]_{\alpha}^{1/\alpha})
\end{equation}
for a constant $C > 0$ depending only on $\alpha$. 
\end{lemma}

\begin{proof}
Since $\bF$ is H\"older continuous we have $w_{\bF}(s,t) \leq [\bF]_{\alpha,h}^{p} |t-s|$ for all $|t-s| \leq h$. Choose now $h$ such that $h^{\alpha} [\bF]_{\alpha,h} C \leq \frac12$ where $C$ is as in Proposition \ref{prop:APriori}. For $|t-s| \leq h$ we have
$$
|x_{st}| \leq C w_{\bF}(s,t)^{\alpha} \leq C [\bF]_{\alpha,h} |t-s|^{\alpha}, 
$$
from which \eqref{eq: local alpha norm x} follows.

From \eqref{eq:LocalToGlobal} we get, choosing now $h \simeq [\bF]_{\alpha}^{-1/\alpha}$, $h^{\alpha - 1} \simeq  [\bF]_{\alpha}^{(1 - \alpha)/\alpha}$
$$
[x]_{\alpha} \leq [x]_{\alpha,h} ( 1 \vee 2 h^{\alpha - 1}) \leq C [\bF]_{\alpha} (1 \vee [\bF]_{\alpha}^{(1 - \alpha)/\alpha})
$$
for some universal constant $C$ depending only on $\alpha$. 
\end{proof}

\subsection{A priori contractive estimates}
Let $p < 3$, and assume $\bF$, $\bG$ are two $p$-rough drivers. We take two solutions $x$ and $y$ of equation \eqref{eq: non linear equation} in the sense of Definition \eqref{def:RDSolution}, with initial conditions $x_0$ and $y_0$ and driven by $\bF$ and $\bG$ respectively. 

To illustrate the ideas of this section, we give the following remark.
\begin{remark}
Assume that $F := \int_{0}^{t}f_r(x)dr$, $G := \int_{0}^{t}g_r(x)dr$, $x$ and $y$ are smooth in time, so that we can write
\begin{align*}
|x_t & - y_t| \leq |x_0 - y_0| +  \left| \int_0^t f_r(x_r) - g_r(y_r) dr \right| \leq |x_0 - y_0| + \int_0^t |f_r(x_r) - f_r(y_r) |dr + \int_0^t \|f_r - g_r\|_{C_b} dr\\
&  \leq |x_0 - y_0| + \int_0^t \|\nabla f_r \|_{C_b} \|x_r - y_r \|dr + \int_0^t \|f_r - g_r\|_{C_b} dr \leq e^{\int_0^t \|\nabla f_r\|_{C_b} dr} ( |x_0 - y_0| + \int_0^t \|f_r - g_r \|_{C_b} dr) 
\end{align*}
where we have used Gronwall's inequality in the last step. The purpose of this subsection is to replicate these estimates also for the rough case. The steps are similar to the previous subsection, except we compare two solutions.	
\end{remark}
We start by writing 
\begin{align*}
\delta x_{st} - \delta y_{st} & = F_{st}(x_s) - G_{st}(y_s) + x_{st}^{\sharp} - y_{st}^{\sharp} = F_{st}(x_s) - F_{st}(y_s) + F_{st}(y_s) - G_{st}(y_s) + x_{st}^{\sharp} - y_{st}^{\sharp} .
\end{align*}
Let $z := x-y$ and $z^{\sharp} := x^{\sharp} - y^{\sharp}$ so that the above gives the estimate
\begin{align}
| \delta z_{st} |  &  \leq  w_{\bF}(s,t)^{1/p} |z_s|   + w_{\bF - \bG}(s,t)^{1/p}   +  w_{z^{\sharp}}(t,s)^{1/p} \label{DifferenceWithSup}.
\end{align}
We begin with the analogue of Lemma \ref{Composition} that allows us to estimate nonlinearities of the remainders.
\begin{lemma}
Let $f, g \in C_b^3$. Then using the notation as in Lemma \ref{Composition} we have the estimate
\begin{align} 
|f(x)^{\sharp}_{st} - g(y)^{\sharp}_{st}| &  \leq \| f - g\|_{C_b^2}  w_{\bF}(s,t)^{2/p}  + \| g\|_{C_b^3} |z_s| (   w_{\bF}(s,t)^{2/p} +  w_{\bG}(s,t)^{2/p} )  \notag \\
 & + \| g\|_{C_b^2} ( w_{z}(s,t)^{1/p} w_{\bF}(s,t)^{1/p} + w_{\bG}(s,t)^{1/p} w_{\bF - \bG}(s,t)^{1/p} +  w_{z^{\sharp}}(s,t)^{2/p} ) \label{eq:ContractionCompositionRemainder} .
\end{align}
\end{lemma}

\begin{proof}
We write
\begin{align*}
f(x)_{st}^{\sharp} - g(y)_{st}^{\sharp} & = [\nabla^2 f]^{2,x}_{st} \delta x_{st}F_{st}(x_s) - [\nabla^2 g]^{2,y}_{st} \delta y_{st} G_{st}(y_s) + [\nabla f]^{1,x}_{st} x_{st}^{\sharp} - [\nabla g]^{1,y}_{st} y_{st}^{\sharp} .
\end{align*}
The first two terms above can be written
\begin{align*}
[\nabla^2 f]^{2,x}_{st} \delta x_{st}F_{st}(x_s) & - [\nabla^2 g]^{2,y}_{st} \delta y_{st} G_{st}(y_s)   = ( [\nabla^2 f]^{2,x}_{st} - [\nabla^2 g]^{2,x}_{st}) \delta x_{st}F_{st}(x_s)   \\
& + ( [\nabla^2 g]^{2,x}_{st} - [\nabla^2 g]^{2,y}_{st}) \delta x_{st}F_{st}(x_s)   +  [\nabla^2 g]^{2,y}_{st}( \delta x_{st} - \delta y_{st}) F_{st}(x_s) \\
& + [\nabla^2 g]^{2,y}_{st}\delta y_{st} (F_{st}(x_s) - G_{st}(x_s))  + [\nabla^2 g]^{2,y}_{st}\delta y_{st} (G_{st}(x_s) - G_{st}(y_s)).
\end{align*}
Which gives the bound
\begin{align*}
|[\nabla^2 f]^{2,x}_{st} \delta x_{st}F_{st}(x_s) & - [\nabla^2 g]^{2,y}_{st} \delta y_{st} G_{st}(y_s) |  \leq   \|\nabla^2 f - \nabla^2 g\|_{C_b}  w_{\bF}(s,t)^{2/p}  +  \|\nabla^3 g\|_{C_b} |z_s|   w_{\bF}(s,t)^{2/p}  \\
& +  \|\nabla^2 g\|_{C} w_{z}(s,t)^{1/p} w_{\bF}(s,t)^{1/p}   +\|\nabla^2 g\|_{C_b} w_{\bG}(s,t)^{1/p} w_{\bF - \bG}(s,t)^{1/p} \\
 & +  \| \nabla^2 g\|_{C_b} w_{\bG}(s,t)^{1/p} w_{\bG}(s,t)^{1/p} |z_s|.
\end{align*}
Now write 
\begin{align*}
 [\nabla f]^{1,x}_{st} x_{st}^{\sharp} &  - [\nabla g]^{1,y}_{st} y_{st}^{\sharp} =  ([\nabla f]^{1,x}_{st} - [\nabla g]^{1,x}_{st}) x_{st}^{\sharp}  + ([\nabla g]^{1,x}_{st} - [\nabla g]^{1,y}_{st}) x_{st}^{\sharp} + [\nabla g]^{1,y}_{st} (x_{st}^{\sharp} - y_{st}^{\sharp}) .
\end{align*}
We see that
\begin{align*}
 |([\nabla f]^{1,x}_{st} - [\nabla g]^{1,x}_{st}) x_{st}^{\sharp}| \lesssim \| \nabla f - \nabla g\|_{C_b} w_{\bF}(s,t)^{2/p}, \quad \textrm{ and }  \quad
 | ( [\nabla g]^{1,y}_{st} (x_{st}^{\sharp} - y_{st}^{\sharp}) ) |  \lesssim \|\nabla g\|_{C_b} w_{z^{\sharp}}(s,t)^{2/p}
\end{align*}
which gives \eqref{eq:ContractionCompositionRemainder}.
\end{proof}


\begin{proposition} \label{ContractionComposition}
Assume that $\bF$ and $\bG$ are $\alpha$-rough drivers with $\alpha = \frac1p$. Then there exists universal constants $C$ such that
\begin{equation} \label{eq:GronwallEstimate} 
\sup_{r \in [0,T]} |x_r - y_r| \leq C (|x_0 - y_0| +   \la \bF- \bG \ra_{p}) e^{C N(w_{\bF},[0,T])}.
\end{equation}
Moreover,
\begin{align} \label{eq:GronwallEstimateHolder} 
\la x - y \ra_{p,[s,t]} 
\leq &C e^{C N(w_{\bF},[0,T])}(|x_0 - y_0| +  \la \bF- \bG \ra_{p} )\\
& \cdot  \big[ ( \la \bF \ra_{p,[s,t]} + \la \bG \ra_{p,[s,t]})  (  1 + \la \bF \ra_{p} + \la \bG \ra_{p} )^2  \big] \nonumber
\end{align}
\begin{align} \label{eq:GronwallEstimateHolderSharp} 
\la x^{\sharp} - y^{\sharp} \ra_{\frac{p}{2},[s,t]} 
\leq & C e^{C N(w_{\bF},[0,T])}(|x_0 - y_0| +  \la \bF- \bG \ra_{p} )\\
& \cdot  \big[ ( \la \bF \ra_{p,[s,t]} + \la \bG \ra_{p,[s,t]})  (  1 + \la \bF \ra_{p} + \la \bG \ra_{p} )^2  \big] \nonumber
\end{align}
for all $s,t$ such that $C  ( \la \bF \ra_{p,[s,t]} + \la \bG \ra_{p,[s,t]})   \leq 1$. In particular, we have uniqueness for equation \eqref{eq: non linear equation} and the solution is continuous w.r.t. the initial condition. 

\end{proposition}

\begin{proof}
Using Chen's relation we get
\begin{align*}
\delta (x-y)_{sut}^{\natural} & = F_{ut}(x)_{su}^{\sharp} - G_{ut}(y)_{su}^{\sharp} + \nabla \FF_{ut}(x_s)  F_{su}(x_s) - \nabla \mathbb{G}_{ut}(x_s)  G_{su}(y_s)  + \FF_{ut}(x)_{su}^{\sharp} - \mathbb{G}_{ut}(x)_{su}^{\sharp}. 
\end{align*}
Replacing $f = F_{ut}$ and $g = G_{ut}$ in \eqref{eq:ContractionCompositionRemainder} we get
\begin{align*}
|F_{ut}(x)_{su}^{\sharp} - G_{ut}(y)_{su}^{\sharp}|  &  \lesssim w_{\bF - \bG}(s,t)^{1/p}  w_{\bF}(s,t)^{2/p}  + w_{\bG}(s,t)^{1/p} |z_s| (   w_{\bF}(s,t)^{2/p} +  w_{\bG}(s,t)^{2/p} ) \\
 & + w_{\bG}(s,t)^{1/p} ( w_{z}(s,t)^{1/p} w_{\bF}(s,t)^{1/p} + w_{\bG}(s,t)^{1/p} w_{\bF - \bG}(s,t)^{1/p} +  w_{z^{\sharp}}(s,t)^{2/p} ).
\end{align*}
Replacing $f = \FF_{ut}$ and $g = \GG_{ut}$ in \eqref{eq:ContractionCompositionRemainder} we get
\begin{align*}
|\FF_{ut}(x)_{su}^{\sharp} - \GG_{ut}(y)_{su}^{\sharp}|  &  \lesssim  w_{\bF - \bG}(s,t)^{2/p}  w_{\bF}(s,t)^{2/p} 
  + w_{\bG}(s,t)^{2/p} |z_s| (  w_{\bF}(s,t)^{2/p} +  w_{\bG}(s,t)^{2/p} ) \\
 & + w_{\bG}(s,t)^{2/p} ( w_{z}(s,t)^{1/p} w_{\bF}(s,t)^{1/p} + w_{\bG}(s,t)^{1/p} w_{\bF - \bG}(s,t)^{1/p} +  w_{z^{\sharp}}(s,t)^{2/p} ) .
\end{align*}
Use also the estimate
$$
|z_{st}^{\sharp}|  \leq |z_s| w_{\bF}(s,t)^{2/p} + w_{\bF - \bG}(s,t)^{2/p} + w_{z^{\natural}}(s,t)^{3/p} .
$$
Let now $s,t$ be such that $w_{\bG}(s,t)^{1/p} ,w_{\bF}(s,t)^{1/p} \leq \frac{C}{2}$ which gives
\begin{align*}
|\delta z_{sut}^{\natural}|   & \lesssim  w_{\bF - \bG}(s,t)^{1/p}  w_{\bF}(s,t)^{2/p}  +  w_{\bF - \bG}(s,t)^{2/p}  w_{\bF}(s,t)^{1/p}  \\
  & + w_{\bG}(s,t)^{1/p} |z_s| (   w_{\bF}(s,t)^{2/p} +  w_{\bG}(s,t)^{2/p} ) \\
 & + w_{\bG}(s,t)^{1/p} \big( w_{z}(s,t)^{1/p} w_{\bF}(s,t)^{1/p} + w_{\bG}(s,t)^{1/p} w_{\bF - \bG}(s,t)^{1/p} \\
  & +  |z_s| w_{\bF}(s,t)^{2/p} + w_{\bF - \bG}(s,t)^{2/p} + w_{z^{\natural}}(s,t)^{3/p}.
\end{align*}
From \eqref{APrioriSewingFirst} and \eqref{APrioriSewing} we get that there exists a universal constant $C$ such that 
\begin{align*}
w_{z^{\natural}}(s,t)^{3/p}  \leq  & C \big( w_{\bF - \bG}(s,t)^{1/p}  w_{\bF}(s,t)^{2/p}  +  w_{\bF - \bG}(s,t)^{2/p}  w_{\bF}(s,t)^{1/p}  \\
  & +  \sup_{r \in [s,t]} |z_r| (  w_{\bG}(s,t)^{1/p}  w_{\bF}(s,t)^{2/p} +  w_{\bG}(s,t)^{3/p} ) \\
 & + w_{\bG}(s,t)^{1/p} \big( w_{z}(s,t)^{1/p} w_{\bF}(s,t)^{1/p} + w_{\bG}(s,t)^{1/p} w_{\bF - \bG}(s,t)^{1/p} \\
  &  + w_{\bF - \bG}(s,t)^{2/p} + w_{z^{\natural}}(s,t)^{3/p}  \big).
\end{align*}
Choose now $s,t$ such that $w_{\bG}(s,t)^{1/p} \leq \frac{C}{2} \wedge 1$, so that
\begin{align*}
w_{z^{\natural}}(s,t)^{3/p}  \leq  &  2 C \big( w_{\bF - \bG}(s,t)^{1/p}  w_{\bF}(s,t)^{2/p}  +  w_{\bF - \bG}(s,t)^{2/p}  w_{\bF}(s,t)^{1/p}  \\
  & +\sup_{r \in [s,t]} |z_r| (  w_{\bG}(s,t)^{1/p}   w_{\bF}(s,t)^{2/p} +  w_{\bG}(s,t)^{3/p} ) \\
 & + w_{\bG}(s,t)^{1/p} \big( w_{z}(s,t)^{1/p} w_{\bF}(s,t)^{1/p} + w_{\bG}(s,t)^{1/p} w_{\bF - \bG}(s,t)^{1/p}   + w_{\bF - \bG}(s,t)^{2/p}   \big).
\end{align*}
For the solution we have
\begin{align}
|\delta z_{st}| & \leq w_{\bF}(s,t)^{1/p}|z_s|   + w_{\bF}(s,t)^{2/p}|z_s|  + w_{z^{\natural}}(s,t)^{3/p}  \notag \\
 & \leq C  \big( w_{\bF}(s,t)^{1/p}  \sup_{r \in [s,t]} |z_r| +  w_{\bF - \bG}(s,t)^{1/p} w_{\bF}(s,t)^{2/p} \notag \\
  &   + w_{\bF - \bG}(s,t)^{2/p}w_{\bF}(s,t)^{1/p}  + w_{\bG}(s,t)^{1/p}  w_{z}(s,t)^{1/p} w_{\bF}(s,t)^{1/p} \big). \label{eq:DifferenceInProof}
\end{align}
Let now $(s,t)$ be such that $w_{\bF}(s,t)^{1/p} , w_{\bG}(s,t)^{1/p} \leq \frac{C}{2}$ we get
\begin{align*}
|\delta z_{st}| \leq \omega_z(s,t)^{1/p} &  \leq C  \big( w_{\bF}(s,t)^{1/p} \sup_{r \in [s,t]} |z_r| +  w_*(s,t)^{3/p} \big),
\end{align*}
where $w_*(s,t) = w_{\bF - \bG}(s,t)^{1/3}w_{\bF}(s,t)^{2/3} + w_{\bF - \bG}(s,t)^{2/3}w_{\bF}(s,t)^{1/3}$.
From the rough Gronwall lemma, \cite[Lemma 2.11]{DeGuHoTi16}, we get 
$$
\sup_{r \in [s,t]} |z_r| \leq C \exp\{ C w_{\bF}(s,t)  \} ( |z_s|  +  w_*(s,t)^{3/p} ) ,
$$
and we notice that this holds for all subintervals $[s,t]$, i.e. no smallness assumption.
Now, choose the finest partition $\tau_k$ of $[s,t]$ such that $w_{\bF}(\tau_k, \tau_{k+1}) = 1$. We have
$$
\sup_{r \in [s,\tau_1]} |z_r|  \leq C(|z_s| + w_*(s,\tau_1)^{3/p})e^C, 
$$
and on $[\tau_1, \tau_2]$ we get
\begin{align*}
\sup_{r \in [\tau_1,\tau_2]} |z_r| &  \leq C(|z_{\tau_1}| + w_*(\tau_1,\tau_2)^{3/p})e^C \leq C( C(|z_s| + w_*(s,\tau_1)^{3/p})e^C + w_*(\tau_1,\tau_2)^{3/p})e^C \\
 & \leq C^2 (|z_s| + w_*(s,\tau_2)^{3/p})e^{2C}
\end{align*}
provided $C>1$. An easy induction shows that
\begin{align*}
\sup_{r \in [\tau_n,\tau_{n+1}]} |z_r|  &  \leq C^n (|z_s| + w_*(s,\tau_{n+1})^{3/p})e^{nC}  \leq  (|z_s| + w_*(s,\tau_{n+1})^{3/p})e^{n ( C + ln(C))}.
\end{align*}
By definition of the greedy partition \eqref{eq: N beta} we get
\begin{equation} \label{eq:GronwallEstimateSharper}
\sup_{r \in [s,t]} |z_r| \leq C  (|z_s| + w_*(s,t)^{3/p}) e^{\tilde{C} N(w_{\bF},[s,t])} .
\end{equation}
Letting $s= 0$ and using the bound $w_*(0,T) \leq w_{\bF - \bG}(0,T) \leq [\bF - \bG]_{\alpha}$ this shows \eqref{eq:GronwallEstimate}. To see \eqref{eq:GronwallEstimateHolder} we plug the above into \eqref{DifferenceWithSup} to get
%
%
%
\begin{align*}
|\delta z_{st}| &  \leq  w_{\bF}(s,t)^{1/p}|z_s| + w_{\bF - \bG}(s,t)^{1/p}  + w_{z^{\natural}}(s,t)^{3/p} \\
 & \leq w_{\bF}(s,t)^{1/p}|z_s| + w_{\bF - \bG}(s,t)^{1/p}  +  2 C \big( w_{\bF - \bG}(s,t)^{1/p}  w_{\bF}(s,t)^{2/p}  +  w_{\bF - \bG}(s,t)^{2/p}  w_{\bF}(s,t)^{1/p}  \\
  & +\sup_{r \in [s,t]} |z_r| (  w_{\bG}(s,t)^{1/p}   w_{\bF}(s,t)^{2/p} +  w_{\bG}(s,t)^{3/p} ) \\
 & + w_{\bG}(s,t)^{1/p} \big( w_{z}(s,t)^{1/p} w_{\bF}(s,t)^{1/p} + w_{\bG}(s,t)^{1/p} w_{\bF - \bG}(s,t)^{1/p}   + w_{\bF - \bG}(s,t)^{2/p}   \big) \\
 & \leq ( w_{\bF}(s,t)^{1/p} + w_{\bG}(s,t)^{1/p}) \sup_{r \in [s,t]} |z_r|  + w_{\bF - \bG}(s,t)^{1/p}  + w_{\bG}(s,t)^{1/p} w_{\bF}(s,t)^{1/p}  w_{z}(s,t)^{1/p}  \\
 \leq & ( w_{\bF}(s,t)^{1/p} + w_{\bG}(s,t)^{1/p}) \sup_{r \in [s,t]} |z_r|  + w_{\bF - \bG}(s,t)^{1/p},
 \end{align*}
using $w_{\bF}(s,t)^{1/p},w_{\bG}(s,t)^{1/p} \leq 1/2$ in the last step. Using \eqref{eq:GronwallEstimateSharper} gives
\begin{align}
|\delta z_{st}| &  
\leq   ( w_{\bF}(s,t)^{1/p} + w_{\bG}(s,t)^{1/p}) C w_*(0,t)^{3/p} e^{\tilde{C} N(w_{\bF},[0,t])}  + w_{\bF - \bG}(s,t)^{1/p}  \nonumber\\
 & \leq  \big[ \big(|z_0| +  w_{\bF - \bG}(0,T)^{1/p} w_{\bF}(0,T)^{2/p} +  w_{\bF - \bG}(0,T)^{2/p} w_{\bF}(0,T)^{1/p} \big)  ( w_{\bF}(s,t)^{1/p} + w_{\bG}(s,t)^{1/p})  \nonumber\\
 & +  w_{\bF - \bG}(s,t)^{1/p}  \big]C e^{\tilde{C} N(w_{\bF},[0,T])} . \label{eq: control bounds}
 \end{align}
This gives \eqref{eq:GronwallEstimateHolder}. The bound \eqref{eq:GronwallEstimateHolderSharp} is proved in a similar way. 
\end{proof}

\begin{corollary} \label{cor: Contraction Composition Global}
	Assume that $\bF$ and $\bG$ are $\alpha$-rough drivers with $\alpha = \frac1p$. Then there exists a universal constant $C$ such that,
	\begin{align*} 
	[x - y]_{\alpha} \leq &
	C e^{C N(w_{\bF},[0,T])}
	(|x_0 - y_0| +  [\bF- \bG]_{\alpha} )
	(  1 +  [\bF]_{\alpha} +  [\bG]_{\alpha}  )^2 \\
	&\cdot
	(( [\bF]_{\alpha} + [\bG]_{\alpha}) \vee ( [\bF]_{\alpha} + [\bG]_{\alpha})^{\frac{1}{\alpha}}).
	\end{align*}
\end{corollary}

\begin{proof}
	Use bounds on the form $w_{\bF}(s,t) \leq [\bF]_{\alpha,h}^{p} |t-s|$ for all $|t-s| \leq h$ in inequality \eqref{eq: control bounds}. This gives the H\"{o}lder estimate
	\begin{align*}
	[z]_{\alpha,h} &  \leq   C e^{C N(w_{\bF},[0,T])} \big[ \big( [\bF- \bG]_{\alpha} [\bF]^2_{\alpha}  +  [\bF- \bG]_{\alpha}^2  [\bF]_{\alpha}  \big)  ( [\bF]_{\alpha,h} + [\bG]_{\alpha,h})   +  [\bF- \bG]_{\alpha,h}  )  \big] \\
	& \leq   C e^{C N(w_{\bF},[0,T])}(|z_0| +  [\bF- \bG]_{\alpha} ) \big[ ( [\bF]_{\alpha,h} + [\bG]_{\alpha,h})  (  1 +  [\bF]_{\alpha} +  [\bG]_{\alpha}  )^2  \big]
	\end{align*}
	which holds when $h$ is such that $|t-s| \leq h$ we have $w_{\bF}(s,t)^{1/p},w_{\bG}(s,t)^{1/p} \lesssim 1$, in particular when $h^{\alpha}( [\bF]_{\alpha,h} + [\bG]_{\alpha,h}) \lesssim 1$.
	
	Let $C$ be the constant given by Proposition \ref{ContractionComposition} and set $h = C^{-\frac{1}{\alpha}} ([\bF]_{\alpha} + [\bG]_{\alpha})^{-\frac{1}{\alpha}}$. It follows by \eqref{eq:LocalToGlobal} and Proposition \ref{ContractionComposition} that (the value of $C$ changes in the following lines, but it only depends on $\alpha$)
	\begin{align*} 
	[x - y]_{\alpha} \leq & [x - y]_{\alpha,h} (1\vee 2h^{\alpha-1})\\
	\leq &
	 C e^{C N(w_{\bF},[0,T])}
	 (|x_0 - y_0| +  [\bF- \bG]_{\alpha} )
	 (  1 +  [\bF]_{\alpha} +  [\bG]_{\alpha}  )^2 
	 ( [\bF]_{\alpha,h} + [\bG]_{\alpha,h}) \\
	 &\cdot (1\vee ( [\bF]_{\alpha} + [\bG]_{\alpha})^{-1+\frac{1}{\alpha}}).
	\end{align*}
	This concludes the proof.
\end{proof}

\subsection{Well-posedness of nonlinear RDEs}
Since uniqueness of equation \eqref{eq: non linear equation} follows from Proposition \ref{ContractionComposition}, it is only left to prove existence of a solution. We do so by using a Picard iteration.

\begin{theorem} \label{thm: wellposedness}
Let $\bF$ be a $p$-variation rough driver. There exists a unique solution $x$ of equation \eqref{eq: non linear equation}, in the sense of Definition \ref{def:RDSolution}, with initial condition $\xi \in \R^d$.
\end{theorem}
\begin{proof}
	Uniqueness is given by Proposition \ref{ContractionComposition}. We study now existence. Define $x^0_t = \xi$, $x_t^1 = F_{0t}(\xi)$ and 
$$
a_{st}^1 := F_{st}(x_s^1) + \FF_{st}(x_s^0,x_s^{1}),
$$
which gives
\begin{align*}
\delta a_{sut}^1 & = - \delta F_{ut}(x^1)_{su} - \delta \FF_{ut}(x^1,x^{0})_{su} + \delta \FF_{sut}(x_s^0,x_s^1) \\
 & = - [\nabla F_{ut}]_{su}^{1,x^1} \delta x_{su}^1 - \delta \FF_{ut}(x^1,x^{0})_{su} +   F_{su}(x_s^{0}) \otimes \nabla F_{ut}(x_s^1)  \\
  & = - [\nabla^2 F_{ut}]_{su}^{2,x^1} (\delta x_{su}^1  \otimes F_{su}(x_s^{0}) ) - \delta \FF_{ut}(x^1,x^{0})_{su} \\
    & = - [\nabla^2 F_{ut}]_{su}^{2,x^1} ( F_{su}(x_s^0)  \otimes F_{su}(x_s^{0}) ) - \delta \FF_{ut}(x^1,x^{0})_{su} \\
     & \leq w_{\bF}(s,t)^{3/p} + \| \nabla \FF_{ut} \|_{C_b} w_{\bF}(s,t)^{1/p} \leq 2 w_{\bF}(s,t)^{3/p}.
\end{align*}
Consequently, there exists a pair $(x^2,x^{2, \natural})$ such that 
$$
\delta x_{st}^2 = F_{st}(x_s^1) + \FF_{st}(x_s^1,x_s^{0}) + x_{st}^{2, \natural}
$$
and we have $|x_{st}^{2, \natural}| \leq C w_{\bF}(s,t)^{3/p}$ for some universal constant $C$.

We prove inductively that there exists universal constants $C$ and $h$ such that for $w_{\bF}(s,t) \leq h$ we have $|x_{st}^{n, \natural}| \leq Cw_{\bF}(s,t)^{3/p}$ and $|\delta x_{st}^{n}| \leq 2 w_{\bF}(s,t)^{1/p} + Cw_{\bF}(s,t)^{3/p}$.

Given $x^{n-1}$ and $x^{n}$ we let
$$
a_{st}^n := F_{st}(x_s^n) + \FF_{st}(x_s^n,x_s^{n-1}) .
$$
We then get
\begin{align*}
\delta a_{sut}^n & = - \delta F_{ut}(x^n)_{su} - \delta \FF_{ut}(x^n,x^{n-1})_{su} +   F_{su}(x_s^{n-1}) \otimes \nabla F_{ut}(x_s^n)   \\
 & = - [\nabla F_{ut}]_{su}^{1,x^n} \delta x_{su}^n - \delta \FF_{ut}(x^n,x^{n-1})_{su} +  \nabla F_{ut}(x_s^n) F_{su}(x_s^{n-1}) \\
  & = - [\nabla F_{ut}]_{su}^{1,x^n} (F_{su}(x_s^{n-1}) )- [\nabla F_{ut}]_{su}^{1,x^n} \FF_{su}(x_s^{n-1},x_s^{n-2}) - [\nabla F_{ut}]_{su}^{1,x^n} x_{st}^{n, \natural}\FF_{su}(x_s^{n-1},x_s^{n-2})  \\
   & -\delta \FF_{ut}(x^n,x^{n-1})_{su} + \nabla F_{ut}(x_s^n) F_{su}(x_s^{n-1}) \\
 & = - [\nabla F_{ut}]_{su}^{2,x^n} ( \delta x_{su}^{n-1} \otimes  F_{su}(x_s^{n-1})) - [\nabla F_{ut}]_{su}^{1,x^n} \FF_{su}(x_s^{n-1},x_s^{n-2}) \\
  & - [\nabla F_{ut}]_{su}^{1,x^n} x_{st}^{n, \natural}\FF_{su}(x_s^{n-1},x_s^{n-2})  -\delta \FF_{ut}(x^n,x^{n-1})_{su} .
  \end{align*}
which gives
\begin{align*}
|\delta a_{sut}^n|  & \leq  w_{\bF}(s,t)^{2/p}(2 w_{\bF}(s,t)^{1/p} + C w_{\bF}(s,t)^{3/p})   + Cw_{\bF}(s,t)^{5/p} + w_{\bF}(s,t)^{3/p} + C w_{\bF}(s,t)^{5/p} \\
 &  + 2 w_{\bF}(s,t)^{2/p}(2 w_{\bF}(s,t)^{1/p} + C w_{\bF}(s,t)^{3/p}) \\
 & = 7 w_{\bF}(s,t)^{3/p} + 5C w_{\bF}(s,t)^{5/p} \leq 8 w_{\bF}(s,t)^{3/p}
\end{align*}
provided $h$ is such that $5C w_{\bF}(s,t)^{2/p} \leq 1$. This gives that there exists $x^{n+1},x^{n+1, \natural}$ such that 
\begin{equation} \label{PicardEquation}
\delta x_{st}^{n+1} = F_{st}(x_s^n) + \FF_{st}(x_s^n,x_s^{n-1}) + x_{st}^{n+1, \natural}, \quad |x_{st}^{n+1, \natural}| \leq C_p 8  w_{\bF}(s,t)^{3/p}   
\end{equation}
so $C \geq C_p 8$ will do. Provided $h$ is such that $w_{\bF}(s,t)^{1/p} \leq 1$ we also get
$$
|\delta x_{st}^n| \leq w_{\bF}(s,t)^{1/p} + w_{\bF}(s,t)^{2/p} + C w_{\bF}(s,t)^{3/p} \leq 2 w_{\bF}(s,t)^{1/p}  + C w_{\bF}(s,t)^{3/p}
$$
which proves the induction hypothesis. 

From Arzel\`{a}-Ascoli we get that there exists a subsequence $x^{n_k}$ converging in $C([0,T];\R^d)$ to some element $x$. Clearly we get 
$$
\sup_{s, t} \left( |F_{st}(x_s^{n_k}) - F_{st}(x_s) | \vee |\FF_{st}(x_s^{n_k}) - \FF_{st}(x_s) | \right) \rightarrow 0 .
$$ 
Since all the terms of \eqref{PicardEquation} (or rather, the one with $n$ replaced by $n_k$) converges, we get that also $x^{n_k, \natural}_{st}$ must converge to a limit denoted $x^{\natural}_{st}$. Then $x$ and $x^{\natural}$ satisfies \eqref{ExpansionEquation} and from the uniform bounds on $x^{n_k, \natural}$ we see that $x$ indeed is a solution. 
\end{proof}

\section{Rough non-linearities} \label{sec: rough non-linearities}

In this section we show how to construct the rough drivers that are used for solving the McKean-Vlasov equation \eqref{MeanFieldRoughSDE}. We start by constructing rough drivers corresponding to It\^{o} theory, i.e. given a vector field $\sigma$ and a Brownian motion $W$, we want to define
$$
W^{\sigma}_{st}(x) = \int_s^t \sigma_r(x) dW_r , \qquad \WW^{\sigma}_{st}(x,y) = \int_s^t W_{sr}^{\sigma}(x) \nabla \sigma_r(y) dW_r,
$$
where the latter integration is in the sense of It\^{o}. As the following example demonstrates, it is not possible to simply integrate a function $\sigma \in C([0,T];C_b^3(\R^d;\R^d))$ to produce a rough driver.

\begin{example}
Let $d=1$ and $\sigma_r(x) = \sin(rx)$, then the mapping $x \mapsto W_{st}^{\sigma}(x)$ is $P$-a.s. \emph{unbounded} as $x \rightarrow \infty$. Indeed, let $s=0$ and $t = 1$ and $x = 2 \pi n$ for $n \in \mathbb{N}$, then $\{ W_{0 1}^{\sigma}(2 \pi n)\}_{n \in \mathbb{N}}$ is an i.i.d. Gaussian sequence, which $P$-a.s. diverges.
\end{example}
The above example shows that we need some decay on our vector fields as $|x| \rightarrow \infty$. We choose to assume that $\sigma$ belongs to a Sobolev space $H^k(\R^d;\R^d)$ where $k$ is large enough to use Sobolev embedding to show that $\bW^{\sigma}$ is a rough driver. The reason for this choice is the relatively simple and well established theory of It\^{o} integration that is available for Hilbert spaces. We conjecture that this regularity can be significantly lowered (e.g. with decay as in \cite[Corollary 9]{BR}) and leave this for future investigation. 

Let $d, m \in \N$ be fixed and let $\bZ \in \mathscr{C}_g^{\bar\alpha}([0,T], \R^m)$, for $\bar \alpha \in (\frac{1}{3}, \frac{1}{2})$. In this section we assume the following
\begin{assumption}
	\label{asm: building rough drivers}
	Let $k\in \N\cup\{0\}$, and $\alpha \in (\frac{1}{3},\bar\alpha)$,
	\begin{enumerate}[label=(\roman*), ref=\ref{asm: building rough drivers} (\roman*)]
		\item \label{asm: building rough drivers: beta} Let $(\beta, \beta^\prime) \in \mathscr{D}_{Z}^{2\alpha}([0,T];H^k)$, as in Section \ref{sec:Notations}.
		\item \label{asm: building rough drivers: sigma} Let $\sigma: [0,T] \rightarrow \mathcal{L}(\R^d; H^k)$ be a continuous function, such that
		\begin{equation*}
		\| \sigma\|_{L_t^\infty \mathcal{L}(\R^d;H^k)} = \sup_{t\in[0,T]} \| \sigma_t\|_{\mathcal{L}(\R^d;H^k)} < \infty.
		\end{equation*}
		\item \label{asm: building rough drivers: sigma p var}Let $p = \alpha^{-1}$, then
		\begin{equation*}
		\la \sigma \ra_{p;\mathcal{L}(\R^d;H^k)} < +\infty.
		\end{equation*}
	\end{enumerate}
\end{assumption}

To simplify the following discussion, we introduce the convenient notation
\begin{equation} \label{LConstant}
L = L(\sigma, \beta, \bZ) : =   \left( 1 + \|\sigma\|_{L_t^\infty \mathcal{L}(\R^d;H^k)} +  \|(\beta, \beta')\|_{Z,\alpha;H^k} \right)(1 + [\bZ]_{\bar{\alpha}}\vee [\bZ]_{\bar{\alpha}}^{\frac{1}{2}}).
\end{equation}

\subsection{Construction of the rough driver}
%

\subsubsection{It\^{o} theory} 
\label{sec:Ito theory}
Let $(\Omega, \mathcal{F}, (\mathcal{F}_t)_{t \in [0,T]}, P)$ be a filtered probability space and let $W$ be a $d$-dimensional Wiener process on it. We assume that $\sigma$ satisfies Assumption \ref{asm: building rough drivers: sigma}, for $k>3+\frac{d}{2}$. 
We define, for $0\leq s \leq t \leq T$,
\begin{equation}
\label{def: W sigma}
W_{t}^{\sigma} := \int_0^t \sigma_r dW_r \in M^2_T(H^k), 
\qquad W_{st}^{\sigma} := W_{t}^{\sigma} - W_{s}^{\sigma},
\end{equation}
where the integral is defined in the sense of It\^{o} on Hilbert spaces, see \cite[Section 2]{MR3410409}. Thanks to Burkholder-Davis-Gundy (BDG) inequality for Hilbert spaces, \cite[Theorem 2.4.7]{MR3410409}, we have for all $\rho\geq 1$ and $0\leq s \leq t \leq T$,
\begin{equation}
\label{eq: bdg first order}
E\sup_{r\in[s,t]}\| W_{sr}^\sigma \|_{H^k}^\rho 
\leq C_\rho \left( \int_{s}^{t}\| \sigma_r\|_{\mathcal{L}(\R^d;H^k)}^{2}dr\right)^{\frac{\rho}{2}} 
\leq C_\rho \| \sigma_t\|_{L_t^\infty \mathcal{L}(\R^d;H^k)}^\rho |t-s|^{\frac{\rho}{2}}.
\end{equation}
We consider now the time-continuous stochastic process,
\begin{equation*}
\left( W_{\cdot}^{\sigma} \otimes \nabla\right) \sigma_{\cdot}: [0,T] \times \Omega \to \mathcal{L}(\R^d;H^{k} \otimes H^{k-1}),
\end{equation*}
with Hilbert-Schmidt norm \eqref{def: hilberth schmidt} bounded as $\| \left( W_{t}^{\sigma} \otimes \nabla\right) \sigma_{t} \|_{ \mathcal{L}(\R^d;H^{k} \otimes H^{k-1})} \leq \| W_{t}^{\sigma} \|_{ H^{k}} \| \sigma_{t} \|_{ \mathcal{L}(\R^d;H^{k})} $, for all $t\in [0,T]$.
Using again It\^{o} theory on Hilbert spaces, we have that $\int_0^t  \left(W_{r}^{\sigma} \otimes \nabla\right) \sigma_r dW_r \in M^2_T(H^k\otimes H^{k-1})$ and we set, for $0\leq s \leq t \leq T$,
\begin{equation}
\label{def: bW sigma}
\WW_{st}^{\sigma} := \int_s^t  \left(W_{r}^{\sigma} \otimes \nabla\right) \sigma_r dW_r - \left(W_{s}^{\sigma} \otimes \nabla\right) W_{st}^{\sigma}
:\Omega \to H^k\otimes H^{k-1}.
\end{equation}
Applying again BDG inequality and inequality \eqref{eq: bdg first order}, we have for all $\rho\geq 1$ and $0\leq s \leq t \leq T$,
\begin{equation}
\label{eq: bdg second order}
E\sup_{r\in[s,t]}\| \WW_{sr}^\sigma \|_{H^k\otimes H^{k-1}}^p 
\leq C_\rho \left( \int_{s}^{t} \left( E\| W_{sr}^{\sigma} \|_{ H^{k}}^\rho\right)^{\frac{2}{\rho}}\| \sigma_r\|_{\mathcal{L}(\R^d;H^k)}^{2}dr\right)^{\frac{p}{2}} 
\leq C_\rho \| \sigma\|_{L_t^\infty \mathcal{L}(\R^d;H^k)}^{2\rho} |t-s|^{\rho}.
\end{equation}


\begin{lemma}
	\label{lem: rough drivers from Ito}
	Let $W$ be a $d$-dimensional Wiener process on the filtered probability space $(\Omega, \mathcal{F}, (\mathcal{F}_t)_{t \in [0,T]}, P)$ and let $\sigma$ satisfy Assumption \ref{asm: building rough drivers: sigma}, with $k > 3+\frac{d}{2}$. Let $W^\sigma$ and $\WW^\sigma$ be defined as in \eqref{def: W sigma} and \eqref{def: bW sigma}, respectively. Then, for every $\alpha \in (\frac{1}{3}, \frac{1}{2})$,  for $P$-a.e. $\omega$
	\begin{equation*}
	\bW^{\sigma} := (W^{\sigma}, \WW^{\sigma}) \in C^{\alpha}([0,T]; C_b^3(\R^d;\R^d)) \times C_2^{2 \alpha}([0,T]; C_b^2(\R^d \times \R^d;\R^d)),
	\end{equation*}
	is a rough driver in the sense of Definition \ref{def:RoughDriver}, and for all $\rho >\frac{2}{1-2\alpha}$, we have
	\begin{equation}
	\label{eq: regularity random rd}
	\| [W^{\sigma}]_{\alpha; C_b^3} \|_{L_{\omega}^\rho} \lesssim  \| \sigma\|_{L_t^\infty \mathcal{L}(\R^d;H^k)},
	\qquad 
	\|[\WW^{\sigma}]_{2 \alpha; C_b^2(\R^d \times \R^d)} \|_{L_{\omega}^\rho} \lesssim  \| \sigma\|_{L_t^\infty \mathcal{L}(\R^d;H^k)}^2.
	\end{equation}
	Moreover, on small time-intervals $|t-s| \leq h \leq T$ we have, for $\bar\alpha \in (\alpha, \frac{1}{2})$,
	$$
	[W^{\sigma}]_{\alpha,h, C_b^3} 
	\leq h^{\bar{\alpha} - \alpha} [W^{\sigma}]_{\bar{\alpha}, C_b^3} , 
	\qquad 
	[\WW^{\sigma}]_{2\alpha,h, C_b^2(\R^d\times \R^d)} 
	\leq h^{\bar{\alpha} - \alpha} [\WW^{\sigma}]_{2\bar{\alpha}, C_b^2(\R^d\times \R^d)},
	\qquad P-a.s.
	$$
\end{lemma}

\begin{proof}
	We first study the space regularity of $\bW$. From the choice of $k$, Sobolev's embedding Theorem \cite[Corollary 9.13]{brezis2011} and inequalities \eqref{eq: bdg first order} and \eqref{eq: bdg second order}, we have that
	\begin{equation*}
	E\sup_{r\in[s,t]}\| W_{sr}^\sigma \|_{C_b^3}^\rho
	\leq E\sup_{r\in[s,t]}\| W_{sr}^\sigma \|_{H^k}^\rho
	\leq C_\rho \| \sigma_t\|_{L_t^\infty \mathcal{L}(\R^d;H^k)}^\rho |t-s|^{\frac{\rho}{2}},
	\end{equation*}
	\begin{equation*}
	E\sup_{r\in[s,t]}\| \WW_{sr}^\sigma \|_{C_b^{3}\otimes C_b^{2}}^\rho
	\leq E\sup_{r\in[s,t]}\| \WW_{sr}^\sigma \|_{H^k\otimes H^{k-1}}^\rho
	\leq C_\rho \| \sigma_t\|_{L_t^\infty \mathcal{L}(\R^d;H^k)}^{2\rho} |t-s|^{\rho}.
	\end{equation*}
	By the Kolmogorov continuity theorem \ref{Kolmogorov}, we obtain \eqref{eq: regularity random rd}.
	
	We check now that Chen's relation \eqref{RDChen} holds $P$-a.s.. Indeed, we have the following,
	\begin{equation*}
	\delta \WW^{\sigma}_{sut} 
	= \int_{u}^{t}(W^{\sigma}_{su}\otimes \nabla) \sigma_r dW_r
	= \left(W^{\sigma}_{su}\otimes \nabla \right) W^{\sigma}_{ut},
	\qquad
	P-a.s.
	\end{equation*}
	To justify the last equality we call $\tilde{H}:=\mathcal{L}(\R^d;H^{k} \otimes H^{k-1})$ and we note that $(W^{\sigma}_{su} \otimes \nabla): \Omega \to L(H^k,\tilde{H})$ is an $\mathcal{F}_u$-measurable random variable taking values in the space of linear operators between two Hilbert spaces. Thanks to the fact that the operator $(W^{\sigma}_{su} \otimes \nabla)$ is measurable with respect to the left-most point of the integral, one can easily adapt \cite[Lemma 2.4.1]{MR3410409} to show that it commutes with the stochastic integral.
	
\end{proof}


%
We shall also need contractive estimates w.r.t. the vector field. 
\begin{lemma} \label{lem:ItoContraction}
	Let $\sigma$ and $\theta$ satisfy Assumption \ref{asm: building rough drivers: sigma}, with $k > 3 +\frac{d}{2}$. Let $\bW^{\sigma}$ and $\bW^{\theta}$ be rough drivers as constructed in Lemma \ref{lem: rough drivers from Ito} w.r.t. the vector fields $\sigma$ and $\theta$. Then, for all $\bar\alpha \in [\alpha, \frac{1}{2})$ and all $\rho>\frac{2}{1-2\bar \alpha}$, there exists $K_\rho \in L^\rho(\Omega)$, such that for all $h\leq T$,
	\begin{equation} \label{ItoDriverContraction}
	[ \bW^{\sigma} - \bW^{\theta} ]_{\alpha,h;C^3_b} \leq  h^{\bar{\alpha} - \alpha} K_\rho (1 + \|\sigma\|_{L_t^\infty \mathcal{L}(\R^d;H^k)} + \|\theta\|_{L_t^\infty \mathcal{L}(\R^d;H^k)})  \|\sigma - \theta\|_{L_t^\infty \mathcal{L}(\R^d;H^k)},
	\qquad P-a.s.
	\end{equation}
\end{lemma}
\begin{proof}
	The proof follows as an application of Kolmogorov continuity theorem as in Lemma \ref{lem: rough drivers from Ito}.
\end{proof}



\subsubsection{Gubinelli integration}

Let $\bar \alpha \in (\frac{1}{3}, \frac{1}{2})$, $\bZ \in \mathscr{C}_g^{\bar\alpha}([0,T], \R^m)$ and let $\beta$ satisfy Assumption \ref{asm: building rough drivers: beta}, for $k\in \N\cup\{0\}$ and $\alpha \in (\frac{1}{3},\bar\alpha)$. Using Gubinelli's integration theory (see \cite[Chapter 4]{FrizHairer}) we define, for each $0\leq s \leq t \leq T$,
\begin{equation}
\label{def: Z beta}
Z^{\beta}_{st} := \int_s^t \beta_r d\bZ_r \in H^k,
\end{equation}
which satisfies (see \cite[Theorem 4.10]{FrizHairer})
$$
\| Z^{\beta}_{st} - \beta_s^j Z_{st}^j - \beta_s^{j,i} \ZZ_{st}^{i,j} \|_{H^k} 
\leq 
C\left( [ Z ]_{\alpha } \Vert \beta^{\sharp} \Vert_{2\alpha ;\mathcal{L}(\R^m; H^k)} + [ \ZZ]_{2\alpha } \Vert \beta^\prime \Vert_{\alpha ;\mathcal{L}(\R^{m\times m}; H^k)}  \right) |t-s|^{3 \alpha}. 
$$
and we have, 
\begin{equation}
\label{eq: hoelder sobolev zeta beta}
\| Z^{\beta}_{st} \|_{H^k} 
\leq  
C
\Vert (\beta, \beta^\prime) \Vert_{Z, \alpha ;H^k}
[\bZ]_{\alpha }
|t-s|^{\alpha} , 
\quad
\| (Z^{\beta})^{\sharp}_{st}  \|_{H^k}  
\leq 
C
\Vert (\beta, \beta^\prime) \Vert_{Z, \alpha ;H^k}
[\bZ]_{\alpha } |t-s|^{2\alpha} 
\end{equation}
For $t\in [0,T]$, we define $Z^{\beta}_{t}:= Z^{\beta}_{0t}$ and we consider $\left( Z_{t}^{\beta} \otimes \nabla\right) \beta_{t} \in \mathcal{L}(\R^m;H^{k} \otimes H^{k-1})$, with Gubinelli derivative 
\begin{equation*}
( Z_{t}^{\beta} \otimes \nabla ) \beta^\prime_{t} + ( \beta_t \otimes \nabla ) \beta_{t} \in \mathcal{L}(\R^{m\times m};H^{k} \otimes H^{k-1}).
\end{equation*}
Consequently we can define the integral $\int_s^t  (Z^{\beta}_{r} \otimes \nabla) \beta_r d\bZ_r \in H^k\otimes H^{k-1}$ via the local expansion
\begin{align*}
\big\| &\int_s^t (Z^{\beta}_{r}\otimes\nabla) \beta_r  d\bZ_r 
-  (Z^{\beta}_{s}\otimes \nabla) \beta_s^j   Z_{st}^j 
- \left( ( Z_{t}^{\beta} \otimes \nabla) \beta^{j,i}_{t} + ( \beta_t^j \otimes \nabla) \beta_{t}^i\right) \ZZ^{i,j}_{st} \big\|_{H^k\otimes H^{k-1}} \\
& \leq C\left( [Z]_{\alpha } [ ((Z^{\beta}\otimes \nabla) \beta )^{\sharp} ]_{2\alpha; \mathcal{L}(\R^m; H^k\otimes H^{k-1})} 
+ [\ZZ]_{2\alpha } [( Z_{t}^{\beta} \otimes \nabla) \beta^\prime_{t} + ( \beta_t \otimes \nabla) \beta_{t}]_{\alpha;\mathcal{L}(\R^{m\times m}; H^k\otimes H^{k-1}) }  \right)  |t-s|^{3 \alpha}. 
\end{align*}
Defining 
\begin{equation}
\label{def: bZ beta}
\ZZ^{\beta}_{st} : = \int_s^t (Z^{\beta}_{r} \otimes \nabla) \beta_r  d \bZ_r - (Z^{\beta}_{s}\otimes\nabla) Z^{\beta}_{st} ,
\end{equation}
we get
\begin{align*}
\|\ZZ^{\beta}_{st} \|_{H^k\otimes H^{k-1}}  
\leq  
C
\Vert (\beta, \beta^\prime) \Vert_{Z,\alpha ;H^k}^2
([\bZ]_{\alpha } + [\bZ]_{\alpha }^2)
|t-s|^{2 \alpha} .
\end{align*}
We have the following lemmas of which we omit the proofs as they follow quite easily from the discussion above, standard computations on rough integrals, and Sobolev embedding Theorem \cite[Corollary 9.13]{brezis2011}.
\begin{lemma}
	\label{lem: rough drivers from Gubinelli}
	Let $\bar \alpha \in (\frac{1}{3}, \frac{1}{2})$ and $\bZ \in \mathscr{C}_g^{\bar\alpha}([0,T], \R^m)$. Assume that $\beta$ satisfies Assumption \ref{asm: building rough drivers: beta}, with $k \geq 4+\frac{d}{2}$ and $\alpha \in (\frac{1}{3},\bar\alpha)$. Let $Z^\beta$ and $\ZZ^\beta$ be defined as in \eqref{def: Z beta} and \eqref{def: bZ beta}, respectively. Then,
	\begin{equation*}
	\bZ^{\beta} := (Z^{\beta}, \ZZ^{\beta}) \in C^{\alpha}([0,T]; C_b^3(\R^d;\R^d)) \times C_2^{2 \alpha}([0,T]; C_b^2(\R^d \times \R^d;\R^d)),
	\end{equation*}
	is a rough driver in the sense of Definition \ref{def:RoughDriver} and we have for time intervals of size $h\leq T$,
	\begin{equation*}
	[ Z^{\beta}_{st} ]_{Z,\alpha,h;C^3_b} 
	\leq  
	C
	\Vert (\beta, \beta^\prime) \Vert_{Z, \alpha ;H^k}
	[\bZ]_{\alpha,h },
	\qquad
	[\ZZ^{\beta}_{st} ]_{2\alpha,h;C^3_b\otimes C^2_b}  
	\leq  
	C
	\Vert (\beta, \beta^\prime) \Vert_{Z,\alpha ;H^k}^2
	( 1 + [\bZ]_{\alpha,h }) [\bZ]_{\alpha,h }.
	\end{equation*}
\end{lemma}

\begin{lemma} \label{lemma:RoughIntegralContraction}
	Let $\bar \alpha \in (\frac{1}{3}, \frac{1}{2})$ and $\bZ \in \mathscr{C}_g^{\bar\alpha}([0,T], \R^m)$. Assume that $\beta$ and $\gamma$ satisfy Assumption \ref{asm: building rough drivers: beta}, with $k \geq 4+\frac{d}{2}$ and $\alpha \in (\frac{1}{3},\bar\alpha)$. 
	Let $\bZ^{\beta}$ and $\bZ^{\gamma}$ be rough drivers constructed as in Lemma \ref{lem: rough drivers from Gubinelli}. Then, on time intervals of size $h\leq T$,
	\begin{equation} \label{RoughDriverContraction}
	[ \bZ^{\beta} - \bZ^{\gamma} ]_{\alpha,h;C^3_b} \leq C (1 + \|(\beta,\beta')\|_{Z, \alpha;H^k} +\|(\gamma, \gamma')\|_{Z, \alpha;H^k}) \|(\beta,\beta') - (\gamma, \gamma')\|_{Z, \alpha;H^k} ([\bZ]_{\alpha,h}\vee [\bZ]_{\alpha,h}^{\frac{1}{2}}) . 
	\end{equation}
\end{lemma}
%

Let us show that the above definition coincides with the usual definition of solutions of rough path equations. 
\begin{lemma} \label{lemma:NonLinearVsClassical}
	Suppose $x : [0,T] \rightarrow \R^d$ is a solution of $dx_t = \bZ^{\beta}_{dt}(x_t)$ in the sense of Definition \ref{def:RDSolution}.
	Then $x$ also solves the classical rough path equation driven by $Z$ with coefficient $\beta$, i.e. $(x, \beta(x)) \in \mathscr{D}_Z^{2 \alpha}$ satisfies the following equation in the sense of Davie \cite{Davie},
	$$
	x_t = \xi  + \int_0^t  \beta_r(x_r) d\bZ_r,
	$$
	where the $\beta(x)$ is also controlled by $Z$ with Gubinelli derivative $\beta'(x) +\nabla\beta(x) \beta(x)$. 
\end{lemma}
\begin{proof}
	Assume $x$ is a solution to the non-linear equation and let us show that it also satisfies
	$$
	\delta x_{st} = \beta_s^j(x_s) Z^j_{st} + ( \beta_s^{j,i}(x_s) + \nabla \beta_s^j(x_s) \beta_s^i(x_s)) \ZZ^{i,j}_{st} + \bar{x}_{st}^{\natural}
	$$
	for some remainder $\bar{x}^{\natural}$. By definition of $Z^{\beta}$ we have
	$$
	|Z^{\beta}_{st}(x_s) - \beta_s^j(x_s)Z^j_{st} - \beta_s^{j,i}(x_s) \ZZ^{i,j}_{st} | \lesssim |t-s|^{3 \alpha} .
	$$
	Moreover
	\begin{align*}
	|\ZZ_{st}^{\beta}(x_s)   - &  \nabla \beta_s^j(x_s) \beta_s^i(x_s) \ZZ^{i,j}_{st}|    \leq | Z_s^{\beta}(x_s) \nabla \beta_s(x_s)^j   Z^j_{st} + Z_s^{\beta}(x_s) \nabla \beta_s^{j,i}(x_s)   \ZZ^{i,j}_{st} - \nabla Z_{st}^{\beta}(x_s)  Z_s^{\beta}(x_s) | \\
	& +  \left| \int_s^t Z^{\beta}_r(x_s) \nabla \beta(x_s)  d\bZ_r - Z_s^{\beta}(x_s) \nabla \beta_s^j(x_s)   Z^j_{st} - (Z_s^{\beta}(x_s) \nabla \beta_s^{j,i}(x_s)   + \nabla \beta_s^j(x_s) \beta_s^i(x_s) ) \ZZ^{i,j}_{st} \right| \\
	&   \lesssim |t-s|^{3 \alpha}
	\end{align*}
	by definition of $Z^{\beta}$ and $\int_s^t \nabla \beta(x_s) Z^{\beta}_r(x_s) d\bZ_r$. This shows that 
	$
	|x_{st}^{\natural} - \bar{x}_{st}^{\natural} | \lesssim |t-s|^{3 \alpha}
	$
	which proves that the solutions coincide. Notice that the above bounds depend on $\|(\beta, \beta')\|_{\alpha, Z; H^k}$ only.
\end{proof}

\subsubsection{Mixed It\^{o} and rough path integration}
\label{section: mixed rough driver}
Let $W$ be a $d$-dimensional Wiener process on the filtered probability space $(\Omega, \mathcal{F}, (\mathcal{F}_t)_{t \in [0,T]}, P)$. Let $\bar \alpha \in (\frac{1}{3}, \frac{1}{2})$, $\bZ \in \mathscr{C}_g^{\bar\alpha}([0,T], \R^m)$.
Assume that $\sigma$ and $\beta$ satisfy Assumption \ref{asm: building rough drivers: sigma} and \ref{asm: building rough drivers: beta} respectively, for $k\in \N\cup\{0\}$ and $\alpha \in (\frac{1}{3},\bar\alpha)$. Let $W^\sigma$ be defined as in \eqref{def: W sigma} and $Z^{\beta}$ be defined as in \eqref{def: Z beta}. We define
\begin{equation}
\label{def: F}
F_{st} := W_{st}^{\sigma}+ Z_{st}^{\beta}.
\end{equation}
We remark that the first term on the right hand side of the above equation is random, whereas the second is deterministic. Define heuristically
\begin{equation}
\label{def: bF}
\FF_{st} 
:= \WW^{\sigma}_{st}
+ \ZZ_{st}^{\beta}
+ \int_s^t (Z_{sr}^{\beta}\otimes \nabla)  \sigma_r  dW_r 
+ \int_s^t ( W_{sr}^{\sigma} \otimes \nabla)  \beta_r d\bZ_r.
\end{equation}
The first two terms in the right hand side are defined as in \eqref{def: bW sigma} and \eqref{def: bZ beta} respectively, we need to make the last two rigorous. For the third term, using the It\^{o} theory in Hilbert spaces as we did is Section \ref{sec:Ito theory}, we see that the integral
\begin{equation*}
\int_s^t (Z_{r}^{\beta}\otimes \nabla)  \sigma_r dW_r \in M^2_T(H^k\otimes H^{k-1}),
\end{equation*}
is well-defined. Indeed, we have $(Z_{r}^{\beta} \otimes \nabla) \sigma_r \in \mathcal{L}(\R^d; H^{k} \otimes H^{k-1})$ for all $0\leq r \leq T$. Hence, we can define
\begin{equation*}
\int_s^t (Z_{sr}^{\beta}\otimes \nabla)  \sigma_r  dW_r 
:= 
\int_s^t (Z_{r}^{\beta}\otimes \nabla)  \sigma_r dW_r
- (Z_{s}^{\beta}\otimes \nabla) W^{\sigma}_{st}
:\Omega \to H^{k} \otimes H^{k-1}.
\end{equation*}
Similarly, we have $ (\sigma_r \otimes \nabla) Z_{r}^{\beta} \in \mathcal{L}(\R^d; H^{k} \otimes H^{k-1})$ and $\int_s^t (\sigma_r \otimes \nabla) Z_{r}^{\beta} dW_r \in M^2_T(H^k\otimes H^{k-1})$. We define
$$
\int_s^t ( W_{sr}^{\sigma} \otimes \nabla)  \beta_r d\bZ_r :=  (W_{st}^{\sigma}\otimes \nabla)Z_{t}^{\beta}  - \int_s^t (\sigma_r \otimes \nabla) Z_{r}^{\beta} dW_r 
:\Omega \to H^{k} \otimes H^{k-1}.
$$

\begin{lemma}
	\label{lem: rough drivers mixed}
	Let $W$ be a $d$-dimensional Wiener process on the filtered probability space $(\Omega, \mathcal{F}, (\mathcal{F}_t)_{t \in [0,T]}, P)$.
	Let $\bar \alpha \in (\frac{1}{3}, \frac{1}{2})$ and $\bZ \in \mathscr{C}_g^{\bar\alpha}([0,T], \R^m)$. Assume that $\sigma$ and $\beta$ satisfy Assumption \ref{asm: building rough drivers: sigma} and \ref{asm: building rough drivers: beta} respectively,, with $k \geq 4+\frac{d}{2}$ and $\alpha \in (\frac{1}{3},\bar\alpha)$. Let $F$ and $\FF$ be defined as in \eqref{def: F} and \eqref{def: bF}, respectively. Then, for $P$-a.e. $\omega$,
	\begin{equation*}
	\bF := (F, \FF) \in C^{\alpha}([0,T]; C_b^3(\R^d;\R^d)) \times C_2^{2 \alpha}([0,T]; C_b^2(\R^d \times \R^d;\R^d)),
	\end{equation*}
	is a rough driver in the sense of Definition \ref{def:RoughDriver}. 
	Moreover, on time intervals of size $h \leq T$ we have that, for all $\rho>\frac{2}{1-2\bar\alpha}$, there exists $K_\rho \in L^\rho(\Omega)$, such that, $P$-a.s.,
	\begin{equation}
	\label{eq: FF bound}
	[\bF]_{\alpha,h;C^3_b}	\leq  (h \vee h^{\frac12})^{\bar\alpha - \alpha}  K_\rho L(\sigma, \beta, \bZ),
	\end{equation}
	where $L$ is defined in \eqref{LConstant}.
\end{lemma}
\begin{proof}
	It is immediate to verify that the couple $(F, \FF)$ satisfies Chen's relation \eqref{RDChen}. We give now estimates on the first order term \eqref{def: F}. As a consequence of the definition of $F$ and Lemma \ref{lem: rough drivers from Gubinelli}, we have, on an interval of size $h\leq T$, 
	\begin{equation*}
	\Vert [F]_{\alpha,h;C_b^3} \Vert_{L_{\omega}^\rho} \leq \Vert [W^{\sigma}]_{\alpha,h;C_b^3} \Vert_{L_{\omega}^\rho}  + C\|(\beta,\beta')\|_{Z, \alpha,h;H^k} [\bZ]_{\alpha,h} . 
	\end{equation*}
	We use now Lemma \ref{lem: rough drivers from Ito} to control the first term in the right hand side.
	
	Now we study the regularity of $\FF$. Using BDG inequality \cite[Theorem 2.4.7]{MR3410409} and inequality \eqref{eq: hoelder sobolev zeta beta}, we have for all $\rho\geq 1$ and $0\leq s \leq t \leq T$, $t-s \leq h$,
	\begin{align*}
	\left\| \int_s^t (Z_{sr}^{\beta} \otimes \nabla)  \sigma_r dW_r \right\|_{L_{\omega}^\rho(H^k\otimes H^{k-1})}
	\leq &
	C_\rho \left( \int_s^t\Vert( Z^{\beta}_{sr}\otimes \nabla) \sigma_{r}\Vert_{\mathcal{L}(\R^d;H^k\otimes H^{k-1})}^2 dr \right)^{\frac{1}{2}} \\
	\leq &
	C_\rho \Vert \sigma \Vert_{L_t^{\infty}\mathcal{L}(\R^d,H^k)} [Z^{\beta} ]_{\alpha; H^{k}} \vert t-s \vert^{\alpha + \frac{1}{2}}\\
	\leq &
	C_\rho \Vert \sigma \Vert_{L_t^{\infty}\mathcal{L}(\R^d,H^k)} \Vert (\beta, \beta^\prime) \Vert_{Z, \alpha ;H^k}
	[\bZ]_{\alpha,h } \vert t-s \vert^{\alpha + \frac{1}{2}}.
	\end{align*}
	By Kolmogorov continuity theorem \ref{Kolmogorov}, we obtain that for every $\rho>\frac{2}{1-2\alpha}$ there exists $K_\rho\in L^\rho(\Omega)$, such that
	\begin{equation*}
	\left[ \int_s^t (Z_{sr}^{\beta} \otimes \nabla)  \sigma_r dW_r \right]_{2\alpha;H^k\otimes H^{k-1}}
	\leq K_\rho \Vert \sigma \Vert_{L_t^{\infty}\mathcal{L}(\R^d,H^k)} \Vert (\beta, \beta^\prime) \Vert_{Z, \alpha ;H^k}
	[\bZ]_{\alpha },
	\qquad P-a.s.
	\end{equation*}
	Similar considerations lead to 
	\begin{equation*}
	\left[ \int_s^t ( W_{sr}^{\sigma} \otimes \nabla)  \beta_r d\bZ_{r} \right]_{2\alpha;H^k\otimes H^{k-1}}
	\leq K_\rho \Vert \sigma \Vert_{L_t^{\infty}\mathcal{L}(\R^d,H^k)} \Vert (\beta, \beta^\prime) \Vert_{Z, \alpha ;H^k}
	[\bZ]_{\alpha },
	\qquad P-a.s.
	\end{equation*}
	Putting together the last inequalities, Lemma \ref{lem: rough drivers from Ito} and Lemma \ref{lem: rough drivers from Gubinelli} yields
	\begin{equation*}
	[\FF_{st} ]_{2\alpha,h;H^k\otimes H^{k-1}}  
	\leq  
	K_\rho \Vert (\beta, \beta^\prime) \Vert_{Z, \alpha ;H^k}
	[\bZ]_{\alpha, h } 
	\left(
	\Vert (\beta, \beta^\prime) \Vert_{Z, \alpha ;H^k}( 1 + [\bZ]_{\alpha,h }) 
	+ \Vert \sigma \Vert_{L_t^{\infty}\mathcal{L}(\R^d,H^k)}
	\right) 
	+ [\WW^{\sigma}]_{2 \alpha,h; H^k\otimes H^{k-1}}.
	\end{equation*}
	Inequality \eqref{eq: FF bound} follows immediately from the Sobolev embedding theorem \cite[Corollary 9.13]{brezis2011} .
\end{proof}

\begin{lemma}
	Let $W$ be a $d$-dimensional Wiener process on the filtered probability space $(\Omega, \mathcal{F}, (\mathcal{F}_t)_{t \in [0,T]}, P)$.
	Let $\bar \alpha \in (\frac{1}{3}, \frac{1}{2})$ and $\bZ \in \mathscr{C}_g^{\bar\alpha}([0,T], \R^m)$. Assume that $\sigma, \theta$ satisfy Assumption \ref{asm: building rough drivers: sigma} and that $\beta, \gamma$ satisfy Assumption \ref{asm: building rough drivers: beta}, with $k \geq 4+\frac{d}{2}$ and $\alpha \in (\frac{1}{3},\bar\alpha)$. Let $\bF$ and $\bG$ be nonlinear rough drivers constructed from $F_{st}:= W_{st}^{\sigma} + Z_{st}^{\beta}$ and	$G_{st} := W_{st}^{\theta} + Z_{st}^{\gamma}$ as in Lemma \ref{lem: rough drivers mixed}.
	
	Then, for all $\rho>\frac{2}{1-2\bar \alpha}$, there exists $K_\rho \in L^\rho(\Omega)$, such that for any time interval of size $h \leq T$,
	\begin{align}
	[ \bF -  \bG ]_{\alpha,h;C_b^3} 
	\leq 
	(h\vee h^{\frac{1}{2}})^{\bar{\alpha} - \alpha} K_\rho M  
	(\| \sigma - \theta \|_{L_t^\infty H^k} 
	+ 
	\| (\beta, \beta^\prime) - (\gamma, \gamma^\prime)\|_{Z, \alpha, h; H^k}),
	\qquad
	P-a.s.
	\label{eq: contractive rough drivers} 
	\end{align}
	where we set $M:= L(\sigma, \beta, \bZ) + L(\theta, \gamma, \bZ)$ and $L$ is defined as in \eqref{LConstant}.
	%
	%
\end{lemma}

\begin{proof}
	We already have contractive estimates from Lemmas \ref{lem:ItoContraction} and \ref{lemma:RoughIntegralContraction} for the It\^{o} and Gubinelli terms. We look now at the mixed integrals. For every $p\geq 1$, we have, for $|t-s| \leq h\leq T$,
	\begin{align*}
	\|\int_{s}^{t}(Z^{\beta}_{sr}\otimes \nabla)\sigma_r dW_r 
	& - \int_{s}^{t} (Z^{\gamma}_{sr}\otimes \nabla)\theta_r dW_r \|_{L_\omega^pH^k\otimes H^{k-1}} \\
	\leq &\| \int_{s}^{t}(Z^{\beta - \gamma}_{sr}\otimes \nabla) \sigma_rdW_r\|_{L_\omega^pH^k\otimes H^{k-1}} 
	+\| \int_{s}^{t}(Z^{\gamma}_{sr}\otimes \nabla) (\sigma_r - \theta_r) dW_r \|_{L_\omega^pH^k\otimes H^{k-1}} \\
	\leq & C_p \left[
	\Vert \sigma \Vert_{L_t^{\infty}\mathcal{L}(\R^d,H^k)} \Vert (\beta, \beta^\prime) - (\gamma, \gamma^\prime) \Vert_{Z, \alpha,h ;H^k}
	+ \Vert \sigma - \theta \Vert_{L_t^{\infty}\mathcal{L}(\R^d,H^k)} \Vert (\gamma, \gamma^\prime) \Vert_{Z, \alpha, h ;H^k}
	\right] \\
	 & \cdot 
	[\bZ]_{\alpha,h } \vert t-s \vert^{\alpha + \frac{1}{2}}.
	\end{align*}
	The same estimates is true for the other mixed term. We can conclude by applying Kolmogorov continuity theorem.
\end{proof}

\subsection{Integrability of the random rough driver}
\label{section: integrability rough driver}
In this section we are concerned with the study of exponential moments of the random rough driver. We will use the approach introduced by \cite{cass2013} and described in \cite[Chapter 11]{FrizHairer}. 

\begin{lemma}
	\label{lem: pre fernique}
	Let $(\Omega := C([0,T];\R^m), \mathcal{B}(\Omega), P)$ be the canonical Wiener space with Cameron-Martin space $\mathcal{H}\subset \Omega$. We define on this space the canonical Wiener process as $W_t(\omega) = \omega(t)$. Let $\bar \alpha \in (\frac{1}{3}, \frac{1}{2})$ and $\bZ \in \mathscr{C}_g^{\bar\alpha}([0,T], \R^m)$. Assume that $\sigma$ and $\beta$ satisfy Assumption \ref{asm: building rough drivers}, with $k \geq 4+\frac{d}{2}$ and $\alpha \in (\frac{1}{3},\bar\alpha)$, and let $\bF$ be defined as in Lemma \ref{lem: rough drivers mixed}.
	Let $p:= \alpha^{-1} \in (2,3)$ and $q\geq 1$, such that $\frac{1}{p} + \frac{1}{q} > 1$. Then, there exists $C := C(p,q) >0$ and a null set $N\subset \Omega$, such that, $\forall \omega \in N^c$, $\forall [s,t] \subset [0,T]$ and $\forall h\in C^{q-var}$,
	\begin{equation*}
	\la \bF \ra_{p, [s,t]}(\omega) \leq C
	\la \sigma \ra_{p;[s,t]}
	\left(g_{st}(\omega - h) + \la h \ra_{q;[s,t]}\right),
	\end{equation*}
	where, $g_{s,t}:\Omega \to \R_+$ is defined as
	\begin{align}
	\label{eq:def of g}
	g_{s,t} := & \la \bF \ra_{p, [s,t]}
	+ \la (W^{\sigma}\otimes\nabla) \sigma\ra_{p;[s,t]}
	+ \la (\sigma\otimes\nabla)W^{\sigma} \ra_{p;[s,t]}\\
	& + \la (Z^{\beta}\otimes\nabla) \sigma\ra_{p;[s,t]}
	+ \la (\sigma\otimes\nabla)Z^{\beta} \ra_{p;[s,t]}\nonumber
	\end{align}
\end{lemma}
\begin{proof}
	The proof of this result follows very closely the proof of \cite[Theorem 11.5]{FrizHairer}. We repeat here the important pieces, where the dependence of the stochastic integrals on the space parameter $x$ has to be taken into account.
	We look at the first order term of $\bF$. By definition, we have
	\begin{equation*}
	F_{st}(\omega) = W^\sigma_{st}(\omega) + Z^{\beta}_{st}.
	\end{equation*}
	For every $s,t \in [0,T]$, the term $W^\sigma_{st}$ is constructed as an $L_{\omega}^2H^k$ limit, hence there exists a sequence of partitions $(\Pi_m)_{m\in \N}$ and a null set $N_{st}$ such that 
	\begin{equation}
	\label{eq:conv ito int on partition}
	W^\sigma_{st}(\omega)
	= \lim_{m\to\infty}\int_{\Pi_m} \sigma_r dW_r(\omega)
	:= \lim_{m\to\infty}\sum_{t_i \in \Pi_m}\sigma_{t_i}(W_{t_{i+1}}(\omega) - W_{t_{i}}(\omega)),
	\end{equation}
	for every $\omega \in N_{st}^c$. We call $N_{1}$ the intersection of $N_{st}$ over all dyadic times and we note that it is still a null set. Similarly, we can construct a null set $N_{2}$ such that the function $W^\sigma(\omega)$ is of bounded $p$-variation for every $\omega \in N_{2}^c$. Let $\omega \in N_{1}^c \cap N_{2}^c$, we have, 
	\begin{equation}
	\label{eq:split integral omega and h}
	\lim_{m\to\infty}\int_{\Pi_m} \sigma_r dW_r (\omega+h)
	= \lim_{m\to\infty}\int_{\Pi_m} \sigma_r(x)dW_r(\omega)
	+ \lim_{m\to\infty}\int_{\Pi_m} \sigma_r(x)dh_r.
	\end{equation}
	The first limit on the right hand side exists because of the choice of the null set that we made in \eqref{eq:conv ito int on partition}. The last limit is well defined as a Young integral, since $\sigma$ and $h$ are of complementary variation, see \cite[Section 4.1]{FrizHairer}. Hence, also the left hand side of \ref{eq:split integral omega and h} converges and is, by definition, $W^\sigma_{st}(\omega+h)$.
	
	Hence, we obtain, $\forall \omega \in N_{1}^c \cap N_{2}^c$, $h\in C^{q-\Var}$, and for all dyadic times $[s,t]\subset [0,T]$,
	\begin{equation}
	\label{eq:first order term translation}
	F_{st}(\omega) = F_{st}(\omega - h) + \int_{s}^{t}\sigma_r dh_r.
	\end{equation}
	To generalize to any subset $[s,t]\subset [0,T]$, we can use a continuity argument, see \cite[Theorem 11.5]{FrizHairer}.
	
	We compute now the $p$-variation in equation \eqref{eq:first order term translation} and we obtain
	\begin{equation*}
	\la F \ra_{p,[s,t]}(\omega) \leq C_p \la \sigma \ra_{p,[s,t]} \left( \la F \ra_{p,[s,t]}(\omega - h) + \la h \ra_{q,[s,t]} \right).
	\end{equation*}

	Proceeding similarly for the second order term $\FF$, we have that there exists a null set $N \subset \Omega$ such that $\forall \omega \in N^c$, $\forall h\in C^{q-var}$ and for all times $[s,t]\subset [0,T]$,
	\begin{align*}
	\FF_{st}(\omega)  
	= & \FF_{st}(\omega-h) 
	+ \int_{s}^{t}(W^{\sigma}_{sr}(\omega - h)\otimes\nabla) \sigma_r dh_r\\
	& + \int_{s}^{t}(\sigma_u\otimes \nabla )W^{\sigma}_{ut}(\omega - h)dh_u 
	+ \int_{s}^{t}\int_{s}^{r}(\sigma_u \otimes \nabla) \sigma_rdh^l_udh_r \\
	& + \left((Z^{\beta}_{0\cdot}\otimes \nabla) \int_{0}^{\cdot}\sigma_r dh_r\right)_{st}
	- \int_{s}^{t}(Z^{\beta}_{sr}\otimes \nabla)\sigma_{r} dh_r 
	+ \int_{s}^{t}(\sigma_r\otimes \nabla )Z^{\beta}_{sr} dh_r.
	\end{align*}
	to obtain the third term on the right hand side, we used stochastic Fubini Theorem  as follows
	\begin{equation*}
	\int_{s}^{t}\int_{s}^{r}(\sigma_u \otimes \nabla) \sigma_rdh_udW_r(\omega - h) 
	= \int_{s}^{t} (\sigma_u \otimes \nabla )W^{\sigma}_{ut}(\omega - h) dh_u.
	\end{equation*}
	We compute the $p$-variation for the second order term. Using inequalities of the type $\sqrt{ab} \leq \sqrt{a} + \sqrt{b}$, for $a, b \in \R_+$, 
	we obtain, for all $\omega \in N$,
	\begin{equation*}
	\la \FF\ra_{p,[s,t]}^{\frac{1}{2}}(\omega)  
	\leq
	C_p
	\la \sigma\ra_{p,[s,t]}
	\left( 
	g(\omega - h) 
	+ \la h \ra_{q,[s,t]}
	\right),
	\end{equation*}
	where $g_{s,t}$ is defined in \eqref{eq:def of g}. This concludes the proof
\end{proof}

For every $s,t \in [0,T]$, we define the control $w_{\bF}(s,t) = \la \bF \ra_{p,[s,t]}^p$ and we construct the greedy partition, following the construction in Section \ref{section: Holder and p-variation spaces}. Let $N_{\beta}$ be defined as in \eqref{eq: N beta}, for any $\beta > 0$. We call $N$ the integer-valued random variable given by
\begin{equation}
\label{def: greedy partition}
N(\omega) 
:= N_1( w_{\bF},[0,T])(\omega),
\end{equation}
for $\omega \in \Omega$.
For $y\geq 0$, let
\begin{equation*}
	\Phi(y):=\frac{1}{\sqrt{2\pi}}\int_{-\infty}^{y} e^{-\frac{x^2}{2}}dx
\end{equation*}
be the cumulative distribution function of a standard Gaussian random variable and $\bar\Phi = 1-\Phi$. We include a straightforward Lemma needed to estimate $N$. 
\begin{lemma}
	\label{lem: from tail to MGF}
	Let $C>0$ and $\bar a \in \R$. If $Y$ is a positive random variable such that $P(Y > t) \leq \bar\Phi(\bar a+ t/c)$, for every $t> a$, then
	\begin{equation*}
	Ee^{sY} \leq e^{as} + e^{-cs\bar a + c^2s^2/2}
	\qquad \forall s>0.
	\end{equation*}
\end{lemma}
\begin{proof}
	We use elementary considerations and Fubini theorem, to obtain
	\begin{align*}
	Ee^{sY} 
	= & \int_{0}^{\infty}P(e^{sY}>t)dt
	=  \int_{0}^{e^{sa}}P(Y>\log t/s)dt + \int_{e^{sa}}^{\infty}P(Y>\log t/s)dt\\
	\leq & e^{sa} + \int_{0}^{\infty}\bar\Phi(\bar a + \log t/cs) dt 
	=  e^{sa} + \int_{0}^{\infty}\frac{1}{\sqrt{2\pi}}\int_{a + \log t/cs} e^{-x^2/2}dxdt \\
	= & e^{sa} + \int_{\R} \int_{0}^{e^{cs(x-\bar a)}} \frac{1}{\sqrt{2\pi}}e^{-x^2/2}dtdx 
	=  e^{sa} + \int_{\R} \frac{1}{\sqrt{2\pi}}e^{cs(x-\bar a)}e^{-x^2/2}dx \\
	= & e^{sa} + e^{-cs\bar a + c^2s^2/2}.
	\end{align*}
\end{proof}

\begin{theorem}
	\label{thm: weibull tail}
	Under the same assumptions of Lemma \ref{lem: pre fernique}, the random variable $N$ defined in \eqref{def: greedy partition} has a Gaussian tail.
	Moreover, there exists $C= C(T, p) > 0$, such that $C$ is bounded when $T$ is small and for all $s>1$,
	\begin{align*}
	Ee^{sN} \leq e^{C (\la \sigma\ra_{p}^{p}+1)L(\sigma, \beta, \bZ)^{p} s^2}.
	\end{align*}
	where $L$ is defined in \eqref{LConstant}.
\end{theorem}

\begin{proof}
	The main ingredient, which is still to prove, is that, for $P$-a.e. $\omega$,
	\begin{equation*}
		N_1( w_{\bF},[0,T])^{\frac{1}{q}}(\omega) \leq C\la \sigma\ra_{p}^{\frac{p}{q}}(g_{0,T}(\omega - h)^{\frac{p}{q}} + \la h\ra_{q,[0,T]}),
	\end{equation*}
	where $g$ is defined as in \eqref{eq:def of g} and $C:=C(p,q)$. The proof of this inequality follows from Lemma \ref{lem: pre fernique} in the same way as the proof of \cite[Lemma 11.12]{FrizHairer}.
	It follows from \cite[Proposition 11.2]{FrizHairer}, that we can take $q=1$, to obtain
	\begin{equation*}
	N_1( w_{\bF},[0,T])(\omega) \leq C\la \sigma\ra_{p}^{p}(g_{0,T}(\omega - h) + \la h\ra_{\mathcal{H}}),
	\end{equation*}
	withe $C:=C(T,p)$.
	By assumption, $\sigma$, $\beta$ and $\bZ$ are of finite $p$-variation. This implies that $g$ is almost surely finite and we can apply the generalized Fernique Theorem \cite[Theorem  11.7]{FrizHairer} as follows. We set $f = N$ and $g = C\la \sigma\ra_{p}^{p} g^{p}$ defined as in \eqref{eq:def of g}. We must now find $a>0$ such that the following set has positive measure,
	\begin{equation*}
		A_a = \{\omega \in \Omega\mid C \la \sigma\ra_{p}^{p} g^{p}(\omega) \leq a\}.
	\end{equation*}
	We know from Lemma \ref{lem: rough drivers mixed} that $E[g^{p}]^{\frac{1}{p}} \leq C L(\sigma, \beta, \bZ)$. From Chebychev inequality, we have (where $C$ may change from a term to the next)
	\begin{equation*}
		P(g^{p}\geq a(C\la \sigma\ra_{p}^{p})^{-1}) 
		\leq \frac{C\la \sigma\ra_{p}^{p}}{a}Eg^{p} 
		\leq \frac{1}{a}C\la \sigma\ra_{p}^{p} L(\sigma, \beta, \bZ)^{p}.
	\end{equation*}
	Using the previous estimates, we obtain that, 
	$$
	P(A_a) 
	= 1-P(g^{p}> a(C\la \sigma\ra_{p}^{p})^{-1})  
	\geq  1-P(g^{p}\geq a(C\la \sigma\ra_{p}^{p})^{-1}) 
	\geq 1 - \frac{1}{a}C \la \sigma\ra_{p}^{p}L(\sigma, \beta, \bZ)^{p}.
	$$
	where $C=C(T,p)$ is again allowed to increase in the last inequality. Moreover,
	\begin{equation}
	\label{eq: c to 0}
	C(T,p) \to 0,
	\qquad
	 \mbox{as } T\to 0.
	\end{equation}
	 If we now fix $a = (C+1) \la \sigma\ra_{p}^{p}L(\sigma, \beta, \bZ)^{p}$, we have that $P(A_a) \geq 1-\frac{C}{C+1} > 0$. From Fernique Theorem \cite[Theorem 11.7]{FrizHairer}, we have, for $r>a$,
	\begin{equation*}
		P(N > r) \leq \bar\Phi(\bar a + r(C\la \sigma\ra_{p}^{p})^{-1}),
	\end{equation*}
	where $\bar a = \hat a - a(C\la \sigma\ra_{p}^{p})^{-1}$ and $\hat a = \Phi^{-1}(P(A_a))$. By our choice of $a$ and the monotonicity of $\Phi^{-1}$, we have that $\hat a \geq \Phi^{-1}(1-\frac{C}{C+1})$, which is a universal constant depending only on $(p,T)$, but can be negative. It follows from \eqref{eq: c to 0} that $\hat a \to \infty$ as $T\to 0$. We apply Lemma \ref{lem: from tail to MGF} that, with $s>1$ (chosen so that $s\leq s^2$), $a$ and $\bar a$ as before and $c=C\la \sigma\ra_{p}^{p}$.
	\begin{align*}
	Ee^{sN} 
	\leq & e^{(C+1) \la \sigma\ra_{p}^{p}L(\sigma, \beta, \bZ)^{p} s} 
	+ e^{-C\la \sigma\ra_{p}^{p}s(\hat a - a(C\la \sigma\ra_{p}^{p})^{-1}) + (C\la \sigma\ra_{p}^{p})^2s^2/2}\\
	\leq & e^{(C+1) \la \sigma\ra_{p}^{p}L(\sigma, \beta, \bZ)^{p} s^2} 
	+ e^{C\la \sigma\ra_{p}^{p}s^2[(-\Phi^{-1}(1-\frac{C}{C+1}) + \frac{C+1}{C}L(\sigma, \beta, \bZ)^{p})+ C\la \sigma\ra_{p}^{p}/2]}\\
	\leq & e^{C (\la \sigma\ra_{p}^{p}+1)L(\sigma, \beta, \bZ)^{p} s^2}.
	\end{align*}
	The constant $C$ is allowed to change again in the last line, but one can easily see that it remains bounded, when $T$ is small enough.
\end{proof}

\subsection{The average It\^{o} formula}
In this section we prove a version of the It\^{o} formula which we need to make the connection between \eqref{RoughFokkerPlank} and \eqref{MeanFieldRoughSDE}. We note that at the present level of knowledge, we don't know how to make an $P$-a.s. It\^{o} formula, but we only have the chain rule when we average over $\Omega$. 

\begin{proposition} \label{prop: adapted solutions}
	Let $(\Omega, \mathcal{F}, (\mathcal{F}_t)_{t \in [0,T]}, P)$ a complete filtered probability space and $W$ be a $d$-dimensional Wiener process on it. Let $\bar \alpha \in (\frac{1}{3}, \frac{1}{2})$, $\bZ \in \mathscr{C}_{wg}^{\bar\alpha}([0,T], \R^m)$.
	Assume that $\sigma$ and $\beta$ satisfy Assumption \ref{asm: building rough drivers}, for $k > \frac{d}{2}+3$ and $\alpha \in (\frac{1}{3},\bar\alpha)$. Let $\bF$ be defined as in Lemma \ref{lem: rough drivers mixed}.
	
	Let $x(\xi)$ be the solution to equation \eqref{eq: non linear equation} driven by $\bF$ with initial condition $\xi \in \R^d$, in the sense of Definition \ref{def:RDSolution}, given by Proposition \ref{thm: wellposedness}.
	
	Let $\Xi :\Omega \to \R^d$ be an $\mathcal{F}_0$-measurable random variable. Then the process $x_t(\Xi)$ is $(\mathcal{F}_t)_{t\geq 0}$ adapted. Moreover, $x$ is a random variable with values in $C^{\alpha}([0,T];\R^d)$.

\end{proposition}
\begin{proof}
	Let $t\in [0,T]$ and call $\bF_{|_{[0,t]}}$ the restriction of $\bF$ on the interval $[0,t]$.
	We know from Proposition \ref{ContractionComposition} that 
	\begin{equation*}
	\R^d \times (C^{\alpha}([0,t]; C_b^3(\R^d;\R^d)) \times C_2^{ 2 \alpha}([0,t]; C_b^2(\R^d \times \R^d;\R^d))) \to \R^d,
	\qquad
	(\xi, \bF_{|_{[0,t]}}) \mapsto x_t,
	\end{equation*}
	is a continuous mapping. Moreover the random variable $(\Xi, \bF_{|_{[0,t]}})$ is $\mathcal{F}_t$-measurable. Hence,
	\begin{equation*}
	\omega \mapsto (\Xi, \bF_{|_{[0,t]}})(\omega) \mapsto x_t(\Xi)(\omega),
	\end{equation*}
	is $\mathcal{F}_t$-measurable. 
	
	In a similar way we see that $x$ is a random variable in $C^{\alpha}([0,T];\R^d)$, since $\omega \mapsto \bF(\omega)$ is measurable and $x$ is continuous w.r.t. the rough driver. 
\end{proof}

\begin{proposition} \label{prop:ItoFormula}
	Under the same assumptions as Proposition \ref{prop: adapted solutions}, let $x_t = x_t(\Xi)$.
	If $\phi \in C_b^3\otimes H^k$, endowed with the norm defined in \eqref{def: tensor norm}, then
	$$
	E[  \phi(x_t) ]  = E[\phi(\Xi)] + \int_0^t \frac12 E[ \nabla_1^2 \phi(x_r) (\sigma_r(x_r) \sigma_r(x_r)^T) ]dr + \int_0^t E[ \nabla_1\phi(x_r)\beta_r(x_r) ] d\Z_r 
	\in H^k,
	$$
	where $E[\nabla_1\phi(x_r) \beta_r(x_r) ] \in  \mathcal{L}(\R^{m}; H^k)$ is controlled by $Z$ with Gubinelli derivative $E[\nabla_1\phi(x_r)( \beta_r'(x_r) + \nabla_1 \beta_r(x_r) \beta_r(x_r)) + \nabla_1^2\phi(x_r) \beta_r(x_r) \otimes \beta_r(x_r)]$.
\end{proposition}

Before we proceed with the proof of Proposition \ref{prop:ItoFormula}, we prove two technical lemmas.
\begin{lemma} \label{lemma:AverageSharp}
Under the same assumptions as Proposition \ref{prop: adapted solutions}, let $x^\sharp$ be defined in \eqref{def: x sharp}. For any $\rho \in \N$ and $|t-s| \leq h \leq T$, we have
$$
E[|x_{st}^{\sharp}|^{\rho} ] \leq C (h^{\rho} \vee h^{\frac{\rho}{2}})^{\bar\alpha - \alpha}  |t-s|^{ 2 \alpha \rho } L^{\rho},
$$
where $L:=L(\sigma, \beta, \bZ)$ is defined in \eqref{LConstant}.
\end{lemma}

\begin{proof}
Define the random variable $Y := C\|\bF\|_{\alpha,h;C^3}$ as in Proposition \ref{prop:APriori} which gives that for $|t-s|^{\alpha} \leq Y^{-1}$ we have 
$
|x_{st}^{\sharp}| \leq Y|t-s|^{2 \alpha} .
$
Writing $\Omega = \{ |t-s|^{\alpha} Y > 1 \} \cup \{ |t-s|^{\alpha}  Y \leq 1 \}$ gives
\begin{align*}
E[ |x_{st}^{\sharp}|^\rho] & = E[ 1_{ |t-s|^{ \alpha} Y > 1 } |x_{st}^{\sharp}|^\rho] + E[ 1_{|t-s|^{\alpha}  Y \leq 1} |x_{st}^{\sharp}|^\rho]  \leq |t-s|^{\rho\alpha} E[ Y^\rho |x_{st}^{\sharp}|^\rho] + E[Y^\rho] |t-s|^{\rho 2 \alpha} \\
& \leq |t-s|^{\rho\alpha} E[ Y^{2\rho}]^{1/2} E[|x_{st}^{\sharp}|^{2\rho}]^{1/2} + E[Y^\rho] |t-s|^{\rho 2 \alpha}.
\end{align*}
Now trivially by the definition of $x^{\sharp}$, we have
$$
|x^{\sharp}_{st}| \leq \left( [x]_{\alpha,h} + \|F\|_{\alpha,h;C} \right) |t-s|^{\alpha} \leq C \|\bF\|_{\alpha,h; C^3} |t-s|^{\alpha},
\qquad
P-a.s.
$$
and the result follows from Lemma \ref{lem: rough drivers mixed}.
\end{proof}

\begin{lemma} \label{lemma:ItoFormula}
Under the same assumptions as Proposition \ref{prop: adapted solutions}, we have
$$
( E[  \phi(x) ], E[ \nabla_1\phi(x)\beta(x) ]) \in \mathscr{D}_Z^{2 \alpha}([0,T]; H^k),
$$
with bounds, on a time interval of size $h\leq T$,
$$
\|( E[  \phi(x) ], E[ \nabla_1\phi(x)\beta(x) ])\|_{Z,\alpha,h; H^k} 
\leq L (h \vee h^{\frac12})^{\bar\alpha - \alpha} \|\phi\|_{C_b^3\otimes H^k} (1 + \| [x]_{\alpha,h} \|_{L_{\omega}^2}),
$$
where $L:=L(\sigma, \beta, \bZ)$ is defined in \eqref{LConstant}.
\end{lemma}

\begin{proof}
We do a first order Taylor expansion to obtain
\begin{align*}
\delta \phi(x)_{st} & = [\nabla_1 \phi]^{1,x}_{st} \delta x_{st} = [\nabla_1 \phi]^{1,x}_{st} W^{\sigma}_{st}(x_s) + [\nabla_1 \phi]^{1,x}_{st} Z_{st}^{\beta}(x_s)  +[\nabla_1 \phi]^{1,x}_{st} x_{st}^{\sharp} \\
& = \nabla_1 \phi(x_s) \beta_s^j(x_s) Z^j_{st} + \phi(x)_{st}^{\sharp},
\qquad P-a.s.
\end{align*}
We have defined 
\begin{align*}
\phi(x)_{st}^{\sharp} & : = [\nabla_1 \phi]^{1,x}_{st} W^{\sigma}_{st}(x_s) + [\nabla_1 \phi]^{1,x}_{st} x_{st}^{\sharp} + [\nabla_1 \phi]^{1,x}_{st} Z_{st}^{\beta}(x_s) - \nabla_1 \phi(x_s) \beta^j_s(x_s) Z^j_{st}.
\end{align*}
We first make some deterministic bounds (i.e. uniformly in $\omega$)
\begin{align*}
\| [\nabla_1 \phi]^{1,x}_{st} Z_{st}^{\beta}(x_s) & - \nabla_1 \phi(x_s) \beta_s^j(x_s) Z^j_{st} \|_{H^k}  \\
&  \leq \| ([\nabla_1 \phi]^{1,x}_{st} -  \nabla_1 \phi(x_s))Z_{st}^{\beta}(x_s)\|_{H^k}  + \|\nabla_1 \phi(x_s)) ( Z_{st}^{\beta}(x_s) - \beta_s^j(x_s) ) Z^j_{st} \|_{H^k} \\
 & \leq \|\phi\|_{C_b^3\otimes H^k} [x]_{\alpha,h} [Z^{\beta}]_{\alpha,h;C_b} |t-s|^{2 \alpha} + \|\phi\|_{C_b^3\otimes H^k} [ (Z^{\beta})^{\sharp}]_{2 \alpha,h} |t-s|^{2 \alpha} \\
  & \leq \|\phi\|_{C_b^3\otimes H^k} \| \bZ^{\beta}\|_{\alpha,h;C_b^3} (1 + [x]_{\alpha,h}) |t-s|^{2 \alpha}. 
\end{align*}
Using that $x$ is adapted we get $E[ \nabla_1 \phi(x_s) W_{st}^{\sigma}(x_s)] = 0$ so that 
\begin{align*}
\| E[ [\nabla_1 \phi]^{1,x}_{st} W^{\sigma}_{st}(x_s)]\|_{H^k} & = \| E[ [\nabla_1^2 \phi]^{2,x}_{st} \delta x_{st} W^{\sigma}_{st}(x_s)] \|_{H^k} \leq \| \phi\|_{C_b^3\otimes H^k} \|[x]_{\alpha,h}\|_{L^2_{\omega}} \| \sigma\|_{L_t^\infty \mathcal{L}(\R^m; H^k)} |t-s|^{\alpha + \frac12} \\
& \leq h^{\frac12 - \alpha} \| \phi\|_{C_b^3\otimes H^k} \|[x]_{\alpha,h}\|_{L^2_{\omega}} \| \sigma\|_{L_t^\infty \mathcal{L}(\R^m; H^k)}  |t-s|^{2 \alpha }.
\end{align*}
Write now
\begin{align*}
\| E[ [\nabla_1 \phi]^{1,x}_{st} x_{st}^{\sharp} ]\|_{H^k} \leq \| \phi\|_{C_b^3\otimes H^k} E[  |x_{st}^{\sharp}| ]  , 
\end{align*}
and the result follows from Lemma \ref{lemma:AverageSharp} with $\rho=1$.
\end{proof}

\begin{proof}[Proof of Proposition \ref{prop:ItoFormula}]
We do a third order Taylor expansion to obtain, $P$-a.s.,
\begin{align*}
\delta \phi(x)_{st}  
= & \nabla_1 \phi(x_s) \delta x_{st} 
+ \frac12 \nabla_1^2 \phi(x_s) (\delta x_{st})^2 
+ [\nabla_1^3 \phi]^{3, x}_{st} ( \delta x_{st})^{ 3} \\
= & \nabla_1 \phi(x_s) W_{st}^{\sigma}(x_s) 
+ \nabla_1 \phi(x_s) \WW^{\sigma}_{st}(x_s) 
+ \nabla_1 \phi(x_s) Z_{st}^{\beta}(x_s) 
+ \nabla_1 \phi(x_s) \ZZ_{st}^{\beta}(x_s)  \\
 & +  \nabla_1 \phi(x_s) \int_s^t  Z_{sr}^{\beta}(x_s) \nabla \sigma_r (x_s)dW_r  
 + \nabla_1 \phi(x_s) \int_s^t  W_{sr}^{\sigma}(x_s) \nabla \beta_r(x_s) d\bZ_r \\
 & + \frac12 \nabla_1^2  \phi(x_s) (W_{st}^{\sigma}(x_s))^{\otimes 2} 
 + \frac12 \nabla_1^2  \phi(x_s) (x_{st}^{\sharp})^{\otimes 2} 
 + \nabla_1^2  \phi(x_s) (W_{st}^{\sigma}(x_s) \otimes x_{st}^{\sharp}) \\
 & 
 +  \nabla_1^2  \phi(x_s) (Z_{st}^{\beta}(x_s) \otimes x_{st}^{\sharp}) 
 + \frac12 \nabla_1^2  \phi(x_s) (Z_{st}^{\beta}(x_s) \otimes Z_{st}^{\beta}(x_s)) \\
 & +  \nabla_1 \phi(x_s) x_{st}^{\natural} 
 + [\nabla_1^3 \phi]^{3, x}_{st} ( \delta x_{st})^{\otimes 3} \\
 =  & \frac12 \int_s^t  \nabla_1^2 \phi(x_r)  (\sigma_r(x_r) \sigma_r(x_r)^T) dr  
 +   \nabla_1 \phi(x_s) \beta_{s}^j(x_s)Z^j_{st} \\
  & + \left( \nabla_1 \phi(x_s) ( \beta_s^{j,i}(x_s) 
  + \nabla_1 \beta_s^j(x_s) \beta_s^i(x_s))  
  + \nabla_1^2 \phi(x_s) (\beta^j(x_s) \otimes \beta^i(x_s)) \right)\ZZ^{i,j}_{st} 
  + \phi(x)_{st}^{\natural}.
\end{align*}
Where we have defined 
\begin{align*}
\phi(x)_{st}^{\natural} 
: = &\nabla_1 \phi(x_s) W_{st}^{\sigma}(x_s) 
+ \nabla_1 \phi(x_s) \WW^{\sigma}_{st}(x_s) 
+   \nabla_1\phi(x_s) \int_s^t  Z_{sr}^{\beta}(x_s) \nabla \sigma_r (x_s)dW_r  \\
& + \nabla_1 \phi(x_s) \int_s^t W_{sr}^{\sigma}(x_s) \nabla \beta_r(x_s) d\bZ_r  
+ \frac12 \nabla_1^2 \phi(x_s) \int_s^t \sigma_r(x_s) \sigma_r(x_s)^T dr  
- \frac12 \nabla_1^2  \phi(x_s) (W_{st}^{\sigma}(x_s))^{\otimes 2} \\
 & + \frac12 \int_s^t (\nabla_1^2 \phi(x_r) - \nabla_1^2\phi(x_s) )( \sigma_r(x_s) \sigma_r(x_s)^T) dr \\
 & + \frac12 \int_s^t  \nabla_1^2 \phi(x_r)  (\sigma_r(x_r) \sigma_r(x_r)^T 
 - \sigma_r(x_s) \sigma_r(x_s)^T) dr \\
& +\nabla_1 \phi(x_s) Z_{st}^{\beta}(x_s) 
- \nabla_1 \phi(x_s) \beta_{s}^j(x_s)Z^j_{st} \\
& + \nabla_1 \phi(x_s) \ZZ_{st}^{\beta}(x_s)  
- \nabla_1 \phi(x_s) ( \beta_s^{j,i}(x_s) 
+ \nabla \beta_s^j(x_s) \beta_s^i(x_s)) \ZZ^{i,j}_{st} \\ 
& + \frac12 \nabla_1^2  \phi(x_s) (Z_{st}^{\beta}(x_s) \otimes Z_{st}^{\beta}(x_s)) 
- \nabla_1^2 \phi(x_s) (\beta_s^j(x_s) \otimes \beta_s^i(x_s)) \ZZ^{i,j}_{st}\\
& +  \frac12 \nabla_1^2  \phi(x_s) (x_{st}^{\sharp})^{\otimes 2} 
+ \nabla_1^2  \phi(x_s) (W_{st}^{\sigma}(x_s) \otimes x_{st}^{\sharp}) \\
& 
+  \nabla_1^2  \phi(x_s) (Z_{st}^{\beta}(x_s) \otimes x_{st}^{\sharp}) 
+ \frac12 \nabla_1^2  \phi(x_s) (Z_{st}^{\beta}(x_s) \otimes Z_{st}^{\beta}(x_s)) \\
& +  \nabla_1 \phi(x_s) x_{st}^{\natural} 
+ [\nabla_1^3 \phi]^{3, x}_{st} ( \delta x_{st})^{\otimes 3}.
\end{align*}
As in Lemma \ref{lemma:NonLinearVsClassical} we note that 
\begin{align*}
\nabla_1 \phi(x_s) Z_{st}^{\beta}(x_s) - \nabla_1 \phi(x_s) \beta^j_{s}(x_s)Z^j_{st} & + \nabla_1 \phi(x_s) \ZZ_{st}^{\beta}(x_s)  - \nabla_1 \phi(x_s) ( \beta_s^{j,i}(x_s) + \nabla \beta^j_s(x_s) \beta^i_s(x_s)) \ZZ^{i,j}_{st}  
\end{align*}
is uniformly in $\omega$ bounded by $|t-s|^{3 \alpha}$ depending only on $\beta$.
Moreover, since $\bZ$ is geometric and $\nabla^2 \phi$ is a symmetric bilinear mapping we get
\begin{align*}
 \frac12 \nabla_1^2  \phi(x_s) (Z_{st}^{\beta}(x_s) & \otimes Z_{st}^{\beta}(x_s))  
 - \nabla_1^2 \phi(x_s) \beta(x_s) \otimes \beta(x_s) \ZZ_{st} \\
  & = \frac12 \nabla_1^2  \phi(x_s) \left( ( Z_{st}^{\beta})^{\sharp}(x_s) \otimes (Z_{st}^{\beta})^{\sharp}(x_s))  \right) 
  + \nabla_1^2  \phi(x_s) ( Z_{st}^{\beta})^{\sharp}(x_s) \st \beta_{s}(x_s) Z_{st})
\end{align*}
where $\st$ denotes the symmetric tensor product. This is clearly bounded by $|t-s|^{3 \alpha}$.

Using Lemma \ref{lemma:AverageSharp} with $\rho=1$ and $\rho=2$ 
and taking the expectation of $\phi(x)^{\natural}_{st}$ we obtain the result. 
\end{proof}

To create the contraction mapping in the appropriate space of measures we shall need to control the difference of two measures induced by two rough SDEs.

\begin{proposition}
	\label{pro: contractivity ito formula}
	Let $(\Omega, \mathcal{F}, (\mathcal{F}_t)_{t \in [0,T]}, P)$ a complete filtered probability space and $W$ be a $d$-dimensional Wiener process on it. Let $\bar \alpha \in (\frac{1}{3}, \frac{1}{2})$, $\bZ \in \mathscr{C}_{wg}^{\bar\alpha}([0,T], \R^m)$.
	Assume that $(\sigma, \beta)$ and $(\theta, \gamma)$ satisfy Assumption \ref{asm: building rough drivers}, for $k > \frac{d}{2}+3$, $\alpha \in (\frac{1}{3},\bar\alpha)$ and $p=\frac{1}{\alpha}$.  Let $\bF$ and $\bG$ be nonlinear rough drivers constructed from $F_{st}:= W_{st}^{\sigma} + Z_{st}^{\beta}$ and $G_{st} := W_{st}^{\theta} + Z_{st}^{\gamma}$ as in Lemma \ref{lem: rough drivers mixed}. Moreover, let $\Xi$ be an $\mathcal{F}_0$-measurable random variable.
	
	Let $x$ and $y$ solutions to equation \eqref{eq: non linear equation} driven by $\bF$ and $\bG$ respectively, with the same initial condition $\Xi$.
	
	If $\phi \in C_b^3\otimes H^k$, endowed with the norm defined in \eqref{def: tensor norm}, we have
	$$
	( E[ \phi(x) - \phi(y)], E[ \nabla_1 \phi(x)\beta(x) - \nabla_1 \phi(y) \gamma(y)] ) \in \mathscr{D}_Z^{2 \alpha}([0,T];H^k).
	$$
	Moreover, there exists $\rho \geq 1$ and $C(T)$ such that $\lim_{T\to 0}C(T) = 0$, and
	\begin{align*}
	\| & (E[ \phi(x) - \phi(y)], E[ \nabla_1 \phi(x)\beta(x)]) - (E[\phi(y)], E[\nabla_1 \phi(y) \gamma(y)]) \|_{Z,\alpha;H^k}
	\\
	& \leq C(T) e^{M^{\bar\rho} } \| \phi \|_{C_b^3\otimes H^k}
	\left(
	\|\sigma - \theta\|_{L_t^\infty \mathcal{L}(\R^d;H^k)} 
	+ \|(\beta, \beta') - (\gamma, \gamma')\|_{Z, \alpha;H^k}  
	\right).
	\end{align*}
	where $M:=K(\la \sigma\ra_{p,[s,t]}^{\rho}+1)(L(\sigma, \beta, \bZ) + L(\theta, \gamma, \bZ))$, $L$ is defined in \eqref{LConstant} and $K = K(\alpha, \rho)>0$ is a universal constant.
\end{proposition}

Before proceeding with the proof, we need the next two technical lemmas.
\begin{lemma}
	\label{lem: ito formula contraction solutions}
	Under the same assumptions of Proposition \ref{pro: contractivity ito formula}, for any $\rho\geq 1$, there exists $\bar\rho \geq \rho$, $C$ and $C(T) > 0$, such that $\lim_{T\to 0}C(T) = 0$ and
	\begin{equation*}
	\|[x-y]_{\alpha}\|_{L_{\omega}^{\rho}} 
	\leq C(T) e^{M^{\bar\rho} }
	\left(
	\|\sigma - \theta\|_{L_t^\infty \mathcal{L}(\R^d;H^k)} 
	+ \|(\beta, \beta') - (\gamma, \gamma')\|_{Z, \alpha;H^k}  
	\right).
	\end{equation*}
\end{lemma}
\begin{proof}
	By applying Corollary \ref{cor: Contraction Composition Global}, \eqref{eq: contractive rough drivers} and \eqref{eq: FF bound}, we see that there exists $\bar \rho \geq 1$ and $K_{\bar \rho} \in L^{\bar \rho}(\Omega)$ such that $P$-a.s.,
	\begin{align*} 
	[x - y]_{\alpha} \leq &
	C e^{C N(w_{\bF},[0,T])}
	[\bF- \bG]_{\alpha}
	(  1 +  [\bF]_{\alpha} +  [\bG]_{\alpha}  )^2
	(( [\bF]_{\alpha} + [\bG]_{\alpha}) \vee ( [\bF]_{\alpha} + [\bG]_{\alpha})^{\frac{1}{\alpha}})\\
	\leq &
	C e^{C N(w_{\bF},[0,T])}
	T^{\bar\alpha - \alpha} K_\rho M^\rho
	(\| \sigma - \theta \|_{L_t^\infty H^k} 
	+
	\| (\beta, \beta^\prime) - (\gamma, \gamma^\prime)\|_{Z, \alpha; H^k}).
	\end{align*}	
	Taking the $L_\omega^\rho$ norm on both sides we conclude the proof, thanks to Theorem \ref{thm: weibull tail}, which gives
	\begin{equation*}
	Ee^{CN(w_{\bF},[0,T])} 
	\leq e^{C (\la \sigma\ra_{p}^{p}+1)L(\sigma, \beta, \bZ)^{p}},
	\end{equation*}
	where $C>0$ is a universal constant.
\end{proof}

\begin{lemma}
	\label{lem: ito formula contraction remainders}
Under the same assumptions of Proposition \ref{pro: contractivity ito formula}, for any $\rho\geq 1$, there exists $\bar\rho \geq \rho$ and $C(T) > 0$, such that $\lim_{T\to 0}C(T) = 0$ and, for all $s,t \in [0,T]$,
$$
\|x^{\sharp}_{st} - y^{\sharp}_{st} \|_{L_\omega^\rho}
\leq C(T) e^{M^{\bar\rho} }
(\|\sigma - \theta\|_{L_t^{\infty} \mathcal{L}(\R^m,H^k)} + \|(\beta, \beta') - (\gamma, \gamma')\|_{Z, \alpha;H^k})|t-s|^{2\alpha}.
$$
\end{lemma}

\begin{proof}
Let $Y := C([\bF]_{\alpha} + [\bG]_{\alpha})(1 + [\bF]_{\alpha} + [\bG]_{\alpha})^2$ where $C$ is the constant given in Proposition \ref{ContractionComposition}. Then, $|t-s|^{\alpha} Y \leq 1$ implies,
$$
|x^{\sharp}_{st} - y_{st}^{\sharp}| \leq |t-s|^{2 \alpha}  Y  e^{CN( w_{\bF},[0,T])}[\bF - \bG]_{\alpha}
$$
and we notice that $E[Y^\rho] \leq  M^{\bar\rho} (T^\rho \vee T^{\frac{\rho}{2}})^{\bar{\alpha} - \alpha}$ for some $\bar\rho \geq \rho \geq 1$ which follows from Lemma \ref{lem: rough drivers mixed} and the Gaussian integrability of $N( w_{\bF},[0,T])$, Theorem \ref{thm: weibull tail}.

We split up $\Omega = \{ |t-s|^{\alpha} Y \leq 1 \} \cup \{ |t-s|^{\alpha} Y > 1 \}$ which gives
\begin{align*}
E[|x^{\sharp}_{st} - y_{st}^{\sharp}|^\rho] \leq |t-s|^{ 2\rho \alpha} E[ Y^{2\rho}  |x^{\sharp}_{st} - y_{st}^{\sharp}|^\rho ] + |t-s|^{2 \rho \alpha} E[ Y^{\rho}  e^{CN( w_{\bF},[0,T])}[\bF - \bG]_{\alpha ,h}^{\rho}] .
\end{align*}
For the first term above we use the crude (in time) bound
\begin{align*}
|x_{st}^{\sharp} - y^{\sharp}_{st}| & \leq [x-y]_{\alpha,h} |t-s|^{\alpha}  + [\bF]_{\alpha,h} |x_s - y_s||t-s|^{\alpha}  + [\bF - \bG]_{\alpha,h}|t-s|^{\alpha} \\
& \leq C(T)Y  e^{CN( w_{\bF},[0,T])} [\bF - \bG]_{\alpha} .
\end{align*}
The result follows from Corollary \ref{cor: Contraction Composition Global}, \eqref{eq:GronwallEstimate} and Theorem \ref{thm: weibull tail}.
\end{proof}

\begin{proof}[Proof of Proposition \ref{pro: contractivity ito formula}]
We write
\begin{align}
\phi(x)_{st}^{\sharp} - \phi(y)_{st}^{\sharp} & : = [\nabla_1 \phi]^{1,x}_{st} W^{\sigma}_{st}(x_s)  - [\nabla_1 \phi]^{1,y}_{st} W^{\theta}_{st}(y_s) + [\nabla_1 \phi]^{1,x}_{st} x_{st}^{\sharp} -[\nabla_1 \phi]^{1,y}_{st} y_{st}^{\sharp} \label{eq: contract into main}\\
& + [\nabla_1 \phi]^{1,x}_{st} Z_{st}^{\beta}(x_s) - \nabla_1 \phi(x_s) \beta_s^j(x_s) Z^j_{st}  - [\nabla_1 \phi]^{1,y}_{st} Z_{st}^{\gamma}(y_s) + \nabla_1 \phi(y_s) \gamma_s^j(y_s) Z^j_{st}.\nonumber
\end{align}
We start from the first term on the right hand side of \eqref{eq: contract into main}, 
\begin{align*}
[\nabla_1 \phi]^{1,x}_{st} W^{\sigma}_{st}(x_s)  &
- [\nabla_1 \phi]^{1,y}_{st} W^{\theta}_{st}(y_s)   \\ 
 = & [\nabla_1^2 \phi]^{2,x}_{st}\delta x_{st} W^{\sigma}_{st}(x_s)  - [\nabla_1^2 \phi]^{2,y}_{st} \delta y_{st} W^{\theta}_{st}(y_s) +  \nabla_1 \phi(x_s) W_{st}^{\sigma}(x_s) 
-  \nabla_1 \phi(y_s) W_{st}^{\theta}(y_s)\\
= & ([\nabla_1^2 \phi]^{2,x}_{st} - [\nabla_1^2 \phi]^{2,y}_{st}) \delta x_{st} W^{\sigma}_{st}(x_s)
+ [\nabla_1^2 \phi]^{2,y}_{st} (\delta x_{st} - \delta y_{st}) W^{\sigma}_{st}(x_s)\\
& + [\nabla_1^2 \phi]^{2,y}_{st} \delta y_{st} (W^{\sigma}_{st}(x_s) - W^{\theta}_{st}(x_s))
+ [\nabla_1^2 \phi]^{2,y}_{st} \delta y_{st} (W^{\theta}_{st}(x_s) - W^{\theta}_{st}(y_s)).
\end{align*}
We have, as an application of H\"older inequality, for $0\leq s\leq t\leq T$,
\begin{align*}
\| E&[([\nabla_1^2 \phi]^{2,x}_{st} - [\nabla_1^2 \phi]^{2,y}_{st}) \delta x_{st} W^{\sigma}_{st}(x_s)] \|_{H^k}
\leq \| \phi \|_{C_b^3\otimes H^k} E[\| x - y\|_{L_t^\infty}\;|\delta x_{st}|\;| W^{\sigma}_{st}(x_s)|]\\
& \leq \| \phi \|_{C_b^3\otimes H^k} \| x - y\|_{L_\omega^2L_t^\infty}\|\delta x_{st}\|_{L_\omega^4} \|W^{\sigma}_{st}(x_s)\|_{L_\omega^4} \\
& \leq C(T)\| \phi \|_{C_b^3\otimes H^k} \| x - y\|_{L_\omega^2L_t^\infty} \|[x]_{\alpha}\|_{L_\omega^4} \|\sigma\|_{L_t^\infty \mathcal{L}(\R^d;H^k)} |t-s|^{2\alpha}.
\end{align*}
where, in the last inequality we used Lemma \ref{lem: rough drivers from Ito}. Similarly, using Lemma \ref{lem: rough drivers from Ito} and \ref{lem:ItoContraction}, we can bound the remaining terms,
\begin{align*}
\| E[[\nabla_1^2 \phi]^{2,y}_{st} (\delta x_{st} - \delta y_{st}) W^{\sigma}_{st}(x_s)]\|_{H^k}
\leq &C(T) \| \phi \|_{C_b^3\otimes H^k} \|[x-y]_{\alpha}\|_{L_\omega^2} \|\sigma\|_{L_t^\infty \mathcal{L}(\R^d;H^k)} |t-s|^{2\alpha},\\
\| E[[\nabla_1^2 \phi]^{2,y}_{st} \delta y_{st} (W^{\sigma}_{st}(x_s) - W^{\theta}_{st}(x_s))] \|_{H^k}
\leq & C(T) \| \phi \|_{C_b^3\otimes H^k}  \|[y]_{\alpha}\|_{L_\omega^2} (1 + \|\sigma\|_{L_t^\infty \mathcal{L}(\R^d;H^k)} + \|\theta\|_{L_t^\infty \mathcal{L}(\R^d;H^k)})  \\
& \cdot \|\sigma - \theta\|_{L_t^\infty \mathcal{L}(\R^d;H^k)} |t-s|^{2\alpha}.\\
\| E[[\nabla_1^2 \phi]^{2,y}_{st} \delta y_{st} (W^{\theta}_{st}(x_s) - W^{\theta}_{st}(y_s)) ]\|_{H^k}
\leq & C(T)\| \phi \|_{C_b^3\otimes H^k} \|x-y\|_{L_\omega^2L_t^\infty} \|[y]_{\alpha}\|_{L_\omega^4} \|\sigma\|_{L_t^\infty \mathcal{L}(\R^d;H^k)} |t-s|^{2\alpha}.
\end{align*}
Summing up the previous inequalities, we get
\begin{align*}
\| E[[\nabla_1 \phi]^{1,x}_{st} W^{\sigma}_{st}(x_s) - [\nabla_1 \phi]^{1,y}_{st} W^{\theta}_{st}(y_s) ]\|_{H^k}
\leq & M C(T) \| \phi \|_{C_b^3\otimes H^k} |t-s|^{2\alpha}
(1+\|[x]_{\alpha}\|_{L_\omega^4}+\|[y]_{\alpha}\|_{L_\omega^4})\\
& \cdot \left( \|x-y\|_{L_\omega^2L_t^\infty} + \|[x-y]_{\alpha}\|_{L_\omega^2} + \|\sigma - \theta\|_{L_t^\infty \mathcal{L}(\R^d;H^k)}  \right).
\end{align*}
The second term in \eqref{eq: contract into main} is bounded as follows using Lemmas \ref{lem: ito formula contraction remainders} and \ref{lemma:AverageSharp},
\begin{align*}
\|E[ &[\nabla_1 \phi]^{1,x}_{st} x_{st}^{\sharp} -[\nabla_1 \phi]^{1,y}_{st} y_{st}^{\sharp} ] \|_{H^k}
\leq 
\Vert \phi \Vert_{C_b^3\otimes H^k}
\left[
\| x_{st}^{\sharp}\|_{L_{\omega}^2} \|x - y\|_{L_{\omega}^2L_t^\infty} 
+ \| x_{st}^{\sharp} - y_{st}^{\sharp}\|_{L_{\omega}^1}
\right]\\
\leq & C(T) \Vert \phi \Vert_{C_b^3\otimes H^k} e^{M^{\bar\rho} } \left(  \|x - y\|_{L_{\omega}^2L_t^\infty} + \|\sigma - \theta\|_{L_t^{\infty} \mathcal{L}(\R^m,H^k)} + \|(\beta, \beta') - (\gamma, \gamma')\|_{Z, \alpha;H^k}\right) |t-s|^{ 2 \alpha }.
\end{align*}
The third term in \eqref{eq: contract into main} is
\begin{align}
& [\nabla_1 \phi]^{1,x}_{st} Z_{st}^{\beta}(x_s) - \nabla_1 \phi(x_s) \beta_s^j(x_s) Z^j_{st}  - [\nabla_1 \phi]^{1,y}_{st} Z_{st}^{\gamma}(y_s) + \nabla_1 \phi(y_s) \gamma^j_s(y_s) Z^j_{st} \nonumber \\
& = 
[\nabla_1^2 \phi]^{2,x}_{st} \delta x_{st} Z_{st}^{\beta}(x_s) 
- [\nabla_1^2 \phi]^{2,y}_{st} \delta y_{st} Z_{st}^{\gamma}(y_s) 
+
\nabla_1 \phi(x_s) (Z^\beta_{st})^{\sharp}(x_s) - \nabla_1 \phi(y_s) (Z^\gamma_{st})^{\sharp}(y_s) \nonumber
\\
& = ([\nabla_1^2 \phi]^{2,x}_{st} - [\nabla_1^2 \phi]^{2,y}_{st})\delta x_{st} Z_{st}^{\beta}(x_s) 
+ [\nabla_1^2 \phi]^{2,y}_{st} (\delta x_{st} - \delta y_{st}) Z_{st}^{\beta}(x_s)
+ [\nabla_1^2 \phi]^{2,y}_{st} \delta y_{st} (Z_{st}^{\beta}(x_s) - Z_{st}^{\beta}(y_s))\nonumber\\
& + [\nabla_1^2 \phi]^{2,y}_{st} \delta y_{st} (Z_{st}^{\beta}(y_s) - Z_{st}^{\gamma}(y_s))
+
\nabla_1 \phi(x_s) (Z^\beta_{st})^{\sharp}(x_s) - \nabla_1 \phi(y_s) (Z^\gamma_{st})^{\sharp}(y_s) 
\label{eq: ito formula contraction z terms}.
\end{align}
We estimate the first term in the right hand side using Lemma \ref{lem: rough drivers from Gubinelli},
\begin{align*}
\|E&[
([\nabla_1^2 \phi]^{2,x}_{st} 
- [\nabla_1^2 \phi]^{2,y}_{st})\delta x_{st} Z_{st}^{\beta}(x_s) 
] \|_{H^k}
\leq \Vert \phi \Vert_{C_b^3\otimes H^k} \|x-y\|_{L_\omega^2L_t^\infty} \|\delta x_{st}\|_{L_\omega^4}\|Z_{st}^{\beta}(x_s) \|_{L_\omega^4}\\
&\leq C(T) \Vert \phi \Vert_{C_b^3\otimes H^k} \|x-y\|_{L_\omega^2L_t^\infty} \|[x]_{\alpha}\|_{L_\omega^4}\| (\beta, \beta^\prime) \|_{Z,\alpha;H^k}[\bZ]_{\bar\alpha} |t-s|^{2\alpha}.
\end{align*}
Similarly, using Lemma \ref{lem: rough drivers from Gubinelli} and \ref{lemma:RoughIntegralContraction},
\begin{align*}
\| E[[\nabla_1^2 \phi]^{2,y}_{st} (\delta x_{st} - \delta y_{st}) Z_{st}^{\beta}(x_s)] \|_{H^k}
\leq & C(T) \| \phi \|_{C_b^3\otimes H^k} \|[x-y]_{\alpha}\|_{L_\omega^2} \| (\beta, \beta^\prime) \|_{Z,\alpha;H^k}[\bZ]_{\bar\alpha} |t-s|^{2\alpha},\\
\|E[[\nabla_1^2 \phi]^{2,y}_{st} \delta y_{st} (Z_{st}^{\beta}(x_s) - Z_{st}^{\beta}(y_s))]\|_{H^k}
 \leq & C(T) \| \phi \|_{C_b^3\otimes H^k} \|[y]_{\alpha}\|_{L_\omega^4} \|x-y\|_{L_\omega^2L_t^\infty} \| (\beta, \beta^\prime) \|_{Z,\alpha;H^k}[\bZ]_{\bar\alpha} |t-s|^{2\alpha},\\
\| E[[\nabla_1^2 \phi]^{2,y}_{st} \delta y_{st} (Z_{st}^{\beta}(y_s) - Z_{st}^{\gamma}(y_s))]\|_{H^k}
 \leq &  M C(T) \| \phi \|_{C_b^3\otimes H^k} \|[y]_{\alpha}\|_{L_\omega^1} (\| (\beta, \beta^\prime) - (\gamma, \gamma^\prime)\|_{Z,\alpha;H^k}) [\bZ]_{\bar\alpha} |t-s|^{2\alpha}.
\end{align*}
We estimate the last term in \eqref{eq: ito formula contraction z terms} using equation \eqref{eq: hoelder sobolev zeta beta} and Lemma \ref{lemma:RoughIntegralContraction},
\begin{align*}
\|E[\nabla_1 \phi(x_s) (Z^\beta_{st})^{\sharp}(x_s) - \nabla_1 \phi(y_s) (Z^\gamma_{st})^{\sharp}(y_s) ] \|_{H^k}
\leq C M & \| \phi \|_{C_b^3\otimes H^k} C(T) [\bZ]_{\bar\alpha} \\
 & \cdot \left(
\| x-y\|_{L_\omega^2L_{t}^\infty} 
+ \| (\beta, \beta^\prime) - (\gamma, \gamma^\prime)\|_{Z,\alpha;H^k}
\right)
|t-s|^{2\alpha}.
\end{align*}
Thus, there exists $\rho\geq 1$ (which may increase from a line to the next) such that the remainder satisfies, for all $s,t \in [0,T]$,
\begin{align*}
\| E[\phi(x)_{st}^{\sharp} -  & \phi(y)_{st}^{\sharp}] \|_{H^k}
\leq  M^{\rho} C(T) \| \phi \|_{C_b^3\otimes H^k} |t-s|^{2\alpha}
(1+\|[x]_{\alpha}\|_{L_\omega^4}+\|[y]_{\alpha}\|_{L_\omega^4})\\
& \cdot \left(\|x-y\|_{L_\omega^2L_t^\infty} + \|[x-y]_{\alpha}\|_{L_\omega^2} + \|\sigma - \theta\|_{L_t^\infty \mathcal{L}(\R^d;H^k)} + \|(\beta, \beta') - (\gamma, \gamma')\|_{Z, \alpha;H^k}  \right)\\
\leq & M^{\rho} C(T) \| \phi \|_{C_b^3\otimes H^k} |t-s|^{2\alpha}
\left(\|[x-y]_{\alpha}\|_{L_\omega^2} + \|\sigma - \theta\|_{L_t^\infty \mathcal{L}(\R^d;H^k)} + \|(\beta, \beta') - (\gamma, \gamma')\|_{Z, \alpha;H^k}  \right).
\end{align*}
In the last inequality we used Lemma \ref{lem:APriori} combined with Lemma \ref{lem: rough drivers mixed}, and also $\|x\|_{L_{t}^\infty} \leq T^\alpha [x]_{\alpha} + |x_0|$.
We check now the Gubinelli derivative, for each $j$ we have
\begin{align*}
\delta & (\nabla_1\phi(x)\beta^j(x))_{st} - \delta (\nabla_1\phi(y))\gamma^j(y))_{st}
=  [\nabla_1^2\phi ]_{st}^{1x}\delta x_{st} \beta^j_t(x_t) - [\nabla_1^2 \phi ]_{st}^{1y}\delta y_{st} \gamma^j_t(y_t)\\
 &+ \nabla_1\phi(x_s)\beta^j_{st}(x_s) - \nabla_1\phi(y_s)\gamma^j_{st}(y_s)
 + \nabla_1\phi(x_s)\delta(\beta^j_{st}(x_\cdot))_{st} - \nabla_1\phi(y_s)\delta(\gamma^j_{st}(y_\cdot))_{st}.
\end{align*}
Similarly as for the remainder, we obtain the following,
\begin{align*}
\|E&[\delta (\nabla_1\phi(x)\beta^j(x))_{st} - \delta ((\nabla_1\phi(y))\gamma^j(y))_{st} ]\|_{H^k}\\
&\leq M^{\rho} C(T) \| \phi \|_{C_b^3\otimes H^k}
|t-s|^{\alpha}
\left(
\|[x-y]_{\alpha}\|_{L_\omega^2} 
+ \|\sigma - \theta\|_{L_t^\infty \mathcal{L}(\R^d;H^k)} 
+ \|(\beta, \beta') - (\gamma, \gamma')\|_{Z, \alpha;H^k}  
\right).
\end{align*}
We conclude by using Lemma \ref{lem: ito formula contraction solutions} to estimate $\|[x-y]_{\alpha}\|_{L_\omega^2} $.
\end{proof}

\section{Linear Rough PDE} \label{sec:LinearPDE}

Let $d, m \in \N$ be fixed and let $\bZ \in \mathscr{C}_g^{\bar\alpha}([0,T], \R^m)$, for $\bar \alpha \in (\frac{1}{3}, \frac{1}{2})$. Let $\sigma$ and $\beta$ satisfy Assumptions \ref{asm: building rough drivers}, for $k$ large enough.
In this section we prove well-posedness of measure-valued solutions to linear rough partial differential equations, which are formally given as
\begin{equation}
\label{eq: linear pde target}
\partial_t \nu_t = \frac12\textrm{Tr}  \nabla^2 (\sigma_t\sigma_t^T \nu_t) - \Div( \beta_t\dot{Z}_t \nu_t) ,  \quad \nu_0 \in \mathcal{P}(\R^d).
\end{equation}
To rigorously define the meaning of a solution to equation \eqref{eq: linear pde target}, we take a slightly more general approach, as described below.

\begin{assumption}
	\label{asm: linear pde}
	Let $n \in \N$ and $\alpha \in (\frac{1}{3}, \bar\alpha)$.
	\begin{enumerate}[label=(\roman*), ref=\ref{asm: building rough drivers} (\roman*)]
		\item 
		\label{asm: linear pde: a}
		Let $a: [0,T] \to C^{n+3}(\R^d;\R^{d\times d})$ be a measurable path such that $a_t^{i,j}(x) \xi^i \xi^j \geq 0$ for all $x, \xi \in \R^d$ and $t\in[0,T]$.
		\item 
		\label{asm: linear pde: X}
		Let $\bX \in \mathscr{C}_g^{\alpha}([0,T]; C^{n+3}_b(\R^d;\R^d))$ be a geometric rough path, as described in Section \ref{sec:Notations}.
	\end{enumerate}
\end{assumption}
The examples we have in mind are $a= \frac12 \sigma\sigma^T$ and $\bX = \int \beta_r d\bZ_r$, as described in Proposition \ref{pro: measure existence linear pde}. In order to describe the main ideas, we argue now on a formal level assuming smoothness in time of $X$; rigorous definitions in the rough path case will be given later in the section. We study uniqueness of solutions to the following linear equation
\begin{equation} \label{LinearPDEGeneral} 
\partial_t \nu_t = \textrm{Tr}\nabla^2 (a_t  \nu_t) + \Div( \dot{X}_t \nu_t) ,  \quad \nu_0 \in \mathcal{P}(\R^d).
\end{equation}

The proof is based on a backward duality trick; suppose we can show \emph{existence} of a sufficiently regular solution to the backward PDE
\begin{equation} \label{eq:BackwardEquation}
\partial_t u_t +   \Tr (a_t  \nabla^2 u_t)  = \dot{X}_t \nabla u_t  , 
\end{equation}
for a given final condition $u_T$, then at least formally we have
\begin{equation} \label{eq:FormalTesting}
\partial_t \nu_t(u_t) =  (\textrm{Tr} \nabla^2 (a_t \nu_t)(u_t) -  \nu_t(  \Tr (a_t \nabla^2 u_t))  + \Div( \dot{X}_t \nu_t) (u_t) +  \nu_t( \dot{X}_t \nabla u_t) = 0, 
\end{equation}
which shows that $\nu_T( u_T) = \nu_0( u_0)$. Now, if $u_T$ is chosen in a class of functions large enough to fully determine $\nu_T$, we see that it will be fully determined by $\nu_0$ and $u_0$, thus showing uniqueness.

For simplicity only, we write equation \eqref{eq:BackwardEquation} on divergence form and as a forward equation as follows
\begin{equation} \label{MainAppendixEquation}
\partial_t u_t =  \Div( a_t \nabla u_t) + \dot{X}_t  \nabla u_t , \qquad u_0 \textrm{ given},
\end{equation}
which can be seen to be equivalent to \eqref{eq:BackwardEquation} by replacing $X_t$ by $( \int_0^t \nabla a_r dr, X_t)$ in \eqref{MainAppendixEquation} and then reversing time, i.e. $u_t \mapsto u_{T-t}$.

The strategy to prove existence of a smooth solution to \eqref{MainAppendixEquation} is as follows. We first show how to give an intrinsic notion of solution of \eqref{LinearPDEGeneral} and \eqref{MainAppendixEquation} in the context of the so-called unbounded rough drivers, see \cite{BaGu15}. We then replace $X$ by smooth vector fields, in which case it is well know that there exists a unique solution of \eqref{MainAppendixEquation} which is smooth provided the coefficients are. 
We then consider the vector of derivatives $f = (u, \nabla u, \dots, \nabla^n u)$ and show that $f$ satisfies a vector valued equation, for which we can find bounds independent of $\dot{X}$. The equation for $f$ will be solved in the space $L^2(\R^d;\R^N)$, thus giving bounds on $u$ in the Sobolev-space $H^n(\R^d)$.

Second, we approximate $\bX$ by a sequence of smooth vector fields and show that the corresponding sequence of solutions converge to a meaningful solution of \eqref{MainAppendixEquation}. Since the solution is in $H^n(\R^d)$ we can use Sobolev embedding \cite[Corollary 9.13]{brezis2011} to show the needed spatial regularity to justify the computations in \eqref{eq:FormalTesting}.

The techniques used to prove the first step are motivated by \cite{BaGu15} and \cite{DeGuHoTi16}, and the main technical tool is the a priori estimate found in \cite{DeGuHoTi16}. 


\subsection{Unbounded rough drivers}
We start by rephrasing \eqref{MainAppendixEquation} in terms of so called unbounded rough drivers. The main motivation for doing so is the a priori estimate from \cite{DeGuHoTi16}.

Assume that $X$ is a smooth path, then equation \eqref{MainAppendixEquation} is well defined as a PDE. Integrating \eqref{MainAppendixEquation} from $s$ to $t$ we obtain
\begin{align*}
 \delta u_{st} & =  \int_s^t \Div( a_r \nabla u_r)dr  +  \int_s^t \dot{X}_r  \nabla u_r dr.
\end{align*}
Iterating the equation into itself we obtain
\begin{equation} \label{eq:Expansion}
\delta u_{st} = \int_s^t \Div( a_r \nabla u_r) dr  + B_{st}^1 u_s + B_{st}^2 u_s + u_{st}^{\natural}
\end{equation}
where at least formally,
\begin{equation} \label{eq:URDFormal}
B_{st}^1 \phi = X_{st}^j \partial_j \phi ,
 \qquad B_{st}^2 \phi = \int_s^t \dot{X}^j_r \partial_j  \int_s^r \dot{X}^i_u \partial_i \phi  du dr
 \end{equation}
and 
\begin{align*}
	u_{st}^{\natural} 
	& = \int_s^t \dot{X}^j_r \partial_j \int_s^{r} \dot{X}^i_{\tau} \partial_i \int_s^{\tau} \dot{X}^l_{\theta} \partial_l u_{\theta} d \theta d\tau dr \\
 	&  + \int_s^t \dot{X}^j_r \partial_j \int_s^r \Div( a_{\tau} \nabla u_{\tau})  d \tau  dr 
 	+ \int_s^t \dot{X}^j_r \partial_j \int_s^r \dot{X}^i_{\tau} \partial_i \int_s^{\tau} \Div( a_{\theta} \nabla u_{\theta}) d \theta d\tau  dr .
\end{align*}
By the usual power counting the remainder term $u^{\natural}$ should be regular in time, but we notice that in general it is a distribution in space. 
Following \cite{BaGu15} we call a scale of spaces a quadruple $(E^{n})_{ n =0}^3$ of Banach spaces such that $E^{n + 1}$ is continuously embedded into $E^{n}$. Let $E^{-n}$ be the topological dual of $E^{n}$ (in general, $E^{-0}\neq E^0$). 
\begin{definition}
\label{def:urd}
An unbounded $\alpha$-rough driver on the scale $(E^n)_{n}$, is a pair $\mathbf{B} = (B^1,B^2)$ of mappings on $E^n$ such that 
\begin{equation}\label{ineq:UBRcontrolestimates}
\| B^1_{st}\|_{\mathcal{L}(E^{n},E^{n - 1})} \lesssim |t-s|^{\alpha}  \ \  \text{for}\ \ -2 \leq n \leq 3 , \quad
\|B^2_{st}\|_{\mathcal{L}(E^{n},E^{n-2})} \lesssim |t-s|^{2 \alpha}  \ \ \text{for}\ \ 0\leq n\leq 2,
\end{equation}
and Chen's relation is satisfied,
\begin{equation}\label{eq:chen-relation}
\delta B^1_{s r t}=0,\qquad\delta B^2_{s r t}= B^1_{r t}B^1_{s r},
\qquad
\forall s < r < t.
\end{equation}
We shall write $\|\bB\|_{\alpha}$ for the smallest constant dominating the bounds in \eqref{ineq:UBRcontrolestimates}.
\end{definition}

We show how to construct an unbounded rough driver given a rough path.

\begin{proposition} \label{prop:RPtoURD}
Let $N\in \N$ and $\bX$ satisfy Assumption \ref{asm: linear pde: X}. Define for $\phi \in C^{\infty}_c(\R^d;\R^N)$
$$
B_{st}^1 \phi(x) = X_{st}^j (x)\partial_j \phi (x) , \qquad B_{st}^2 \phi(x) = (\nabla_1^{\otimes} \XX_{st})^j(x,x) \partial_j \phi(x) + \XX_{st}^{i,j}(x,x) \partial_i \partial_j \phi(x)
$$
where $\nabla^{\otimes}_1 : C_b^3(\R^d \times \R^d;\R^{d \times d}) \rightarrow C_b^{2}(\R^d \times \R^d;\R^d)$ is the linear extension of the map defined on the algebraic tensor as 
\begin{equation}
\label{eq: def nabla tensor}
\nabla^{\otimes}_1 ( f \otimes g)^j(x,y) = g^i(y) \partial_i f^j(x) .
\end{equation}
Then $\bB:= (B^1,B^2)$ is an unbounded rough driver on both scales $E_n :=  W^{n,\rho}(\R^d;\R^N)$, $\rho \geq 1$, and $E_n := C_b^n(\R^d;\R^N)$. Moreover, the mapping $\bX \mapsto \bB$ is continuous in the operator norm. 
\end{proposition}

\begin{proof}
Let $0\leq s\leq \theta \leq t$. By Chen's relation for rough paths \eqref{eq: chen rough paths}, and \eqref{eq: def nabla tensor}
\begin{align*}
\delta \big[ (\nabla_1^{\otimes} \XX)^j(x,x) \partial_j \phi(x) \big]_{s \theta t} &  =   \nabla_1^{\otimes} (X_{s \theta} \otimes X_{\theta t})^j(x,x) \partial_j \phi(x)   = X_{\theta t}^i(x) \partial_i X_{s \theta}^j(x) \partial_j \phi(x)
\end{align*}
which gives
\begin{align*}
\delta B^2_{s \theta t} \phi(x) & =  X_{\theta t}^i(x) \partial_i X_{s \theta}^j(x) \partial_j \phi(x) + X_{s \theta}^i(x) X_{\theta t}^j(x) \partial_i \partial_j \phi(x) =  X_{\theta t}^i (x) \partial_i [ X_{s \theta}^j(x) \partial_j \phi (x)].
\end{align*}
Continuity of the mapping follows immediately from the continuity of $\nabla_x^{\otimes}$.
\end{proof}
We notice that there is no zero order term in the above unbounded rough driver. We include such a term by considering a rough path $\bX \in \mathscr{C}^{\alpha}([0,T];C_b^3(\R^{1+d};\R^{1+d}))$, i.e. with an additional spatial variable. Then, for $\phi \in C_c^{\infty}(\R^d;\R^N)$ let 
\begin{align*}
B_{st}^1 \phi(x) & = X_{st}^j (x)\partial_j \phi (x) + X_{st}^0 \phi(x) \\
B_{st}^2 \phi(x) & = (\nabla_1^{\otimes} \XX_{st})^j(x,x) \partial_j \phi(x) + \XX_{st}^{i,j}(x,x) \partial_i \partial_j \phi(x)  \\
 & \quad + \XX_{st}^{0,0}(x,x) \phi(x) + \XX_{st}^{0,j}(x,x) \partial_j \phi(x) + (\nabla_1^{\otimes} \XX_{st})^0(x,x)  \phi(x),
\end{align*}
where we make the convention that summation over repeated indexes are over $1 \leq j \leq d$, i.e. excluding $0$.

With this in hand we can define the notion of a solution of \eqref{LinearPDEGeneral}.

\begin{definition} \label{def:measure URD definition}
A path $\nu : [0,T] \rightarrow \mathcal{M}(\R^d) \subset (C_b(\R^d))^*$ is a solution to \eqref{LinearPDEGeneral} if for all $\phi \in C_b^3(\R^d)$ the mapping defined by
\begin{equation} \label{eq:Natural}
\nu_{st}^{\natural}(\phi) := \delta \nu_{st}(\phi) - \int_s^t \nu_r( \Tr(a_r \nabla^2 \phi)) dr - \nu_s( B_{st}^1 \phi) - \nu_s(B_{st}^2 \phi)
\end{equation}
satisfies $|\nu_{st}^{\natural}(\phi)| \lesssim |t-s|^{3 \alpha} \|\phi\|_{C_b^3}$. Above $\bB = (B^1,B^2)$ is the unbounded rough driver constructed from $\bX$ as in Proposition \ref{prop:RPtoURD}.
\end{definition}

We see now that, in the special case when $a= \frac12 \sigma\sigma^T$ and $\bX = \int \beta_rd\bZ_r$, existence of solutions follows from the results of Sections \ref{sec:NonLinear} and \ref{sec: rough non-linearities}.

\begin{proposition}
	\label{pro: measure existence linear pde}
	Let $\rho \geq 2$ and let $(\Omega, \mathcal{F},(\mathcal{F}_t)_{t \in [0,T]}, P)$ be a probability space that supports a $d$-dimensional Brownian motion $W$ and an $\mathcal{F}_0$-measurable random variable, $\Xi \in L^{\rho}(\Omega;\R^d)$ such that the push-forward measure $P_*(\Xi) = \nu_0$.
	Let $\bZ \in \mathscr{C}_{wg}^{\bar{\alpha}}([0,T];\R^m)$ be a a weakly geometric rough path. Under Assumption \ref{asm: building rough drivers}, we have
	\begin{enumerate}
		\item 
		$\bB$, generated by the rough path $\int \beta_r d\bZ_r$ as in Proposition \ref{prop:RPtoURD}, is an unbounded rough driver as in Definition \ref{def:urd}.
		\item 
		There exists a solution $\nu$ of \eqref{LinearPDEGeneral} driven by $\bB$, in the sense of Definition \ref{def:measure URD definition}. This solution is given by $\nu_t = \mathcal{L}(x_t)$, where, for $P$-a.e. $\omega \in \Omega$, $x(\omega)$ is the unique solution to equation \eqref{eq: non linear equation} with initial condition $\Xi(\omega)$, driven by the random rough driver $\bF$ constructed in Lemma \ref{lem: rough drivers mixed}.
	\end{enumerate}
\end{proposition}

\begin{proof}
	From Sobolev embedding theorem \cite[Corollary 9.13]{brezis2011} , we have $\beta \in \mathscr{D}_{Z}^{2\alpha}([0,T]; C_b^3(\R^d;\R^d))$. Thus, using the construction \eqref{eq: rough path from rough integration}, we have that $\int \beta_r d\bZ_r$ is a rough path over $C_b^3(\R^d;\R^d)$. The first claim follows now by Proposition \ref{prop:RPtoURD}.
	
	We prove now the second claim. It follows from Proposition \ref{prop: adapted solutions} that the stochastic process $(x_t)_{t\in [0,T]}$ is adapted. We can thus define $\nu := \mathcal{L}(x)$ and denote by $\nu_t$ the induced time-marginals. From It\^o's formula, Proposition \ref{prop:ItoFormula}, we get 
$$
\nu_t(\phi) = \nu_0(\phi)  + \int_0^t \frac12 \nu_r( \Tr( \nabla^2 \phi \sigma_r \sigma_r^T)) dr + \int_0^t \nu_r( \nabla \phi \beta_r) d\bZ_r.
$$
The proof is complete once we show that $\int_0^t \nu_r( \nabla \phi \beta_r) d\bZ_r$ has an expansion in terms of the unbounded rough driver. 
Recall that we get from Lemma \ref{lemma:AverageSharp}, we have
$$
(\nu( \nabla\phi \beta), \nu( \nabla^2 \phi (\beta \otimes \beta)+ \nabla \phi( \nabla \beta \beta + \beta^{\prime})) \in \mathscr{D}_Z^{2 \alpha}([0,T];H^k)
$$
and this gives, using the sewing lemma \ref{SewingLemma},
$$
\left|\int_s^t \nu_r( \phi \beta_r) d\bZ_r - \nu_s( \nabla\phi \beta_s^j) Z^j_{st} - \nu_s( \nabla^2 \phi \beta_s^j \otimes \beta_s^i+ \nabla \phi( \nabla \beta_s^j \beta_s^i + \beta_s^{j,i}) \ZZ_{st}^{i,j} \right| \lesssim \|\phi\|_{C_b^3} |t-s|^{3 \alpha} .
$$
Regrouping the terms we can write
$$
\left|\int_s^t \nu_r( \phi \beta_r) d\bZ_r - \nu_s\big( \beta_s^j  \nabla\phi Z_{st}^j + \beta_s^{j,i}  \nabla \phi \ZZ^{i,j}_{st})  - \nu_s( \beta_s^j \nabla   (\beta_s^i \nabla \phi)  \ZZ^{i,j}_{st} ) \right| \lesssim \|\phi\|_{C_b^3} |t-s|^{3 \alpha}.
$$
By definition of $B^1$ we get
$$
\|B_{st}^1 \phi - \beta_s^j \nabla \phi Z^j_{st} - \beta_s^{j,i} \nabla \phi \ZZ^{i,j}_{st} \|_{C_b} \lesssim \|\phi\|_{C^{1}_b} |t-s|^{3 \alpha}  , 
$$
which gives
$$
\left|\nu_s \big( \beta_s^j  \nabla\phi Z^j_{st} + \beta_s^{j,i}  \nabla \phi \ZZ^{i,j}_{st})  - \nu_s(B_{st}^1 \phi)\right| \lesssim \|\phi\|_{C^{1}_b} |t-s|^{3 \alpha} .
$$
Moreover
\begin{align*}
\| B_{st}^2 \phi - & \beta_s^j \nabla   (\beta_s^i \nabla \phi)  \ZZ^{i,j}_{st} \|_{C_b}    \lesssim \big\|\beta_s^j \nabla  B_s^1 \phi Z^j_{st} + \beta_s^{j,i} \nabla  B_s^1 \phi \ZZ^{i,j}_{st}   + \beta_s^j \nabla ( \beta_s^i \nabla ) \phi \ZZ^{i,j}_{st} \\
& - (\beta_s^j \nabla   Z^j_{st} + \beta_s^{j,i} \nabla  \ZZ_{st}^{i,j}) B_s^1 \phi - \beta_s^j \nabla (\beta_s^i \nabla  \phi)\ZZ_{st}^{i,j} \big\|_{C_b}  + \|\phi\|_{C_b^{2}} |t-s|^{3 \alpha}   \\
 & =  \|\phi\|_{C_b^{2}} |t-s|^{3 \alpha} .
\end{align*}
This shows that we may rewrite the equation for $\nu$ as
$$
\delta \nu_{st}(\phi) = \int_s^t \frac12 \nu_r( \Tr( \nabla^2 \phi \sigma_r \sigma_r^T) )dr + \nu_s( B_{st}^1 \phi+ B_{st}^2 \phi) + \nu_{st}^{\natural}(\phi)
$$
where $\nu^{\natural} \in C_2^{3 \alpha}([0,T]; (C_b^3(\R^d))^*)$ is a remainder.
\end{proof}

\subsection{A priori estimates for smooth vector fields} \label{step1}

For this section we consider an approximation of equation \eqref{eq:Expansion}, driven by a smooth (in time) driver,
\begin{equation} \label{eq:BVEquation}
\partial_t u = \Div( a \nabla u)  + \dot{X}  \nabla u
\end{equation}
where $X$ is smooth. We will find bounds on $u$ in $H^n(\R^d)$ depending only on a canonical unbounded rough driver generated by $X$. The first step towards this goal is to write $u$ and all the derivatives as a vector in an $L^2$ space.

Let $u$ denote the (smooth) solution of \eqref{eq:BVEquation} and let $f = (u, \nabla u, \dots, \nabla^n u)$ denote the vector of gradients as taking values in the truncated tensor algebra $T^{(n)}(\R^d) = \bigoplus_{q=0}^n (\R^d)^{\otimes q}$. We will simply write $gv$ for the 1-contractive product 
$$
(\R^d)^{\otimes q} \times (\R^d)^{\otimes r} \rightarrow (\R^d)^{\otimes (q+r - 2)} ,
$$
e.g. for a $g \in (\R^d)^{\otimes 2}$ and $v \in \R^d$ the product $gv$ has component $i$ given by $g_{ij}v_j$.

Using Leibniz formula we have
$$
\nabla^q (\dot{X} \nabla u ) = \sum_{j=0}^q { q \choose j }\,  \nabla^{q-j} \dot{X} \nabla^{j+1} u = \sum_{j=0}^{q-1} {q \choose j} \,  \nabla^{q-j} \dot{X} f^{(j+1)} + \dot{X} \nabla f^{(q)}  =: \dot{X} \nabla f^{(q)} + M_{\dot{X}}^{(q)} f
$$
where $M_{\dot{X}}^{(q)} : T^{(n)}(\R^d) \rightarrow (\R^d)^{\otimes q}$ is given by
$$
M_{\dot{X}}^{(q)} \big( \bigoplus_{j=0}^n y^{(j)} \big) = \sum_{j=0}^{q-1} {q \choose j} \, \nabla^{q-j}  \dot{X}  \, y^{(j+1)} .
$$
We notice that the above sum is in $(\R^d)^{\otimes q}$ since we are doing a contractive product of $(\R^d)^{\otimes (q-j + 1)}$ and $(\R^d)^{\otimes (j+1)}$.

For each $q$ we have
\begin{align*}
\partial_t f^{(q)} & = \nabla^q \Div (a  \nabla u) + \nabla^q( \dot{X}  \nabla u ) = \Div ( a \nabla f^{(q)}) + \nabla (M_a^{(q)} f) + (\dot{X}  \nabla f^{(q)} + M_{\dot{X}}^{(q)} f)  \\
& = \Div ( a  \nabla f^{(q)}) +  M_{\nabla a}^{(q)} f + M_a^{(q)}  \nabla f + (\dot{X}  \nabla f^{(q)} + M_{\dot{X}}^{(q)} f)  .
\end{align*}
This gives that $f$ satisfies the $T^{(n)}(\R^d)$-valued equation
\begin{equation} \label{TruncatedTensorEquation}
\partial_t f = \Div( a  \nabla f)  + \dot{V}  \nabla f + \dot{Y} f 
\end{equation}
where we have set 
\begin{equation} \label{TensorCoefficients}
\dot{V} = \bigoplus_{q=0}^n ( M^{(q)}_a + \dot{X}), \qquad \dot{Y} = \bigoplus_{q=0}^n ( M^{(q)}_{\nabla a} +  M^{(q)}_{\dot{X}}) .
\end{equation}

\begin{remark}
We notice that if we replace $X$ above by $X^{\epsilon}$ where $\bX^{\epsilon}$ converges to a rough path $\bX$, then the corresponding coefficients $V^{\epsilon}$, $Y^{\epsilon}$ have canonical rough path lifts, $\bV^{\epsilon}$ and $\bY^{\epsilon}$, with values in $C_b^3$ which remain bounded uniformly in $\epsilon$. This comes from the fact that there are canonical iterated integrals between the $C_b^3$-valued paths $t \mapsto \int_0^{\cdot} a_r(x) dr$ and $t \mapsto X_t(x)$,
$$
\int_s^t X_{sr}(x) a_r(y) dr , \quad \int_s^t a_{sr}(x) dX_r(y)
$$
where the first term is simply the Riemann-integral and the second term is defined using integration by parts as before. 
\end{remark}

Given the previous construction, we consider now a system of equations. We remark that this is not just a vector valued version of the results found in \cite{HH}, since we are not interested in energy estimates. Indeed, the matrix $a$ is allowed to be degenerate but we require spatial smoothness.
We consider the equation
\begin{equation} \label{eq:SystemNonIndex}
\partial_t f =   \Div(a  \nabla f)  + \dot{V}  \nabla f + \dot{Y} f  , 
\end{equation}
for given functions $a$ and $\dot{V},\dot{Y}$ smooth in time, and a given initial condition $f_0$. The solution is a vector valued function $f :[0,T] \times \R^d \rightarrow \R^N$, and the coefficients are on the form 
\begin{align*}
&\dot{Y}: [0,T] \times \R^d \rightarrow \R^{N} \otimes \R^N , \quad \dot{V} : [0,T] \times \R^d \rightarrow \mathcal{L}(\R^d \otimes \R^N ; \R^N) ,\\
&a: [0,T] \times \R^d \rightarrow \R^d \otimes \R^d. 
\end{align*}
We will assume that $a$ is diagonal in \eqref{eq:SystemNonIndex}, so component $l$ reads
\begin{equation} \label{eq:SystemComponents}
\partial_t f^l =   \partial_i (a^{i,j}\partial_j f^l) +  \dot{V}_{i}^{l,m} \partial_i  f^m + \dot{Y}^{l,m} f^m , \quad 1 \leq l \leq N .
\end{equation}

We begin with our main a priori estimate.
\begin{proposition}
Assume $f$ is a solution of \eqref{eq:SystemNonIndex}. Then there exists a constant $C =C(a,B^1,B^2)$ such that 
\begin{equation} \label{eq:APrioriInequlity}
\sup_{ t \in [0,T]} \|f_t\|_{L^2(\R^d;\R^N)} \leq C \|f_0\|_{L^2(\R^d)} , \quad \|\delta f_{st}\|_{H^{-1}(\R^d;\R^N)} \leq C|t-s|^{\alpha}
\end{equation}
where $(B^1,B^2)$ is an unbounded rough driver depending only on the rough path lift of the path $(V,Y)$.
\end{proposition}
\begin{proof}
The finite-dimensional tensor $(f^{\otimes 2})^{n,l} := f^{n}f^{l}$ then satisfies
\begin{align*}
\partial_t (f^{\otimes 2}) & =  2 f \st \partial_t f  = 2 f  \st \Div(a \nabla f) +  2 f \st \dot{V}  \nabla f + 2 f \st \dot{Y} f  \\
 & = 2 f  \st \Div(a  \nabla f)+ \dot{\bar{V}}  \nabla f^{\otimes 2}  + \dot{\bar{Y}} f^{\otimes 2} 
\end{align*}
where
$$
\bar{V} := 2 \textrm{Id} \st V := \textrm{Id} \otimes V + V \otimes \textrm{Id} , \qquad \bar{Y} := 2 \textrm{Id} \st Y := \textrm{Id} \otimes Y + Y \otimes \textrm{Id} ,
$$
both belongs to the space $\mathcal{L}(\R^N \otimes \R^N;\R^N \otimes \R^N)$. Define now the unbounded rough driver
\begin{equation} \label{TensorURD}
B_{st}^{\otimes 2,1} \phi = \int_s^t  \dot{\bar{V}}_{r}  \nabla \phi  +  \dot{\bar{Y}}_{r} \phi dr , \qquad  B_{st}^{\otimes 2,2} \phi = \int_s^t \left[  \dot{\bar{V}}_{r}   \nabla \cdot  +  \dot{\bar{Y}}_{r} \right]  \int_s^r  \dot{\bar{V}}_{\theta }  \nabla \phi  +  \dot{\bar{Y}}_{\theta } \phi d\theta  dr
\end{equation}
and the drift
$$
m_t^{\otimes 2}(\phi) = -  \int_s^t 2 ( a \nabla f^n, \nabla f^l \phi^{l,n}  )   -   (  f^n f^l , \nabla (a  \nabla \phi^{l,n})) dr
$$
for functions $\phi : \R^d \rightarrow \R^N \otimes \R^N$. 
This gives the dynamics
$$
\delta f_{st}^{\otimes 2} = \delta m^{\otimes 2}_{st}  + B_{st}^{\otimes 2,1} f_s^{\otimes 2} + B_{st}^{\otimes 2,2} f_s^{\otimes 2} +f_{st}^{\otimes 2, \natural}
$$
on the scale $(W^{n, \infty}(\R^d; \R^N \otimes \R^N))_n$. Let $\phi \in W^{2, \infty}(\R^d; \R^N \otimes \R^N)$ and write
\begin{align*}
|\delta m_{st}^{\otimes 2}(\phi)| 
  & \leq \| \phi \|_{W^{1, \infty}}  \int_s^t 2 \| (\nabla f^l)^T a \nabla f^n \|_{L^1}  + \|a\|_{W^{1,\infty}} \|  f \|_{L^2}^2 dr  ,
\end{align*}
which shows that $m^{\otimes 2}$ has bounded variation in $(W^{2, \infty}(\R^d; \R^N \otimes \R^N)^*)$. 

Now, by the a priori bounds, \cite[Theorem 2.9]{DeGuHoTi16}, we get
$$
\|f^{\otimes 2, \natural}_{st}\|_{(W^{3, \infty})^*} \leq C \left( \|f\|_{L^{\infty}(s,t;L^2)}^2  |t-s|^{3 \alpha} +  |t-s|^{\alpha}  \int_s^t 2 \| (\nabla f^l)^T a \nabla f^n \|_{L^1}  + \|a\|_{W^{1,\infty}} \|  f \|_{L^2}^2 dr \right) .
$$
where $C$ depends on $\|\bB^{\otimes 2}\|_{\alpha}$.
Testing $f^{\otimes 2}$ against the $N \times N$ identity matrix $I_{N \times N}$ and using that $a$ is positive semi-definite we get
\begin{align*}
\delta (\|f\|_{L^2}^2)_{st} & = \delta m^{\otimes 2}_{st}(I_{N \times N}) + f_s^{\otimes 2}( B_{st}^{\otimes 2, 1, *}I_{N \times N} + B_{st}^{\otimes 2, 2, *}I_{N \times N}) + f_{st}^{\otimes 2, \natural}(I_{N \times N}) \\
 & \leq - 2 \int_s^t \| (\nabla f^n)^T a \nabla f^n  \|_{L^1}  dr +  \|f_s\|_{L^2}^2 \|\bB^{\otimes 2}\|_{\alpha} |t-s|^{ \alpha} +  \|f_s\|_{L^2}^2 \|\bB^{\otimes 2}\|_{\alpha} |t-s|^{2 \alpha} + \|f^{\otimes 2, \natural}_{st}\|_{(W^{3, \infty})^*} \\
 & \leq - 2 \int_s^t \| (\nabla f^n)^T a \nabla f^n  \|_{L^1}   dr + C  \|f\|_{L^{\infty}(s,t;L^2)}^2  |t-s|^{\alpha} \\
  &  + C |t-s|^{\alpha} \int_s^t 2 \| (\nabla f^l)^T a \nabla f^n \|_{L^1}  + \|a\|_{W^{1,\infty}} \|  f \|_{L^2}^2 dr
\end{align*}
Note that $\|  (\nabla f^l)^T a \nabla f^n \|_{L^1}   \leq N \| (\nabla f^n)^T a \nabla f^n \|_{L^1}$. Indeed, write $a = \frac12 \sigma \sigma^T$ and use the Cauchy-Schwarz inequality 
$$
(\nabla f^l)^T a \nabla f^n = \frac12 (\sigma^T \nabla f^l)^T (\sigma^T \nabla f^n) \leq \frac12 |\sigma^T \nabla f^l| |\sigma^T \nabla f^n| . 
$$
Summing over $l$ and $n$ gives that the above is bounded by $ \frac{N}{2} \sum_{n=1}^N |\sigma^T \nabla f^n|^2$. Integrating w.r.t. $x$ we get the claim.

If we choose $s,t$ such that $C N |t-s|^{\alpha}  \leq \frac12$ we get
\begin{align*}
\delta (\|f\|_{L^2}^2)_{st} & \leq  C  \|f\|_{L^{\infty}(s,t;L^2)}^2 \big( |t-s|^{\alpha} + \|a\|_{W^{1, \infty}} (t-s) \big)  .
\end{align*}
From the rough Gronwall lemma, \cite[Lemma 2.11]{DeGuHoTi16}, the first bound of \eqref{eq:APrioriInequlity} holds.

For the second inequality we notice that the evolution of $f$ on $W^{n,2}(\R^d;\R^N)$ reads
\begin{equation} \label{eq:fExpansion}
\delta f_{st} = \delta m_{st} + B_{st}^1 f_s + B_{st}^2 f_s + f_{st}^{\natural} 
\end{equation}
where $m_t = \int_0^t \Div(a_r \nabla f_r) dr$ and we have defined the unbounded rough driver 
\begin{equation} \label{URD}
B_{st}^{1} \phi = \int_s^t  \dot{V}_{r}  \nabla \phi  +  \dot{Y}_{r} \phi dr , \qquad  B_{st}^{2} \phi = \int_s^t \left[  \dot{V}_{r}  \nabla \cdot  +  \dot{Y}_{r} \right]  \int_s^r  \dot{V}_{\theta }  \nabla \phi  +  \dot{Y}_{\theta } \phi d\theta  dr
\end{equation}

Since the operator is self-adjoint it is easy to bound the variation of $m$ in $H^{-2}$;
$$
|\delta m_{st}(\phi)| \leq (t-s)\|f\|_{L^{\infty}([s,t]; L^2)} \|a\|_{W^{1,\infty}} \|\phi\|_{H^2} \leq (t-s) C \|f_0\|_{ L^2} \|a\|_{W^{1,\infty}} \|\phi\|_{H^2} .
$$
This gives, using \cite[Theorem 2.9]{DeGuHoTi16},
\begin{equation} \label{SystemRemainder}
\|f_{st}^{\natural}\|_{H^{-3}} \lesssim  C |t-s|^{3 \alpha}  \|f_0\|_{L^2} .
\end{equation}
where $C$ depends on $\|\bB\|_{\alpha}$ and $\|a\|_{W^{1,\infty}}$. 
Take now a mollifier $\psi_{\eta}$ and decompose $\phi = \psi_{\eta} * \phi + (I-\psi_{\eta}) * \phi$ for any $\eta > 0$ and any test function $\phi \in H^1(\R^d;\R^N)$. 
This gives
$$
|(\delta f_{st}, (I-\psi_{\eta})* \phi) |\lesssim \|f\|_{L^{\infty}([s,t]; L^2)} \|(I-\psi_{\eta})* \phi\|_{L^2} \lesssim \|f_0\|_{ L^2}  \|\phi\|_{H^1} \eta ,
$$
and for the smooth part $\psi_{\eta}*\phi$ we use the equation \eqref{eq:fExpansion} to get
\begin{align*}
|(\delta f_{st}, \psi_{\eta} * \phi) | & \lesssim (t-s) \|f_0\|_{ L^2} \|a\|_{W^{1,\infty}} \|\psi_{\eta} * \phi\|_{H^2} + \|\bB\|_{\alpha}  \|f_0\|_{ L^2}  \|\psi_{\eta} *\phi\|_{H^1} + \|\bB\|_{\alpha}  \|f_0\|_{ L^2}  \|\psi_{\eta} * \phi\|_{H^2} \\
& +   |t-s|^{3 \alpha} C \|f_0\|_{L^2} \|\psi_{\eta} * \phi\|_{H^3} \\ 
& \leq \|f_0\|_{L^2} C \big[ (t-s) \eta^{-1}  + |t-s|^{\alpha}  + |t-s|^{2 \alpha}  \eta^{-1} +  |t-s|^{3 \alpha}  \eta^{-2} \big]  \|\phi\|_{H^1} .
\end{align*}
Choosing $\eta = |t-s|^{\alpha}$ we get
the second inequality in \eqref{eq:APrioriInequlity}.
\end{proof}

\subsection{Existence of a smooth solution}
With the previous a priori estimates at hand, we are ready to prove existence of a solution.

\begin{theorem}
Let Assumption \ref{asm: linear pde} hold for $n > 6 + \frac{d}{2}$ and let $u_0 \in C_c^{\infty}(\R^d)$ be given. Then there exists a solution to \eqref{MainAppendixEquation} which belongs to $C_b^6$ and 
\begin{equation} \label{Smoothsolution}
\delta u_{st} = \int_s^t \Div( a_r \nabla u_{r}) dr +B_{st}^1 u_s + B_{st}^2 u_s + u_{st}^{\natural}
\end{equation}
holds in $C^3_b$ in the sense that $u^{\natural} \in C_2^{3 \alpha}([0,T]; C_b^3(\R^d))$, where $\bB = (B^1,B^2)$ is the unbounded rough driver constructed from $\bX$ as in Proposition \ref{prop:RPtoURD}.
\end{theorem}

\begin{proof}

Denote by $u^{\epsilon}$ the solution of \eqref{MainAppendixEquation} when $X$ is replaced by  $X^{\epsilon}$, which we write
\begin{equation} \label{eq:ApproximateEquation}
\delta u_{st}^{\epsilon} = \int_s^t \Div( a \nabla u_{r}^{\epsilon}) dr + X_{st}^{\epsilon}  \nabla  u_{s}^{\epsilon} +\int_s^t  \dot{X}_{r}^{\epsilon}  \nabla (X_{sr}^{\epsilon}  \nabla u_{s}^{\epsilon} )dr +  u_{st}^{\epsilon, \natural} .
\end{equation}
Setting $f^{\epsilon} =  (u^{\epsilon}, \dots, \nabla^n u^{\epsilon})$ and choosing $N$ large (in fact $N = 1 + d + \dots + d^n$) we see that \eqref{TruncatedTensorEquation} is on the form \eqref{eq:SystemNonIndex} where $V^{\epsilon}$ and $Y^{\epsilon}$ are defined from $X^{\epsilon}$ using \eqref{TensorCoefficients}. We then build the unbounded rough driver $\bB^{\epsilon, \otimes 2}$ and $\bB^{\epsilon}$ from $V^{\epsilon}$ and $Y^{\epsilon}$ according to \eqref{TensorURD} and \eqref{URD} respectively. 

By the assumptions on $a$, $\bX$ and $u_0$ we get
$$
\sup_{t \in  [0,T]} \| u^{\epsilon}_t \|_{H^n}^2  =  \sup_{t \in  [0,T]} \sum_{k = 0}^n \|\nabla^k  u^{\epsilon}_t \|_{L^2(\R^d)}^2 = \sup_{t \in [0,T]} \|f_t^{\epsilon}\|^2_{L^2(\R^d; T^{(n)}(\R^d))} \leq C \sup_{t \in [0,T]} \|f_0\|^2_{L^2(\R^d; T^{(n)}(\R^d))} = C \| u_0 \|_{H^n}^2
$$
for some constant $C$. 
For $\phi \in H^{n+1}$, define $\Phi \in L^2(\R^d;T^{(n)}(\R^d))$ by $\Phi = (\phi, \nabla \phi, \dots, \nabla^n \phi)$ and notice
\begin{align*}
(\delta u_{st}^{\epsilon}, \phi)_{H^n} & = \sum_{k=0}^n ( \delta \nabla^k u_{st}^{\epsilon}, \nabla^k \phi)_{L^2(\R^d)} = (\delta f_{st}^{\epsilon}, \Phi)_{L^2(\R^d;T^{(n)}(\R^d))} \\
 & \leq C |t-s|^{\alpha} \|\Phi\|_{H^1(\R^d;T^{(n)}(\R^d))} \leq C |t-s|^{\alpha} \|\phi\|_{H^{n+1}(\R^d)}
\end{align*}
Since $H^{n+1}$ and $H^{n-1}$ are dual w.r.t. to the inner product on $H^n$, we get $\|\delta u^{\epsilon}\|_{H^{n-1}(\R^d)} \leq  C  |t-s|^{\alpha}$.
By similar reasoning we get $\|u^{\natural}_{st}\|_{H^{n-3}(\R^d)} \leq  C |t-s|^{3 \alpha}$ using \eqref{SystemRemainder}.

Since $u^{\epsilon}$ lies in a bounded set of $C^{\alpha}([0,T];H^{n-1}(\R^d)) \cap C([0,T];H^n(\R^d))$, by Arzel\`{a}-Ascoli there exists a subsequence $u^k := u^{\epsilon_k}$ converging in $C([0,T];H^n_w(\R^d))$ some element $u$. Here $H_w^n(\R^d)$ denotes $H^n(\R^d)$ equipped with the weak topology. 
Choosing now $n > 6 + \frac{d}{2}$ and using Sobolev embedding \cite[Corollary 9.13]{brezis2011}  we get that $u^{\epsilon, \natural}$ is bounded in $C^{3 \alpha}_2([0,T];C^3_b(\R^d))$ and $u \in C([0,T];C_b^6(\R^d))$.

It is straightforward to take the limit in \eqref{eq:ApproximateEquation} and use the uniform bounds on $u^{\epsilon, \natural}$ to obtain \eqref{Smoothsolution}. 
%
\end{proof}

\subsection{Uniqueness}

\begin{theorem} \label{thm:LinearUniqueness}
Let Assumption \ref{asm: linear pde} hold for $n > 6 + \frac{d}{2}$. Then solutions of \eqref{LinearPDEGeneral} are unique. 
\end{theorem}

\begin{proof}
Let $\nu$ be a solution to \eqref{LinearPDEGeneral}, i.e. for all $\phi \in C_b^3$ we have
$$
\delta \nu_{st}(\phi) = \int_s^t \nu_r( \Tr(a_r \nabla^2 \phi)) dr - \nu_s(  B_{st}^1  \phi) +\nu_s( B_{st}^2 \phi) + \nu_{st}^{\natural}(\phi)
$$
where $\nu^{\natural} \in C_2^{3 \alpha}([0,T];(C_b^3(\R^d))^*)$ and $\bB = (B^1,B^2)$ is the unbounded rough driver constructed from $\bX$. Let $u$ be the solution of the backward equation \eqref{eq:BackwardEquation} with final condition $\psi \in C^{\infty}_c(\R^d)$ so that
$$
\delta u_{st} = -  \int_s^t \textrm{Tr}( a_r  \nabla^2 u_r) dr +  B_{st}^1 u_s +   B_{st}^2 u_s + u_{st}^{\natural}
$$
holds in $C^3_b$. We then have
\begin{align}
\delta \nu(u)_{st}  \notag &  = \delta \nu_{st}(u_s) + \nu_s( \delta u_{st}) + \delta \nu_{st}( \delta u_{st}) =  \int_s^t \nu_r( \textrm{Tr} [a_r \nabla^2 u_s]) dr + \nu_s( - B_{st}^1u_s  + B_{st}^2 u_s]) + \nu_{st}^{\natural}(u_s) \\
\notag  &  - \int_s^t \nu_s\big( \textrm{Tr}[ a_r  \nabla^2 u_r] \big) dr +  \nu_s(B_{st}^1 u_s) +   \nu_s(B_{st}^2 u_s) + \nu_s(u_{st}^{\natural}) \\
\notag  & + \nu_{st}^{\sharp}( u_{st}^{\sharp}) + \nu^{\sharp}_{st}(B^{1}_{st} u_s) - \nu_{s}(B^{1}_{st} u_{st}^{\sharp}) - \nu_{s}( B^1_{st} B^{1}_{st} u_{s}) \\
\notag   & =   \int_s^t \delta \nu_{sr}(\textrm{Tr}[ a_r \nabla^2 u_r]) dr - \int_s^t \nu_r \big( \textrm{Tr}[ a_r \delta \nabla^2 u_{sr}] \big) dr \\ 
   & +   \nu_{st}^{\natural}(u_s)  +   \nu_s(u_{st}^{\natural})  + \nu_{st}^{\sharp}( u_{st}^{\sharp}) + \nu^{\sharp}_{st}(B^{1}_{st} u_s) - \nu_{s}(B^{1}_{st} u_{st}^{\sharp})  \label{ProductExpansion}
\end{align}
where we have defined 
$$
\nu_{st}^{\sharp} := \delta \nu_{st} + B_{st}^{1,*}\nu_s, \qquad u_{st}^{\sharp} := \delta u_{st} - B_{st}^{1}u_s
$$
and we have used that the path is geometric which gives $\nu_{s}( B^1_{st} B^{1}_{st} u_{s})  = \nu_{s}( B^2_{st}  u_{s})$. Using the equations for $u$ and $\nu$ we get $\nu^{\sharp} \in C_2^{2 \alpha}([0,T];(C_b^3(\R^d))^*)$ and $u^{\sharp} \in C_2^{2 \alpha}([0,T];C_b^3(\R^d))$. Using this and analyzing every term in \eqref{ProductExpansion} we see that 
$$
|\delta \nu(u)_{st}| \lesssim |t-s|^{ 3 \alpha} , \quad \Longrightarrow \quad \nu_t(u_t) = \textrm{const} .
$$
and in particular $\nu_T(\psi) = \nu_0(u_0)$. If $\bar{\nu}$ is any other solution with the same initial condition, the same analysis gives $\bar{\nu}_T(\psi) = \nu_0(u_0)$ which gives that $\nu_T(\psi) = \bar{\nu}_T(\psi)$. Since $\psi$ was arbitrary the result follows. 
\end{proof}

\section{The McKean-Vlasov equation} \label{sec:ControlledMeasures}
Let $d, m \in \N$ be fixed. Let $(\Omega, \mathcal{F}, (\mathcal{F}_t)_{t \in [0,T]}, P)$ be a complete filtered probability space and $W$ be a $d$-dimensional Wiener process on it. Let $\Xi :\Omega \to \R^d$ be an $\mathcal{F}_0$-measurable random variable. Let $\bZ \in \mathscr{C}_g^{\bar\alpha}([0,T], \R^m)$, for $\bar \alpha \in (\frac{1}{3}, \frac{1}{2})$. Moreover let $\alpha \in (\frac{1}{3},\bar\alpha)$ and $p=\frac{1}{\alpha}$.

In this section we prove well-posedness of the equation
\begin{equation}
\label{eq: mcvlasov srde}
dx_t = \sigma(\mathcal{L}(x_t),x_t) dW_t + \beta(\mathcal{L}(x_t),x_t) d\bZ_t, \quad x_0 = \Xi \in \R^d .
\end{equation}
We start by defining the notion of solution we shall use.
\begin{definition} \label{def: McKean-Vlasov}
	Let $\rho \geq 1$ and $\alpha \in (\frac{1}{3},\frac12]$. We say that an $(\mathcal{F}_t)_{t\geq 0}$-adapted stochastic process $x : \Omega \times [0,T] \to \R^d$ is a solution to equation \eqref{eq: mcvlasov srde} with initial condition $\Xi \in L^{\rho}(\Omega,\mathcal{F}_0;\R^d)$, if
	\begin{enumerate}[label=(\roman*), ref=\ref{def: McKean-Vlasov} (\roman*)]
		\item 
		\label{def: McKean-Vlasov: mu}
		$\mu_t := \mathcal{L}(x_t)$ is such that
		\begin{equation*}
		(\mu(\beta), \mu(\nabla \beta \mu(\beta))) \in \mathscr{D}_{Z}^{2\alpha}([0,T]; H^k).
		\end{equation*}
		and $\bF^{\mu}$ defined from $\sigma(\mu)$ and $\beta(\mu)$ as in Lemma \ref{lem: rough drivers mixed} is a rough driver in the sense of Definition \eqref{def:RoughDriver}.
		\item 
		\label{def: McKean-Vlasov: equation}
		$P$-almost surely, $x$ satisfies
		\begin{equation*}
		dx_t = \bF^{\mu}_{dt}(x_t),
		\qquad
		x_0 = \Xi,
		\end{equation*}
		in the sense of Definition \ref{def:RDSolution}.
	\end{enumerate}
\end{definition}

Before proceeding we state the assumptions that will be in force throughout the section.
\begin{assumption}
	\label{asm: mckean-vlasov coefficients}
	Let $k>\frac{d}{2}+3$ and $\rho \geq 1$,
	\begin{enumerate}[label=(\roman*), ref=\ref{asm: building rough drivers} (\roman*)]
		\item \label{asm: mckean-vlasov coefficients: beta} We assume $\beta \in \mathcal{L}(\R^m, C_b^3\otimes H^k)$.
		\item \label{asm: mckean-vlasov coefficients: sigma} Let $\sigma: \mathcal{P}_\rho(\R^d) \to \mathcal{L}(\R^d; H^k)$ be a measurable function, such that there exists a constant $C_{\sigma}>0$, with
		\begin{equation*}
		\| \sigma(\mu) - \sigma(\nu)\|_{\mathcal{L}(\R^d; H^k)} \leq C_{\sigma}W_{\rho}(\mu,\nu),
		\qquad
		\| \sigma(\mu)\|_{\mathcal{L}(\R^d; H^k)} \leq C_{\sigma},
		\qquad
		\forall \mu,\nu \in \mathcal{P}_{\rho}(\R^d).
		\end{equation*}
	\end{enumerate}
\end{assumption}

We now introduce a suitable space of measures in which will be useful for proving well-posedness of \eqref{eq: mcvlasov srde}. The set up is reminiscent of the controlled space as introduced in \cite{Gub04}, but tailored for measures on path spaces.

\begin{definition}
	\label{def: controlled measure}
	Let $\rho \geq 1$. We say that a pair $(\mu,\gamma) \in \mathcal{P}_{\rho}( C^{\alpha}_0([0,T];\R^d) ) \times C^{\alpha}([0,T];\mathcal{L}(\R^m;C_b^3(\R^d;\R^d))$ is controlled by $Z$ provided for every $\phi \in C_b^3 \otimes H^k$ we have that 
	$$
	( \mu( \phi), \mu(   \nabla_1 \phi \gamma )) \in \mathscr{D}_Z^{2 \alpha}([0,T]; H^k). 
	$$
	Here we used the notation
	$$
	\mu(\phi)_t = \int_{C^{\alpha}} \phi(\omega_t, \cdot) d \mu( \omega),  \qquad \mu( \nabla_1 \phi \gamma )_t^j = \int_{C^{\alpha}}  \nabla_1 \phi(\omega_t, \cdot)  \gamma_t^j(\omega_t) d \mu( \omega).
	$$
\end{definition}
For $\rho\geq 1$, we denote by $\mathcal{M}_Z^{2 \alpha, \rho}$ the set of all such controlled pairs equipped with the metric
$$
d\big( (\mu,\gamma), (\nu,\zeta) \big) = W_{\rho}(\mu,\nu) + [\gamma - \zeta]_{\alpha; C^3_b} + \sup_{\Vert \phi \Vert_{C^3_b \otimes H^k} \leq 1}
\Vert
(\mu(\phi) - \nu(\phi), \mu(  \nabla_1 \phi  \gamma) - \nu(\nabla_1 \phi \zeta ) ) 
\Vert_{Z, \alpha;H^k}.
$$

\begin{remark}
We note that in Definition \ref{def: McKean-Vlasov: mu} the law, $\mu_t = \mathcal{L}(x_t)$, of the solution is only defined for the time-marginals, and a priori it is not clear how to construct from this a measure on the path space $C_0^{\alpha}([0,T];\R^d)$. However, since $x$ satisfies the equation in Definition \ref{def: McKean-Vlasov: equation}, $x$ is a random variable in $C^{\alpha}([0,T];\R^d)$, and letting $h \rightarrow 0$ in \eqref{eq: local alpha norm x} and \eqref{eq: FF bound} we see that $x$ takes values in $C_0^{\alpha}([0,T];\R^d)$. Hence it induces the measure $\mathcal{L}(x)$ on $C_0^{\alpha}([0,T];\R^d)$ which clearly has time-marginals $\mu_t$.
\end{remark}

\begin{remark}
	Let $\beta$ and $\sigma$ satisfy Assumption \ref{asm: mckean-vlasov coefficients}, with $k > \frac{d}{2} + 3$, and let $(\mu,\gamma) \in \mathcal{M}_Z^{2 \alpha, \rho}$. Then, $\sigma(\mu)$ and $\mu (\beta), \mu(\nabla_1\beta \gamma)$ satisfy Assumption \ref{asm: building rough drivers}. Assumption \ref{asm: building rough drivers: beta} is verified by replacing $\varphi = \beta^i$, for $i=1,\dots,m$, in Definition \ref{def: controlled measure}. Assumption \ref{asm: building rough drivers: sigma} follows trivially by the boundedness in Assumption \ref{asm: mckean-vlasov coefficients: sigma}. We are only left with verifying \ref{asm: building rough drivers: sigma p var}. For all $s,t \in [0,T]$,
	\begin{equation}
	\label{eq: sigma is hoelder}
	\|\sigma(\mu_t) - \sigma(\mu_s) \|_{\mathcal{L}(\R^d; H^k)}
	\leq C_{\sigma} W_{\rho}(\mu_t, \mu_s)
	\leq C_{\sigma}\left(
	\int_{C^{\alpha}_0} |\omega_t - \omega_s|^{\rho} d\mu(\omega)
	\right)^{\frac{1}{\rho}} 
	\leq C_\sigma\left(
	\int_{C^{\alpha}_0} [\omega]_{\alpha}^{\rho} d\mu(\omega)
	\right)^{\frac{1}{\rho}} |t-s|^{\alpha}.
	\end{equation}
	This gives that $\sigma \in C^\alpha H^k \subset C^{p-var}H^k$, if $p=\frac{1}{\alpha}$.
\end{remark}

\begin{theorem} \label{thm:well posedness}
Suppose $\sigma$ and $\beta$ satisfies Assumption \ref{asm: mckean-vlasov coefficients} and $\rho \geq 2$. For any $\Xi_0 \in L^{\rho}(\Omega, \mathcal{F}_0; \R^d)$ there exists a unique solution $x$ of \eqref{eq: mcvlasov srde} in the sense of Definition \ref{def: McKean-Vlasov}.
\end{theorem}

\begin{proof}
We fix $\sigma,\beta$ satisfying Assumptions \ref{asm: mckean-vlasov coefficients} and construct the following mappings
\begin{equation} \label{Mappings}
\begin{array}{ccccccc}
\mathcal{M}_Z^{2 \alpha,\rho} & \rightarrow &  C^{\alpha} ([0,T];H^k) \times \mathscr{D}_Z^{2 \alpha}([0,T]; H^k) & \rightarrow  & \mathscr{C}^{\alpha} & \rightarrow &  \mathcal{M}_Z^{2 \alpha, \rho} \\
(\mu, \gamma) & \mapsto &  (\sigma(\mu), (\beta(\mu), \beta(\mu)') ) &  \mapsto  & \bF^{\mu}  & \mapsto & (\mathcal{L}(x), \beta(\mu)).
\end{array}
\end{equation}
and we shall use the notation $ \Gamma(\mu, \gamma) := (\mathcal{L}(x), \beta(\mu))$. 
By letting $h \rightarrow 0$ in \eqref{eq: local alpha norm x} and \eqref{eq: FF bound} we see that $\mathcal{L}(x)$ is supported on $C_0^{\alpha}([0,T];\R^d)$.
In Lemma \ref{lemma: invariance} and Lemma \ref{lemma: contraction} we show that $\Gamma$ is a contraction mapping on a subset of $\mathcal{M}_Z^{2 \alpha,\rho}$ for a small time parameter $T_0 \leq T$. Then, noting that $T_0 = T_0(\rho,\alpha, \sigma, \beta, \bZ)$ does not depend on the initial condition $\Xi_0$, the solution can be constructed iteratively on the full time interval $[0,T]$ by concatenation of the solutions defined on $[0,T_0]$, $[T_0, 2T_0]$ etc. 
\end{proof}


\begin{lemma} \label{lemma: invariance}
Define
\begin{equation} 
\label{def: bar L}
\bar L = \bar L(\sigma, \beta, \bZ) : =   \left(  1 + C_{\sigma} + \|\beta\|_{C_b^3 \otimes H^k} \right)(1 + [\bZ]_{\bar{\alpha}}\vee [\bZ]_{\bar{\alpha}}^{\frac{1}{2}}),
\end{equation}
and the closed subset of $\mathcal{M}_Z^{2 \alpha,\rho}$,
\begin{equation*}
\mathcal{B}_T := \left\{
(\mu,\gamma) \in \mathcal{M}_Z^{2 \alpha,\rho} 
\mid
d((\mu,\gamma), (\mu, \delta_0)) \leq 1,
\quad
W_{\rho}(\mu, \delta_0) \leq \frac13 \bar{L}^{-1}
\right\}.
\end{equation*}
	Assume Assumption \ref{asm: mckean-vlasov coefficients} with $\rho \geq 2$. There exists a small time $T = T(\rho,\alpha, \sigma, \beta, \bZ)$, such that $\Gamma$ leaves $\mathcal{B}_T$ invariant.
\end{lemma}
\begin{proof}
	We start by looking at the controlled function, 
	\begin{align*}
	\|\delta  \beta(\mu)_{st}\|_{H^k} & = \|\int_{C^{\alpha}_0} \delta \beta(\omega_{\cdot}, \cdot)_{st} d \mu(\omega) \|_{H^k} 
	\leq \| \beta \|_{C_b^3 \otimes H^k} \int_{C^{\alpha}_0} |\delta \omega_{st}|  d \mu(\omega) \\
	&\leq \| \beta \|_{C_b^3 \otimes H^k} \int_{C^{\alpha}_0} [\omega]_{\alpha}  d \mu(\omega) |t-s|^{ \alpha} \leq \| \beta \|_{C_b^3 \otimes H^k} W_{\rho}( \mu, \delta_0) |t-s|^{ \alpha} \leq \frac13 |t-s|^{\alpha}
	\end{align*}
	To show the bounds on the rough driver, start by noting that, by linearity,
	$$
	\| \beta(\mu), \beta(\mu)' \|_{Z, \alpha; H^k} \leq \|\beta\|_{C_b^3 \otimes H^k} \sup_{\|\phi\|_{C_b^3 \otimes H^k} \leq 1} \| \mu(\phi), \mu( \nabla_1 \phi \gamma) \|_{Z,\alpha; H^k} \leq \|\beta\|_{C_b^3 \otimes H^k}
	$$
	and thanks to \eqref{eq: sigma is hoelder}, 
$
	\|\sigma(\mu)\|_{L_t^\infty \mathcal{L}(\R^d;H^k)} \leq C_{\sigma}.
$
	This gives that for $(\mu,\gamma) \in \mathcal{B}_T$, we have $L(\sigma(\mu), \beta(\mu), \bZ) \leq \bar L(\sigma, \beta, \bZ)$, where $L$ is defined in \eqref{LConstant}. The previous observation and \eqref{eq: FF bound} imply
	$$
	[ \bF^{\mu}]_{\alpha}  = [ \bF^{\mu}]_{\alpha; C_b^3} 
	\leq \bar L T^{\frac{\bar{\alpha} - \alpha}{2}} K_{\rho} 
	$$
	for any $\alpha < \bar{\alpha}$ and for any $\rho \geq 1$ and for a random variable $K_{\rho} \in L^{\rho}(\Omega)$.
	From the a priori estimates \eqref{eq:a priori solution} we see that there exists a constant $C>0$, depending only on $\rho$ (which may change from an inequality to the next), such that
	\begin{equation*}
	\|[x]_{\alpha}\|_{L_\omega^\rho}
	\leq  C  (E[[\bF]_{\alpha}^{\rho} \vee [\bF]_{\alpha}^{\rho/\alpha}])^{1/\rho}
	\leq 
	C T^{\frac{\bar{\alpha} - \alpha}{2}}  \bar{L}^{1/\alpha}.
	\end{equation*}
	We may now choose $T \leq (3C\bar{L}^{1+1/\alpha})^{-\frac{2}{\bar\alpha - \alpha}} $ such that 
	\begin{equation}
	\label{eq: moment small}
	\bar{L} \|[x]_{\alpha}\|_{L_\omega^\rho} \leq \frac13.
	\end{equation}
	From Lemma \ref{lemma:ItoFormula} we get, 
	\begin{align*}
	\sup_{ \| \phi\|_{C_b^3 \otimes H^k} \leq 1}\|( \mathcal{L}(x)(\phi), \mathcal{L}(x)( \nabla_1 \phi \beta(\mu)) ) \|_{Z, \alpha;H^k} 
	\leq & T^{\frac{\bar\alpha - \alpha}{2}} 
	(1 + \| [x]_{\alpha} \|_{L_{\omega}^2}) \bar{L}
	\leq T^{\frac{\bar\alpha - \alpha}{2}} 
	(1 +C \bar{L}^{1/\alpha}) \bar{L}.
	\end{align*}
	and we choose $T  \leq (3(1+C)(1+\bar{L}^{1+1/\alpha}))^{-\frac{2}{\bar\alpha - \alpha}} $ such that the above is bounded by $\frac13$. This shows that
	\begin{equation*}
	d(\Gamma(\mu, \gamma), (\delta_0,0))
	 = \|[x]_{\alpha}\|_{L_\omega^\rho} 
	+ [\beta(\mu)]_{\alpha;C^3}
	+ \sup_{ \| \phi\|_{C_b^3 \otimes H^k} \leq 1}\|( \mathcal{L}(x)(\phi), \mathcal{L}(x)( \nabla_1 \phi \beta(\mu)) ) \|_{Z, \alpha;H^k} \leq \frac13.
	\end{equation*}
	This, together with \eqref{eq: moment small} implies $\Gamma(\mathcal{B}_T)\subset \mathcal{B}_T$.
\end{proof}
\begin{lemma} \label{lemma: contraction}
	Assume Assumption \ref{asm: mckean-vlasov coefficients} with $\rho \geq 2$. There exists a constant $0<c<1$ and a small time $T = T(\rho,\alpha, \sigma, \beta, \bZ)$, such that, for all $(\mu,\gamma), (\nu,\zeta) \in \mathcal{B}_T$, we have
	\begin{align*}
	d\big( \Gamma(\mu,\gamma), \Gamma(\nu,\zeta) \big) 
	\leq c
	d\big( (\mu,\gamma), (\nu,\zeta) \big) .
	\end{align*}
\end{lemma}
\begin{proof}
	Let $M =K(\la \sigma\ra_{p,[s,t]}^{\rho}+1)(L(\sigma, \beta, \bZ) + L(\theta, \gamma, \bZ)) $ be defined as in Lemma \ref{pro: contractivity ito formula}. We have seen in the proof of Lemma \ref{lemma: invariance} that, for $(\mu,\gamma) \in \mathcal{B}_T$, we have $L(\sigma(\mu), \beta(\mu), \bZ) \leq \bar L(\sigma, \beta, \bZ)$. Moreover, from \eqref{eq: sigma is hoelder}, we have $\la \sigma(\mu)\ra_{p}^{p} \leq C_\sigma T \left(
	\int_{C^{\alpha}_0} [\omega]_{\alpha}^{\rho} d\mu(\omega)
	\right)^{\frac{1}{\rho}} \leq \frac{T}{3} $, for $\mu \in \mathcal{B}_T$. Hence $M\leq K\bar L$, for some universal constant $K=K(\alpha, \rho)$.
	We estimate the Wasserstein distance of the image laws, as given in \eqref{Mappings}. From Lemma \ref{pro: contractivity ito formula}, there exists $\bar \rho \geq \rho$ and $C(T) > 0$, such that $\lim_{T\to 0}C(T) = 0$ and
	\begin{align}
	W_{\rho}(\mathcal{L}(x),\mathcal{L}(y))
	\leq \| [x - y]_{\alpha} \|_{L_{\omega}^{\rho}}
	\leq & C(T) e^{M^{\bar\rho} }
	\big(
	\|\sigma(\mu) - \sigma(\nu)\|_{L_t^\infty \mathcal{L}(\R^d;H^k)} 
	 \\
	& + \|(\beta(\mu), \beta(\mu)') - (\beta(\nu), \beta(\nu)')\|_{Z, \alpha;H^k}  
	\big)\nonumber\\
	\leq & C(T) e^{(K\bar L)^{\bar\rho} }
	d\big( (\mu,\gamma), (\nu,\zeta) \big). \label{eq: lem contraction wasserstein}
	\end{align}
	We study now the Gubinelli derivative. For all $s,t \in [0,T]$, we have
	\begin{align*}
	\|\beta(\mu)_{st} - \beta(\nu)_{st}\|_{H^k}
	\leq &
	\|(\mu( \nabla \beta \gamma  )_s - \nu( \nabla \beta \zeta)_s) Z_{st}\|_{H^k}
	+
	\|\mu(\beta)_{st}^{\sharp} - \zeta(\beta)_{st}^{\sharp}\|_{H^K}\\
	\leq & \| (\mu(\beta), \mu( \nabla \beta  \gamma  )) - (\nu(\beta), \nu( \nabla \beta \zeta ))\|_{Z,\alpha,H^k}
	([\bZ]_{\alpha} + |t-s|^{2\alpha}).
	\end{align*}
	Hence, using $\bar \alpha > \alpha$ and $\bar L \geq L$,
	\begin{equation}
	\label{eq: lem contraction gubinelli}
	[\beta(\mu)- \beta(\nu)]_{\alpha;C_b^3}\leq [\beta(\mu)- \beta(\nu)]_{\alpha;H^k}
	\leq  \|\beta\|_{C_b^3\otimes H^k}T^{\bar\alpha - \alpha} \bar L
	d\big( (\mu,\gamma), (\nu,\zeta) \big) .
	\end{equation}
	For the last term in the definition of the metric $d$, we have, using Proposition \ref{pro: contractivity ito formula} and proceeding as in \eqref{eq: lem contraction wasserstein}
	\begin{align}
	\sup_{\Vert \phi \Vert_{C^3_b \otimes H^k} \leq 1} &\|  (E[ \phi(x)], E[ \nabla_1 \phi(x)\beta(x)]) - (E[\phi(y)], E[\nabla_1 \phi(y) \gamma(y)]) \|_{Z,\alpha,H^k}
	\leq C(T) e^{(K\bar L)^{\bar\rho} }
	d\big( (\mu,\gamma), (\nu,\zeta) \big). \label{eq: lem contraction third term}
	\end{align}
	We now add together \eqref{eq: lem contraction wasserstein}, \eqref{eq: lem contraction gubinelli}, and \eqref{eq: lem contraction third term} to obtain
	\begin{equation*}
	d\big( \Gamma(\mu,\gamma), \Gamma(\nu,\zeta) \big)
	\leq C(T) e^{(K\bar L)^{\bar\rho} }
	d\big( (\mu,\gamma), (\nu,\zeta) \big).
	\end{equation*}
	Choosing $T=T(\rho,\alpha, \sigma, \beta, \bZ)$ small enough, depending on $\bar L$, we conclude the proof.
\end{proof}

\section{Non local rough PDEs} \label{sec:NonLinearFP}
Let $d, m \in \N$ be fixed. Let $\bZ \in \mathscr{C}_g^{\bar\alpha}([0,T], \R^m)$, for $\bar \alpha \in (\frac{1}{3}, \frac{1}{2})$. Moreover let $\alpha \in (\frac{1}{3},\bar\alpha)$ and $p=\frac{1}{\alpha}$. Let $\sigma$ and $\beta$ satisfy Assumption \ref{asm: mckean-vlasov coefficients}.

We turn to the Fokker-Planck equation induced by the rough diffusion, which formally reads
\begin{equation} \label{non-localPDE}
\partial_t \mu =  \frac12 \textrm{Tr} \nabla^2 ([\sigma(\mu) \sigma(\mu)^T \mu])   - \Div(\beta(\mu)\mu) \dot{Z} , \quad \mu_0 \in \mathcal{P}(\R^d) .
\end{equation}

%
We define the notion of a solution in a similar way as in the linear case, Definition \ref{def:measure URD definition}, but where now the unbounded rough driver depends on the solution itself. 
\begin{definition} \label{def:PDESolution}
	We say that a path $\mu : [0,T] \rightarrow  \mathcal{P}_{\rho}(\R^d)$ is a solution of \eqref{non-localPDE} with initial condition $\mu_0\in \mathcal{P}_{\rho}(\R^d)$ provided 
	\begin{enumerate}[label=(\roman*), ref=\ref{def:PDESolution} (\roman*)]
		\item
		\label{def:PDESolution: mu}
		for all $\varphi\in C_b^3\otimes H^k$,
		\begin{equation*}
		(\mu(\varphi), \mu(\nabla \varphi \beta(\mu)) \in \mathscr{D}_{Z}^{2\alpha}([0,T]; H^k).
		\end{equation*}
		\item
		\label{def:PDESolution: equation}
		$\mu$ satisfies \eqref{eq:Natural} with the unbounded rough driver $\bB = \bB^{\mu}$ defined from
		$$
		X^{\mu}_{st} = \int_s^t  \beta(\mu_r,\cdot)  d\bZ_r , \qquad \XX_{st}^{\mu} =  \int_s^t  \beta(\mu_r,\cdot)  \int_s^r  \beta(\mu_u,\cdot)  d\bZ_u d \bZ_r
		$$ 
		as in Proposition \ref{prop:RPtoURD}, and $a_t = \frac12 \sigma(\mu_t) \sigma(\mu_t)^T$.
	\end{enumerate}
	
\end{definition}


Existence of a solution to \eqref{non-localPDE} is relatively straightforward. 

\begin{theorem}
	\label{thm: main existence}
	Suppose $\sigma$ and $\beta$ satisfies Assumptions \ref{asm: mckean-vlasov coefficients}, $\mu_0 \in \mathcal{P}_{\rho}(\R^d)$ for $\rho \geq 2$ and $\bZ \in \mathcal{C}_{wg}^{\bar{\alpha}}([0,T];\R^m)$ for $\bar{\alpha} \in (\frac13, \frac12)$. 
	Let $(\Omega, \mathcal{F},(\mathcal{F}_t)_{t \in [0,T]}, P)$ be a complete probability space that supports a $d$-dimensional Brownian motion $W$ and an $\mathcal{F}_0$-measurable random variable, $\Xi \in L^{\rho}(\Omega;\R^d)$ such that the push-forward measure $P_*(\Xi) = \mu_0$.
	Then, there exists a solution $\mu$ of \eqref{non-localPDE}, in the sense of Definition \ref{def:PDESolution}. This solution is given by $\mu_t = \mathcal{L}(x_t)$, where $x$ is the unique solution to the McKean-Vlasov equation \eqref{eq: mcvlasov srde} with initial condition $\Xi$, in the sense of Definition \ref{def: McKean-Vlasov}.
\end{theorem}


\begin{proof}
The proof is completed by following the same steps as in Proposition \ref{pro: measure existence linear pde} except the unbounded rough driver depends on the solution itself. 
\end{proof}


The following result will be crucial for proving uniqueness of the non-local Fokker-Planck equation.
\begin{proposition} \label{prop:GeometricRP}
	Let $\bar \alpha \in (\frac{1}{3}, \frac{1}{2})$, $\alpha \in (\frac{1}{3}, \bar\alpha) $ and $\bZ \in \mathscr{C}_{wg}^{\bar\alpha}([0,T], \R^m)$ is weakly geometric. Define for $(\mu, \gamma) \in \mathcal{M}_Z^{2 \alpha, \rho}$ and $\phi \in C_b^3 \otimes H^k$,
	\begin{equation}
	\label{def: rough path phi}
	X^{\phi}_{st} = \int_s^t  \phi(\mu_r,\cdot)  d\bZ_r , \qquad \XX_{st}^{\phi} =  \int_s^t  \phi(\mu_r,\cdot)  \int_s^r  \phi(\mu_u,\cdot)  d\bZ_u d \bZ_r.
		\end{equation}
Then $\bX^{\phi} \in \mathscr{C}_g^{\alpha}([0,T]; H^k)$. 
\end{proposition}

\begin{proof}
We prove this result in two steps. First we show that the controlled path $(\mu(\phi),\mu(\nabla_1 \phi \gamma))$ can be continuously approximated by controlled paths which takes values in a finite-dimensional space. This clearly gives that $\bX^{\phi}$ can be approximated by a sequence of finite dimensional rough paths. In the second step we use that the finite dimensional rough path is weakly geometric to find a smooth approximation of $\bX^{\phi}$.

\textbf{Step 1}.
	For simplicity we only show this for $\phi \in C^3_b(\R^d) \otimes L^2(\R^d)$, the general case follows by replacing $\phi$ by $D^{\beta}_2 \phi$ for $|\beta| \leq k$. Let $\{e_n \}$ be an orthonormal basis of $L^2(\R^d)$ and define
	$$
	\phi^N(x,y) := \sum_{n=1}^N \langle \phi(x, \cdot),e_n \rangle e_n(y).
	$$
	We now show that $(\phi^N(\mu), \nabla \phi^N(\mu) \gamma) \rightarrow (\phi(\mu), \nabla \phi(\mu) \gamma) $ in $\mathscr{D}_Z^{2 \alpha'}([0,T]; L^2)$ for any $ \alpha' \in (\alpha, \bar{\alpha})$.
	
	Start with the first component. 
	\begin{align*}
	\| \delta \phi^N(\mu)_{st} - \delta \phi(\mu)_{st} \|_{L^2}^2 & = \sum_{n > N} |\langle \delta \phi(\mu)_{st} ,e_n \rangle |^2 = \sum_{n > N} \left| \int_{\R^d} \int_{C^{\alpha}_0} \delta \phi(\omega,y)_{st} e_n(y) d\mu(\omega) dy  \right|^2 \\
	& = \sum_{n > N} \left| \int_{\R^d} \int_{C^{\alpha}_0} \int_0^1 \nabla_1 \phi(\omega_s + \theta \delta \omega_{st},y) \omega_{st} e_n(y)  d \theta d\mu(\omega)  dy \right|^2 \\
	&  =  \sum_{n > N} \left| \int_{C^{\alpha}_0} \int_0^1 \langle \nabla_1 \phi(\omega_s + \theta \delta \omega_{st}), e_n \rangle  \omega_{st} d \theta d\mu(\omega) \right|^2 \\
	& \leq   \int_{C^{\alpha}_0} \int_0^1  \sum_{n > N} |\langle \nabla_1 \phi(\omega_s + \theta \delta \omega_{st}) , e_n \rangle |^2 [\omega]_{\alpha}^2   d \theta d\mu(\omega) |t-s|^{2 \alpha} .
	\end{align*}
	Now for fixed $\omega$, $\theta$ and every $s,t \in [0,T]$ we have the monotone convergence   
	$$
	\sum_{n > N} |\langle \nabla_1 \phi(\omega_s + \theta \delta \omega_{st}), e_n \rangle |^2  \rightarrow 0
	$$
	as $N \rightarrow \infty$ since $\phi \in C_b^3(\R^d) \otimes L^2(\R^d)$. Moreover, for fixed $N$, as a function of $s$ and $t$ the above is continuous. By Dini's theorem we get
	$$
	\sup_{s,t} \sum_{n > N} |\langle \nabla_1 \phi(\omega_s + \theta \delta \omega_{st}), e_n \rangle |^2  \rightarrow 0
	$$
	as $N \rightarrow \infty$. This gives
	$$
	[\delta \phi^N(\mu) - \phi(\mu)]_{\alpha,L^2}^2 \leq  \sup_{s,t} \int_{C^{\alpha}_0} \int_0^1  \sum_{n > N} |\langle \nabla_1 \phi(\omega_s + \theta \delta \omega_{st}), e_n \rangle |^2 [\omega]_{\alpha}^2 d \theta  d\mu(\omega)   \rightarrow 0
	$$
	by monotone convergence. In a similar way one can show that $\nabla_1 \phi^N(\mu \gamma)$ converges to $\nabla_1 \phi(\mu \gamma)$ in $C^{\alpha}([0,T];L^2(\R^d))$.
	
	To see the convergence of the remainder, $\phi^N(\mu)_{st}^{\sharp} := \delta \phi^N(\mu)_{st} - \nabla_1 \phi^N(\mu_s \gamma_s) Z_{st}$, we note first that this term is obviously bounded in $C_2^{2 \alpha}([0,T];L^2(\R^d))$. Furthermore, writing
	\begin{align*}
	\|\phi^N(\mu)_{st}^{\sharp} - \phi(\mu)_{st}^{\sharp}\|_{L^2}^2 & 
	= \sum_{n > N} \left|  \int_{C^{\alpha}_0} \delta \langle \phi(\omega),e_n \rangle_{st} - \langle \nabla_1 \phi(\omega_s),e_n \rangle \gamma_s(\omega_s) Z_{st} d \mu(\omega) \right|^2 \\
	& \leq  \int_{C^{\alpha}_0} \sum_{n > N}   \left| \delta \langle \phi(\omega),e_n \rangle_{st} - \langle  \nabla_1 \phi(\omega_s),e_n \rangle \gamma_s(\omega_s) Z_{st} \right|^2 d \mu(\omega) .
	\end{align*} 
	Using Dini's theorem and monotone convergence as before we get that for any $\epsilon > 0$ there exists $N_{\epsilon}$ such that for all $N \geq N_{\epsilon}$ we have $\sup_{s,t} \|\phi^N(\mu)_{st}^{\sharp} - \phi(\mu)_{st}^{\sharp}\|_{L^2} < \epsilon$. 
	
	This gives, uniformly in $s,t$
	$$
	\|\phi^N(\mu)_{st}^{\sharp} - \phi(\mu)_{st}^{\sharp}\|_{L^2} \leq \epsilon \wedge C|t-s|^{2\alpha} \leq \epsilon^{1- \kappa}C^{ \kappa} |t-s|^{\kappa 2 \alpha} 
	$$
	where we have used the geometric interpolation $a \wedge b \leq a^{1 - \kappa} b^{\kappa}$ for any $\kappa \in (0,1)$. By choosing $\kappa$ correctly we get $\phi^N(\mu)^{\sharp} \rightarrow \phi(\mu)^{\sharp}$ in $C_2^{2 \alpha'}([0,T];L^2(\R^d))$.

\textbf{Step 2}.	
	We now proceed to prove that $\bX^{\phi}$ can be approximated by a smooth path. Let $\epsilon >0$. From the above continuity we can choose $N$ such that
	$$
	[\bX^{\phi^N} -  \bX^{\phi}]_{\alpha'} < \frac{\epsilon}{2},
	$$
	where $\bX^{\phi^N}$ is constructed by replacing $\phi$ with $\phi^N$ in \eqref{def: rough path phi}.
	
	As spelled out in Lemma \ref{lem: approx finite dim rough paths}, there exists $\alpha < \alpha'$ and a smooth path $X^{N, \epsilon}$ such that $[\bX^{\phi^N} -  \bX^{N, \epsilon}]_{\alpha}  < \frac{\epsilon}{2}$. 	
	This gives
	$$
	[\bX^{N, \epsilon} - \bX^{\phi}]_{\alpha}  \leq [ \bX^{N, \epsilon} - \bX^{\phi^N}]_{\alpha} + [ \bX^{\phi^N} - \bX^{\phi}]_{\alpha'}  < \epsilon.
	$$
\end{proof}

%
%

\begin{theorem} \label{thm:Uniqueness}
	Suppose $\sigma$, $\beta$ satisfies Assumptions \ref{asm: mckean-vlasov coefficients} for $k > 9 + d$ and $\mu_0 \in \mathcal{P}_{\rho}(\R^d)$ is given with $\rho \geq 2$. Then there exists at most one solution $\mu$ of \eqref{non-localPDE} in the sense of Definition \ref{def:PDESolution}.
\end{theorem}

\begin{proof}
Let $\mu$ be a solution of \eqref{non-localPDE}. From the the assumptions on $\beta$ and $\sigma$ we may construct the time-dependent coefficients $(\sigma(\mu), ( \beta(\mu), \nabla_1 \beta( \beta(\mu) \mu)))$ from which we construct the rough driver $\bF^{\mu}$ as in Lemma \ref{lem: rough drivers mixed}. Denote by $x^{\mu}$ the solution of 
$$
dx^{\mu}_t   = \sigma(\mu_t, x_t^{\mu}) dW_t + \beta(\mu_t, x_t^{\mu}) d \bZ_t ,
$$
i.e. $dx^{\mu}_t = \bF_{dt}^{\mu}(x_t)$. From Proposition \ref{prop:ItoFormula} we see that $\nu$ satisfies 
\begin{equation} \label{LinearPDE}
\partial_t \nu  = \frac12 \textrm{Tr} \nabla^2 ([\sigma(\mu) \sigma(\mu)^T \nu])  - \Div( \beta(\mu) \nu) \dot{\bZ} .
\end{equation}
as in Definition \ref{def:PDESolution}, where $X_{st}(x) =  \int_s^t \beta(\mu_r,x) d\bZ_r$ and $\XX_{st}(x,y) =  \int_s^t \beta(\mu_r,x) \int_s^r \beta(\mu_u,y) d\bZ_u d\bZ_r$. From the assumption on $\beta$, the Sobolev embedding \cite[Corollary 9.13]{brezis2011}  $H^k \subset C_b^{n+3}(\R^d;\R^d)$ for $k > \frac{d}{2} + n + 3$ and Proposition \ref{prop:GeometricRP} we see that $\bX \in \mathscr{C}_g^{\alpha}([0,T];C_b^{n+3}(\R^d;\R^d))$. Now if $n > 6 + \frac{d}{2}$, we get from Theorem \ref{thm:LinearUniqueness} that there exists at most one solution of \eqref{LinearPDE}. In particular, we see that $\mu_t = \nu_t$ which gives that $x^{\mu}$ is a solution of \eqref{eq: mcvlasov srde}. Since this equation is well-posed, this uniquely describes $\mu$.
\end{proof}

\appendix

\section{Appendix}

\subsection{Kolmogorov continuity theorem}

In this section we prove a Kolmogorov continuity type theorem for rough drivers. The proof is done exactly as in \cite[Theorem 3.1]{FrizHairer}, so we only sketch the proof to convince the reader that the steps are the same. 

\begin{theorem} \label{Kolmogorov}
Suppose $\bF = (F,\FF)$  is a random rough driver such that 
$$
E[\| F_{st} \|_{C_b^3}^q ] \leq C |t-s|^{\beta q}, \quad E[\| \FF_{st} \|_{C_b^2}^{q/2} ] \leq C |t-s|^{\beta q}
$$ 
for $q$ and $\beta$ such that $q \beta > 1$. Then for every $\alpha \in (0, \beta - \frac{1}{q})$ we have 
$$
E[ \|F\|_{\alpha;C_b^3}^q]  \leq C , \qquad E[ \|\FF\|_{\alpha;C_b^2}^{q/2}]  \leq C 
$$
and if $\beta - \frac{1}{q} > \frac13$ then $\bF$ is rough driver for $\alpha \in ( \frac13, \beta - \frac{1}{q})$.
\end{theorem}
\begin{proof}
Take $T=1$ for simplicity and denote by $D_n$ the uniform partition of $[0,1]$ with mesh $2^{-n}$ and let 
$$
K_n := \sup_{t \in D_n} \|F_{t, t + 2^{-n}}\|_{C^3_b} , \quad \KK_n := \sup_{t \in D_n} \|\FF_{t, t + 2^{-n}}\|_{C^2_b}.
$$
By assumption on $\bF$ we get
$$
E[K_n^q] \leq E[ \sum_{t \in D_n}  \|F_{t, t + 2^{-n}}\|_{C^3_b}^q]  \lesssim 2^{-n(1 - \beta q)}, \quad  E[\KK_n^{q/2}] \leq E[ \sum_{t \in D_n}  \|\FF_{t, t + 2^{-n}}\|_{C^3_b}^{q/2}]  \lesssim 2^{-n(1 - \beta q)} . 
$$
Let $s,t \in \bigcup D_n$ and choose $m$ such that $|D_{m+1}| < |t-s| \leq |D_m|$. There exists a partition $\{ t_i \}_{i=0}^N$ of $[s,t]$ such that $(t_i, t_{i+1}) \in D_n$ for some $n \geq m+1$, and for each fixed such $n$ there are at most two such intervals from $D_n$. We get
$$
\|F_{st}\|_{C_b^3} \leq \sum_{i=0}^{N-1} \|F_{t_i t_{i+1}}\|_{C_b^3} \leq 2 \sum_{n \geq m+1} K_n
$$
and using $\FF_{st} = \sum_{i=0}^{N-1} \FF_{t_i t_{i+1}} + \nabla F_{t_i t_{i+1}} F_{s t_i}$, which is easily seen from Chen's relation, we get
\begin{align*}
\|\FF_{st}\|_{C_b^2} \leq \sum_{i=0}^{N-1} \|\FF_{t_i t_{i+1}}\|_{C_b^2} + \| F_{t_i t_{i+1}}\|_{C_b^3} \|F_{s t_i}\|_{C_b^3} \leq  2 \sum_{n \geq m+1} \KK_n + \left( 2 \sum_{n \geq m+1} K_n \right)^2 .
\end{align*}
This gives
$$
\frac{\|F_{st}\|_{C_b^3}}{|t-s|^{\alpha}} \leq K_{\alpha}, \qquad \frac{\|\FF_{st}\|_{C_b^2}}{|t-s|^{\alpha}} \leq \KK_{\alpha}
$$
where
$$
K_{\alpha} := 2 \sum_{n \geq 0} \frac{K_n}{|D_n|^{\alpha}}, \qquad \KK_{\alpha} := 2 \sum_{n \geq 0} \frac{\KK_n}{|D_n|^{\alpha}}
$$
which belongs to $L^q(\Omega)$ and $L^{q/2}(\Omega)$ respectively. This proves the claim.
\end{proof}

\subsection{Weakly geometric rough paths} \label{symmetry}
We prove that rough path integration w.r.t. a weakly geometric rough path yields a weakly geometric rough path. 
\begin{lemma} \label{lem:WGtoWG}
Assume $\bZ$ is weakly geometric and $E$ is a separable Hilbert space and $(Y,Y') \in \mathscr{D}_Z^{2 \alpha}([0,T];E)$. Then the rough path $\bX$ defined by
$$
X_{st} := \int_s^t Y_r^k d\bZ_r^k , \quad \XX_{st} := \int_s^t X_r \otimes Y_r^k d \bZ_r^k - X_s \otimes X_{st}
$$
is also weakly geometric. 
\end{lemma}

\begin{proof}
Let $\{ e_i \}$ be an orthonormal basis of $E$ and use the component notation
$$
g^i = \langle g, e_i \rangle_E, \quad g \in E, \qquad h^{i,j} = \langle h, e_i \otimes e_j \rangle_{E \otimes E}, \quad h \in E \otimes E.
$$
The components of the integrals may thus be spelled out
$$
X_{st}^i = \int_s^t Y_r^{i,k} d\bZ_r^k , \qquad \XX_{st}^{i,j} = \int_s^t X_r^i Y_r^{j,k} d\bZ_r^k - X_s^{i} X_{st}^j
$$
where the above are scalar integrals defined by their local expansions
$$
\Xi_{st}^i = Y_{s}^{i,k} Z_{st}^k + Y_s^{i,k,l}\ZZ_{st}^{l,k} , \quad \Xi_{st}^{i,j} = X_s^i Y_{s}^{j,k} Z_{st}^k + (Y_s^{i,l} Y_s^{j,k} + X_s^i Y_s^{i,k,l}) \ZZ_{st}^{l,k} 
$$
respectively. Since $\Xi^{i,j}_{st} - X_s^i \Xi_{st}^j = Y_s^{i,l} Y_s^{j,k} \ZZ_{st}^{l,k}$ and by definition of $X$ we get
\begin{align*}
| \XX_{st}^{i,j} - Y_s^{i,l} Y_s^{j,k} \ZZ_{st}^{l,k}| \lesssim |t-s|^{3 \alpha} , \qquad |X_{st}^i - Y_s^{i,k} Z_{st}^k| \lesssim |t-s|^{2 \alpha}
\end{align*}
which gives
\begin{align*}
| \XX_{st}^{i,j} +\XX_{st}^{j,i} - X_{st}^{i}X_{st}^{j}| \lesssim |Y_s^{i,l} Y_s^{j,k} \ZZ_{st}^{l,k} + Y_s^{j,l} Y_s^{i,k} \ZZ_{st}^{l,k} - Y_s^{j,k} Z_{st}^{k} Y_s^{j,l} Z_{st}^{l}| +  |t-s|^{3 \alpha}.
\end{align*}
Now, since $\bZ$ is weakly geometric we have
\begin{align*}
Y_s^{i,l} Y_s^{j,k} \ZZ_{st}^{l,k} + Y_s^{j,l} Y_s^{i,k} \ZZ_{st}^{l,k} = Y_s^{i,l} Y_s^{j,k} ( \ZZ_{st}^{l,k} + \ZZ_{st}^{k,l}) = Y_s^{i,l} Y_s^{j,k} Z_{st}^{l} Z_{st}^{k} 
\end{align*}
which gives 
\begin{align*}
| \XX_{st}^{i,j} +\XX_{st}^{j,i} - X_{st}^{i}X_{st}^{j}| \lesssim  |t-s|^{3 \alpha}.
\end{align*}
It is straightforward to check that the above left hand side is the increment from $s$ to $t$ of the function $t \mapsto
\XX_{0t}^{i,j} +\XX_{0t}^{j,i} - X_{t}^{i}X_{t}^{j}
$. Since $3 \alpha > 1$ we get that this function is constant and equal to 0. 
\end{proof}

In the next lemma we show how to construct the approximation in Proposition \ref{prop:GeometricRP}.

\begin{lemma}
	\label{lem: approx finite dim rough paths}
	Fix $N,K,d,m>0$ , $\bar\alpha \in (\frac{1}{3}, \frac{1}{2})$ and let $\bZ\in \mathscr{C}_{wg}^{\bar \alpha}([0,T];\R^m)$ be a weakly geometric rough path.
	Moreover, for $i=1,\dots,d$, $n=1,\dots,N$ and $k=1,\dots,K$, let $e_{n} \in L^2(\R^d)$ be an orthonormal basis and $\theta^{i,k,n} \in \mathscr{D}_{Z}^{2\alpha^\prime}([0,T], \R)$, for $\alpha^\prime \in (\frac{1}{3}, \bar \alpha)$ . Let $\phi = \phi^{i,k} = \sum_{n=1}^N \theta^{i,k,n} e_n$ and construct $\bX^\phi$ as in \eqref{def: rough path phi}. Then, for every $\alpha \in (\frac13, \alpha^\prime)$ there exists $\bX^{\epsilon}$ such that
	\begin{equation*}
	\varrho_{\alpha} (\bX^{\phi}, \bX^{\epsilon})  \to 0, 
	\qquad \mbox{for } \epsilon \to 0.
	\end{equation*}
	
\end{lemma}

\begin{proof}
We take $(\bar e_i)_{i=1,\dots,d}$ an orthonormal basis of $\R^d$ and, for $\bar\iota= 1, \dots, dN$, we define $\xi^{\bar\imath} := e_{n} \bar e_i \in L^2(\R^d;\R^d)$, where $\bar \imath, i, n$ satisfy the relation 
\begin{equation}
\label{eq: bar iota}
\bar \iota := d(n-1) +i, 
\quad i=1,\dots,d, \; n=1,\dots,N.
\end{equation}
Let $V^N$ be the finite dimensional vector space defined as 
\begin{equation*}
V^N:=\operatorname{span}
\{
\xi^{\bar\imath} 
\mid
\bar\imath=1,\dots,dN
\}
\subset
L^2(\R^d;\R^d).
\end{equation*}
We note that $\dim(V^N) = dN$. On this space we construct a rough path as follows,
for $\bar\imath,\bar\jmath = 1, \dots, dN$,
\begin{equation*}
X^{\bar\imath}_{st} 
:= 
C^{\bar\imath}_{st} \xi^{\bar\imath}
:= 
\left(\int_{s}^{t}\theta^{i,k,n}_r d\bZ^{k}_r\right) \; \xi^{\bar\imath},
\quad
\XX^{\bar\imath\bar\jmath}_{st} 
:= 
C^{\bar\imath\bar\jmath}_{st} \xi^{\bar\imath}\otimes\xi^{\bar\jmath}
:=
\left(\int_{s}^{t}\theta^{i,k,n}_r \int_{s}^{u}\theta^{j,l,m}_u d\bZ^{l}_ud\bZ^{k}_r\right) \xi^{\bar\imath}\otimes\xi^{\bar\jmath}.
\end{equation*}
Here and in the following we always assume that the triples $(\bar\imath, i, n)$ and $(\bar\jmath, j, m)$ satisfy relation \eqref{eq: bar iota}. Moreover, we always use the convention that we are summing over repeated indices, in this case $k,l = 1,\dots,K$.
It is immediate to see that $\bX^{\phi} = (\sum_{\bar\imath=1}^{dN}X^{\bar\imath}, \sum_{\bar\imath, \bar\jmath=1}^{dN}\XX^{\bar\imath,\bar\jmath})$.

We prove now that $(X, \XX)$ is geometric, i.e. that the following relation holds
\begin{equation*}
2\operatorname{Sym}(\XX)_{st} = X_{st}\otimes X_{st}, 
\quad
\forall s,t \in [0,T].
\end{equation*}
Let us look more in detail what the tensor product on the right hand side is, for $\bar\imath,\bar\jmath = 1, \dots, dN$,
\begin{equation}
\label{eq: tensor bar i bar j}
(X_{st}\otimes X_{st})^{\bar\imath, \bar\jmath}
=
(C^{\bar\imath}_{st}C^{\bar\jmath}_{st} )\xi^{\bar\imath}\otimes \xi^{\bar\jmath}
=
(C^{\bar\imath}_{st}C^{\bar\jmath}_{st} e_n e_m ) \bar e_{i}\otimes \bar e_{j}.
\end{equation}
Each of these terms is a tensor product which is mostly zero. Let us now describe each component of \eqref{eq: tensor bar i bar j}. We start by introducing the indexes
\begin{equation*}
\quad \hat \imath := d(\bar\imath - 1) + f,
\quad f=1,\dots,d, \; \bar\imath=1,\dots,dN.
\end{equation*}
We assume from now that the couple $(\imath, f)$ and $(\jmath, g)$ always assume the previous relation. We obtain
\begin{equation*}
(X_{st}\otimes X_{st})^{\hat\imath, \hat\jmath}
=
((X_{st}\otimes X_{st})^{\bar\imath, \bar\jmath})^{f,g}
=
C^{\bar\imath}_{st}C^{\bar\jmath}_{st} e^n e^m
\delta_{i,j,f,g}.
\end{equation*}
Similarly, we see that
\begin{equation*}
(\operatorname{Sym}(\XX)_{st})^{\hat\imath, \hat\jmath} 
=
\XX_{st}^{\hat\imath, \hat\jmath} + \XX_{st}^{\hat\jmath, \hat\imath}
=
(C^{\bar\imath,\bar\jmath}_{st} e^n e^m)(e_i\otimes e_j)^{f,g}
=
C^{\bar\imath,\bar\jmath}_{st} e^n e^m \delta_{i,j,f,g}.
\end{equation*}
The symmetry condition reduces to verify the scalar equality
\begin{equation*}
C^{\bar\imath}_{st}C^{\bar\jmath}_{st} = C^{\bar\imath,\bar\jmath}_{st},
\end{equation*}
which is satisfied thanks to Lemma \ref{lem:WGtoWG}.

The rough path $\bX^\phi$ is thus in $\mathscr{C}_{wg}^{\alpha^\prime}([0,T],V^N)$. Since $V^N$ is a finite dimensional space, we can find a smooth approximation $\bX^{\epsilon}$ in $\mathscr{C}^{\alpha}([0,T], V^N)$, for some $\alpha \in (\frac{1}{3}, \alpha^\prime)$. Hence, since $V^N \subset L^2(\R^d;\R^d)$, this is also an approximation in $\mathscr{C}^{\alpha}([0,T], L^2(\R^d;\R^d))$.
\end{proof}
\subsection{A separable subspace of the H\"{o}lder space}

\begin{proposition} \label{prop:C0 separable}
The space $C_0^{\alpha}([0,T];E)$ is equal to the closure of $C^{1}([0,T];E)$ with respect to the $C^{\alpha}$-topology. In particular, $C_0^{\alpha}([0,T];E)$ is separable if $E$ is separable.
\end{proposition}

\begin{proof}
For simplicity we assume $E = \R$. We clearly have $[f]_{\alpha,h} \leq h^{1 - \alpha}\|\nabla f\|_{\infty}$ so that $C^1([0,T]) \subset C_0^{\alpha}([0,T])$, which shows one inclusion by taking the closure.

To see the reversed inclusion, we take $f \in C_0^{\alpha}([0,T])$, a standard mollifier $\rho_n(u) = n \rho(n u)$ and let $f^n_t = \int_0^T f_u \, \rho_n(t - u) du = \int_t^{T-t} f_{t-u} \, \rho_n(u) du$. Then $f^n$ is smooth and we get for $|t-s| \leq h$
\begin{align*}
|f^n_t - f^n_s| \leq \int_t^{T-t} |f_{t - u} - f_{s- u}| \rho_n(u) du \leq [f]_{\alpha,h}  |t-s|^{\alpha} 
\end{align*}
so that $[f^n]_{\alpha,h} \leq [f]_{\alpha,h}$. Let us show that $f^n$ converges uniformly to $f$. 
\begin{align*}
|f_t - f_t^n| & = \left| f_t  \int_0^T \rho_n(t-u) du - \int_0^T f_u \, \rho_n(t-u) du \right| \leq [f]_{\alpha} \int_0^T |t-u|^{\alpha}\rho_n(t-u) du \\
 & \leq [f]_{\alpha} \int_{\R} |t-u|^{\alpha}\rho_n(t-u) du = [f]_{\alpha} n^{-\alpha} \int_{\R} |r|^{\alpha}\rho(r) dr
\end{align*}
which converges to 0 uniformly in $t$. 

Now, write
\begin{align*}
[f - f^n]_{\alpha} \leq [f - f^n]_{\alpha,h} + \sup_{(s,t) \in \Delta_T : |t-s| \geq h} \frac{|\delta f_{st} - \delta f_{st}^n|}{|t-s|^{\alpha}} \leq 2 [f]_{\alpha,h} + 2 h^{-\alpha} \|f - f^n\|_{\infty}
\end{align*}
which gives
$$
\lim_{n \rightarrow \infty} [f - f^n]_{\alpha} \leq 2 [f]_{\alpha,h}.
$$
By assumption on $f$, letting $h \rightarrow 0$ gives that $f^n \rightarrow f$ in $C^{\alpha}([0,T])$.
\end{proof}

\bibliographystyle{abbrv}
\bibliography{bibliography} 

\end{document}